\documentclass[11pt]{amsart}

\usepackage[tmargin=2cm,bmargin=2cm,lmargin=2cm,rmargin=2cm]{geometry}
\usepackage{amssymb,amsthm, times}
\usepackage{delarray,verbatim}
\usepackage{mathrsfs}
\usepackage{ifpdf}
\ifpdf
\usepackage[pdftex]{graphicx}
 \else
\usepackage[dvips]{graphicx}
 \fi

\usepackage{bm}

\usepackage{ifpdf}
\usepackage{color}
\definecolor{webgreen}{rgb}{0,.5,0}
\definecolor{webbrown}{rgb}{.8,0,0}
\definecolor{emphcolor}{rgb}{0.5,0.95,0.95}
\newcommand{\rd}{\textcolor[rgb]{1.00,0.00,0.00}}

\usepackage{hyperref}
\hypersetup{%
	colorlinks=true,
	linkcolor=webbrown,
	filecolor=webbrown,
	citecolor=webgreen,
	breaklinks=true}
\ifpdf \hypersetup{pdftex,
	pdfstartview=FitH, 
	bookmarksopen=true,
	bookmarksnumbered=true
} \else \hypersetup{dvips} \fi

\renewcommand{\tilde}{\widetilde}
\renewcommand{\bar}{\overline}

\numberwithin{equation}{section}

\newtheorem{theorem}{Theorem}[section]
\newtheorem{proposition}[theorem]{Proposition}

\newtheorem{remark}[theorem]{Remark}
\newtheorem{lemma}[theorem]{Lemma}

\newtheorem{assump}{Assumption}[section]

\newcommand {\R}{\mathbb{R}}
\newcommand {\F}{\mathcal{F}}

\newcommand {\N}{\mathbb{N}}
\newcommand {\p}{\mathbb{P}}

\newcommand {\G}{\mathcal{G}}
\newcommand {\E}{\mathbb{E}}

\newcommand{\diff}{{\rm d}}

\newcommand{\lev}{L\'{e}vy }

\newcommand{\px}{\mathbb{P}_x}

\newcommand{\ol}{\overline}
\newcommand{\ul}{\underline}

\newcommand{\du}{\,{\rm d}u}

\begin{document}
\title[Non-zero-sum optimal stopping game with continuous versus periodic exercise opportunities]{{Non-zero-sum optimal stopping game with \\ continuous versus periodic exercise opportunities}}

	
\author[J.L. P\'erez]{Jos\'e Luis P\'erez$^*$}\thanks{$^*$Department of Probability and Statistics, Centro de Investigaci\'on en Matem\'aticas, A.C. Calle Jalisco S/N
C.P. 36240, Guanajuato, Mexico.
Email: jluis.garmendia@cimat.mx}

\author[N. Rodosthenous]{Neofytos Rodosthenous$^\dagger$}
\thanks{$^\dagger$Department of Mathematics, University College London, Gower St, London WC1E 6BT, UK. Email: n.rodosthenous@ucl.ac.uk}

\author[K. Yamazaki]{Kazutoshi Yamazaki$^{**}$}
\thanks{$^{**}$School of 
Mathematics and Physics, The University of Queensland, St Lucia,
Brisbane, QLD 4072, Australia. Email: k.yamazaki@uq.edu.au. In part supported by JSPS KAKENHI grants no.\ 19H01791, 20K03758, and 24K06844, Open Partnership Joint Research
Projects grant no.\ JPJSBP120209921, and the start-up grant from the
University of Queensland.  }

\date{}

\begin{abstract} 
We introduce a new non-zero-sum game of optimal stopping with asymmetric exercise opportunities. 
Given a stochastic process modelling the value of an asset, one player observes and can act on the process continuously, while the other player can act on it only periodically at independent Poisson arrival times. 
The first one to stop receives a reward, different for each player, while the other one gets nothing.  
We study how each player balances the maximisation of gains against the maximisation of the likelihood of stopping before the opponent. 
In such a setup, driven by a \lev process with positive jumps, we not only prove the existence, but also explicitly construct a Nash equilibrium with values of the game written in terms of the scale function. Numerical illustrations with put-option payoffs are also provided to study the behaviour of the players' strategies as well as the quantification of the value of available exercise opportunities.
\\ 
\\
\noindent \small{\noindent  {AMS 2020} Subject Classifications: 
60G51, 
60G40, 
91A15, 
90B50. 
}
\end{abstract}

\maketitle

\vspace{-5mm}
\section{Introduction}

We consider a game of timing between two players who are after the same underlying asset, whose value is evolving stochastically according to a spectrally positive \lev process $X$. 
The players can choose when to stop the game based on their (asymmetric) available opportunities, and the first one to stop receives an associated reward, while the other one gets nothing. 
More precisely, player $C$ can observe the process $X$ {\it continuously} and stop it without delay. 
On the contrary, player $P$ can stop the process $X$ only {\it periodically} at  random times, given by the jump times of an independent Poisson process with rate $\lambda > 0$. 
Both players are aware of competition, and player $C$ knows (can estimate) that the opponent's rate of exercise opportunities is $\lambda>0$, but cannot know the actual (random) times of these exercise opportunities (only finds out after player $P$ stops). 
The objective of this paper is to analyse the aforementioned asymmetry in exercise opportunities and answer   
the question {\it ``What is the optimal strategy for each of the two players, that balances their maximisation of gains (resp.,\ minimisation of costs) against the simultaneous likelihood maximisation of acting before their opponent?''}. 
Given that each player tries to optimise their individual reward/cost based on their available exercise opportunities, this competition can be formulated as a new non-zero-sum game (NZSG) of optimal stopping under continuous versus periodic exercise opportunities.

Even though player $C$ faces no constraints in optimising their decision making process, clearly player $P$ must optimise their own decision making under the exogenous constraint of having  only periodic exercise opportunities available in their arsenal. 
Such constraints on a player's ability to act on the underlying asset may represent the liquidity effect, where the random times (of Poisson jumps) model the times at which decision making becomes available to the player. 
In particular, liquidity restrictions involving Poisson exercise opportunities, were used in the optimal exercise of perpetual American options \cite{DW2002} and lookback options \cite{GL05} (optimal stopping problems; see also \cite{HZ22} for an extension including the ability to control the frequency of opportunities at a cost),  
in reversible investment strategies \cite{LW16} (optimal switching problem), 
in the optimal conversion/calling strategies in convertible bonds \cite{LS19} (zero-sum game), 
and in the optimal portfolio rebalancing \cite{RZ02} (optimal control problems; see also \cite{PT08, APW14} for other related portfolio optimisation problems). 
%
In this sense, our game extends the applications in decision making under liquidity constraints to the category of constrained NZSGs of optimal stopping, where one of the players faces no constraint while the other one faces liquidity constraints and can act only periodically, when an opportunity arises. 
\vspace{1mm}
{\bf Our Motivation.}
This framework can have a number of applications in real-world business situations, where operations research has a core role to play (see, e.g.~\cite{SKSS08} and references therein). We motivate our game via the scenario when 
two firms, players $C$ and $P$, aim at acquiring another firm, which could be financially-distressed, or simply a new entity  in the market (e.g.\ start-up).  
In particular, the acquisition of financially distressed firms is of significant importance in the field of corporate finance (see \cite{Hotchkiss, JoryMadura2009, Zhou} and references therein). 
The primary benefit and motivation for firms $C$ and $P$ to acquire the target firm are the opportunities to enlarge and diversify their business by absorbing the firm at a low cost. 
Each firm thus wants to wait until the market value $X$ (or cost of acquisition) of the target firm is sufficiently low, since their rewards are naturally decreasing in $X$, by taking into account that $X$ may also experience positive jumps.  
Such target firms could be venture capital, start-ups, or R\&D firms, with potential upward jumps due to successful investments, technological breakthroughs and innovations, e.g.\ the successful development of new drugs by a pharmaceutical firm and the acquisition of patents for new technologies
(the modelling of values $X$ of such firms by spectrally positive \lev processes is common in the literature, e.g.\ \cite{BE} for venture capital management, \cite{Avanzi, BKY} for optimal dividend problems). 
Such value dynamics require firms to carefully assess the market conditions and make strategic decisions about if and when to acquire the target firm.
However, the main trade-off is that each firm is competing with their opponent, who also aims at acquiring the same target firm, hence they need to make a buy-out decision early enough before the target firm is bought by their opponent. 
Finally, their strategic decision making should also take into consideration that firm $C$ has no constraints while firm $P$ faces potential decision making constraints. 

To give more context, firm $C$ may be an established, large corporation with a diversified business (e.g.\ General Electric (GE), SONY, Samsung, etc., engaging in various areas of the market, such as electricity, insurance, banking, and entertainment), a stable customer base and better access to credit, which afford them financial flexibility. Firm $C$ is thus more consistently engaged in the game with the freedom to make decisions at any point.
On the contrary, firm $P$ may be of smaller size and less diversified, who often faces more significant financial limitations (e.g. due to limited access to credit, lower cash reserves, smaller customer base). Due to these liquidity constraints, firm $P$ has a lower frequency of opportunities to stop first, and becomes part of the game periodically when specific conditions or opportunities align with their constraints. 
This dynamic adds an interesting asymmetry and a strategic element to the decision making in this game, highlighting the strategic challenges and imbalances associated with liquidity constraints.
It is also worth noting that the benefit of acquisition is lower for firm $C$ than for $P$ (see, e.g.\ \cite{Zhou}), due to the lower synergy effect on the business of already-diversified firms, as a consequence of acquisition, compared to less-diversified ones, and the associated costs tend to increase in the firm's size and complexity.

Naturally, this competition between firms $C$ and $P$ involves their strategic decision-making to acquire a target firm at a favourable low value $X$. The challenge is to determine each firm's optimal strategy, considering the trade-offs amongst the target firm's value $X$, potential future positive jumps in $X$, the fact that once the target firm is acquired the game ends, the liquidity constraints of firm $P$, the potential lower benefit of acquisition for firm $C$, and of course the potential actions of each firm's competitor. This scenario therefore gives rise to the constrained NZSG of optimal stopping proposed in this paper.

For the interested reader, we present below some alternative real-life scenarios that also share the characteristics of this game. 

{\it Retail inventory supply.}
Two retailers (players $C$ and $P$) are competing for purchasing a limited supply/inventory of a specific product from a common supplier. 
Retailer $C$ may be an established, larger entity in the market, with a stable customer base, better access to credit and flexibility to make decisions at any point. 
Retailer $P$ may be a smaller or newer entity in the market with financial limitations and higher liquidity constraints, who consequently has less opportunities to stop first.  
Both retailers observe the dynamic pricing offered by the supplier for the inventory, whose variations exhibit both a traditional Brownian uncertainty, modelling the continuous or gradual price adjustments based on market conditions (e.g.\ raw material cost fluctuations) or the scarcity premium (price increases as the product becomes scarcer), as well as an additional element of uncertainty (unpredictability) coming from positive price jumps over time. The latter jumps are linked to sudden changes in market conditions (e.g.\ increased demand, unexpected shortages, changes in the cost of production, political instability, sanctions against countries, transportation difficulties). 
A prominent example is that of rare earth materials and rare metals, which tend to experience positive jumps in pricing due to geopolitical instability and market domination \cite{FT_rare_metal2,FT_rare_metal1}.
In addition, retailer $C$ may incur a cost penalty when stopping first, representing the trade-off for their flexibility to make decisions more freely, missed bulk purchase discounts, or a fee for leveraging this advantageous position. This introduces an additional layer of complexity and strategic decision-making, which aligns with the idea that advantages in strategic positioning often come with certain costs. 

{\it Technology procurement in IT industry}.
Player $C$ is a well-established technology company with a solid financial foundation (larger market presence, historical stability, no liquidity constraints) and player $P$ is an innovative start-up in the IT industry, operating with higher liquidity constraints due to its smaller size (priority is stability in the dynamic technology sector). Both companies wish to buy hardware from a technology supplier, who offers dynamic pricing (e.g.\ reflecting the rapid and sudden evolution of technology and market dynamics) and companies must decide when to make procurement decisions. The strategic position of the more established technology company (player $C$) could be acknowledged by a potential cost penalty (e.g.\ tied to missing out on future dynamic pricing advantages, early adoption incentives, exclusive deals), or the cost for remaining actively involved throughout the game. A couple of examples of such hardware are graphic cards, whose sudden increase in prices was related to Bitcoin miners buying large amounts \cite{Wilson}, international politics \cite{FT_chips2,FT_chips1}, and the recent shortage of chip-making equipment (e.g.\ lithography machines for manufacturing semiconductors) contributing to positive jumps in semiconductor prices \cite{Jeon23}.

{\it Real estate development bidding}. 
Player $C$ is an experienced and well-established real estate developer with a proven track record and better access to financial resources (operates without liquidity constraints) and player $P$ is a small real estate development firm entering the market, facing higher liquidity constraints due to limited resources and a smaller portfolio. Both real estate developers wish to bid for land from a landowner (supplier), who offers dynamic pricing of the land (e.g.\ based on planned or sudden urban development, changes in zoning laws, unexpected demand) and developers must decide when to submit bids to acquire the land and develop their projects. 
In general, real estate prices tend to experience positive jumps due to sudden changes in economics and migration \cite{Kiss,Pontines}.
The trade-off for the well-established real estate developer (player $C$) could be recognised by a potential cost penalty (e.g.\ related to missing favourable pricing terms, early access to prime locations, exclusive opportunities), or the cost for remaining actively involved throughout the game (see, e.g.\ \cite{Grenadier} for another two-player game in real estate development from the `time-to-build’ viewpoint).

\vspace{1mm}
{\bf Our mathematical contributions.}
The main contributions of the paper are the following: 
\begin{enumerate}
\item 
 we prove the {\it existence} of a Nash equilibrium for this class of NZSGs of optimal stopping with 
{continuous versus periodic exercise opportunities} 
and concave and decreasing reward functions. 

\vspace{1mm}
\item 
we develop a methodology that not only achieves the crucial theoretical existence of a Nash equilibrium, but also provides its {\it explicit construction}, which is of fundamental importance to applications;  

\vspace{1mm}
\item 
this explicit nature of our results allows us to also 
$(i)$ find a {\it unique} Nash equilibrium that is Pareto-superior to any other in case there is more than one, 
and present $(ii)$ a straightforward numerical in-depth analysis of a case study with put-option-type payoffs,
$(iii)$ {\it comparative statics} with respect to the rate $\lambda$ of player $P$'s exercise opportunities, 
as well as 
$(iv)$ a numerical study of the {\it value of available exercise opportunities}. 
\end{enumerate}
To the best of our knowledge, this is the first paper studying a NZSG of optimal stopping with continuous versus periodic exercise opportunities. 

The literature on NZSGs of optimal stopping is mainly concerned with the existence of a Nash equilibrium under continuous observations. 
They have been studied via different methodologies, including 
quasi-variational inequalities (QVIs) \cite{BF}, 
QVIs in Dirichlet forms \cite{Nagai},
stochastic processes theory \cite{Etourneau86},
backward induction \cite{Ohtsubo87},
potential theory of Ray-Markov processes \cite{CattiauxLepeltier90}, 
the martingale approach \cite{HuangLi90},  
Snell envelope theory \cite{HZ10} and special utility-based arguments for certain applications \cite{Kuhn04} (see also subgame-perfect equilibrium methods \cite{RS14} and $\varepsilon$--equilibria \cite{Morimoto87, LarakiSolan13}). 
Beyond existence results, examples where such equilibria have been constructed are very limited and are always under a setting of
{continuous observations of continuous processes}, such as convertible bonds pricing \cite{CDW10} and derivation of sufficient conditions for threshold strategy optimality \cite{Att17, DeAFM}. 

A random periodicity in 
{exercise opportunities} has so far been considered in one-player optimal stopping settings {and zero-sum-games (ZSGs)} with various applications in the literature. 
The American put/call option with exercise times restricted to be Poisson arrival times has been studied by \cite{DW2002} in a Brownian motion model, \cite{Lempa12} in a linear diffusion model, \cite{PY_American} in a \lev model and {\cite{PPY2021} with negative discount rate}, 
while an endogenous bankruptcy (Leland-Toft model) in the Poisson observation setting was studied by \cite{PPSY2020}. 
{The shape of the value function under optimal stopping of diffusions over Poisson jump times was then studied theoretically in \cite{Hob21}, while the inclusion of stochastic control over the rate of Poisson jumps was further considered in \cite{HZ22}.
Finally, a zero-sum game (ZSG) under Poisson exercise opportunities was recently studied in \cite{LS19}.}
On the other hand, to the best of our knowledge, NZSGs of optimal stopping, featuring periodic 
{exercise opportunities} have not been studied before.

In this paper, we expand the literature on NZSGs of optimal stopping both by proving the existence of a Nash equilibrium and by explicitly constructing the optimal threshold strategies, in a novel setting where $(i)$ the underlying dynamics is given by a L\'evy model with positive jumps and $(ii)$ players have asymmetric 
{exercise opportunities (continuous versus periodic)}.  
Either of these features sets our framework outside the well-developed standard theory (see, e.g. references above).
Drawing on the theory of optimal stopping for diffusion processes is also non-feasible as it has limitations when dealing with jumps.
It is also worth mentioning that, in zero-sum games of optimal stopping, the characterisation of threshold equilibria is given by the optimisation of a single expected reward, while in our NZSG we deal with the joint optimisation of two coupled expected rewards, whose complexity amplifies due to the 
{asymmetry in exercise opportunities}.

In order to tackle our game and the aforementioned complexity, we develop in this paper a more ``direct'' approach for proving the existence and eventually for explicitly constructing a Nash equilibrium. 
We firstly obtain optimality in the restricted class of ``threshold-type'' strategies, by combining our probabilistic approach based on the cutting-edge fluctuation theory of L\'evy processes with a traditional best response approach. Then, building on these results, we upgrade the optimality over all possible stopping times. 
To achieve this, we propose an amalgamated methodology which involves: 
(i) a verification approach for half of the coupled system (associated to player $P$), extending results in \cite{DW2002} to L\'evy models and random time horizon; 
(ii) a reformulation of the other half of the coupled system (associated to player $C$, for which the verification approach is non-feasible) to one involving stochastic path-dependent killing;
(iii) an introduction of an average problem approach (developed in \cite{LZ}, \cite{RZ1} and \cite{Surya} for single-player problems) to  solve our NZSG with asymmetric exercise opportunities.

To be more precise, we firstly use a {\it probabilistic approach} to obtain the values associated with ``threshold-type strategies'' in terms of scale functions. 
We then show the existence and explicitly construct a Nash equilibrium in the class of threshold strategies via first-order conditions, for a wide class of underlying \lev processes with positive jumps. 
This prescribes that player $C$ stops (i.e. buys the asset) at the first time the asset price $X$ decreases to (or is below) a threshold $a^*$, while player $P$ stops at the first 
{exercise opportunity} when $X$ is below a threshold $l^*$ such that $l^* > a^*$, illustrated in Figure \ref{plot_illustration}.
The latter is similar to the threshold strategy in a discrete-time model, where $X$ may cross below $l^*$ and before an exercise opportunity appears, may recover above it resulting in player $P$ not stopping. 
What makes our problem interesting is that while $X$ is below $l^*$, it may even go below $a^*$ before an exercise opportunity appears for player $P$, resulting in player $C$ exercising instead. 
The power of our direct approach is also reflected by the fact that it does not rely on (or a priori assume) neither the continuous nor the smooth-fit condition. 
Nevertheless, we prove a posteriori that the value function of player $C$ is $C^1$ at $a^*$ in both unbounded and bounded variation L\'evy models and that of player $P$ is $C^1$ at $l^*$ (resp.\ $C^2$) when it is of bounded (resp.\ unbounded) variation 
(see Proposition \ref{smooth}). 

\begin{figure}[htbp]
\vspace{-3mm}
\begin{center}
\begin{minipage}{1.0\textwidth}
\centering
\begin{tabular}{cc}
 \includegraphics[scale=0.55]{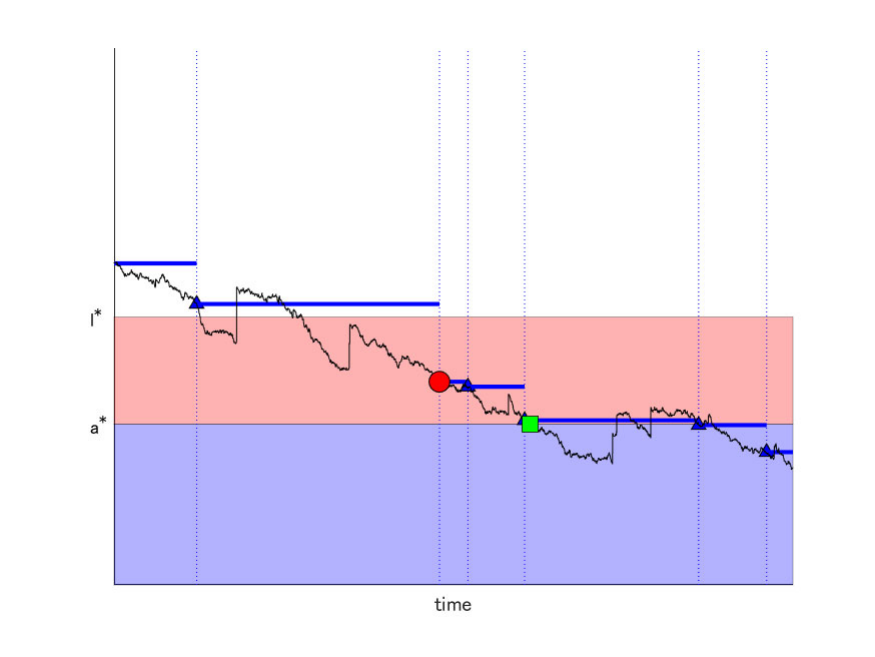} &  \includegraphics[scale=0.55]{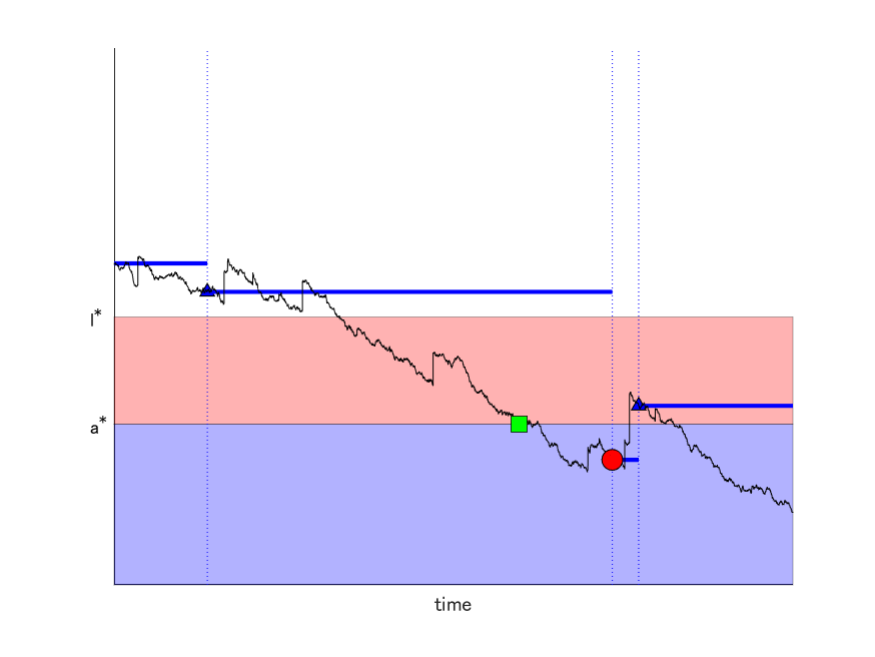} \\
 \textbf{Case 1} & \textbf{Case 2}
\end{tabular}
\vspace{-4mm}
\caption{Illustration of player $C$ and player $P$'s stopping strategies. 
The solid black trajectory shows the path of $X$ and the piecewise horizontal blue lines show player $P$'s most recent 
{exercise opportunity, whose} observation times are shown by dotted vertical lines. 
Given some $l^* > a ^*$, player $P$ stops at the first observation time of $X$ below $l^*$ (indicated by red circles) and player $C$ stops at the classical hitting time below $a^*$ (indicated by green squares). 
\textbf{Case 1} shows the scenario when player $P$ 
{has an exercise opportunity} when $X$ is in $(a^*, l^*]$ (shown by the red strip) and thus exercises first. 
\textbf{Case 2} shows the scenario when player $P$ does not get 
{an exercise opportunity before $X$} enters $(-\infty, a^*]$ (shown by the blue strip) and thus player $C$ exercises first.} 
\label{plot_illustration}
\end{minipage}
\end{center}
\end{figure}

We further prove that these threshold strategies form a Nash equilibrium in the strongest sense possible -- optimal amongst all possible stopping strategies.
A standard verification approach has limitations when used to upgrade the ``optimality'' of player $C$'s threshold strategy under the generality of our problem formulation in L\'evy models and reward functions (additional assumptions seem to be required on both the model and rewards, that are also hard to verify). 
Hence, we take a different route for player $C$, based on our already obtained results and a transformation of the problem to one with stochastic path-dependent discounting and analyse the latter via an average problem approach. 
Then, we prove that the value of player $P$'s threshold strategy is sufficiently regular (see Proposition \ref{smooth}) to apply It\^o's formula, and then we develop a verification result (Lemma \ref{verification_split}) with variational inequalities to upgrade their ``optimality''. 
Combining these two different routes for players $C$ and $P$ leads to the NZSG's solution.
In all, our proposed amalgamated methodology allows for the construction of a Nash equilibrium and verification of its optimality over all possible stopping strategies for each player, without imposing further assumptions than the initial mild ones.  
We further point out that our ideas and techniques can be replicated to deal with a much wider range of games with such an asymmetry of exercise opportunities. 
In addition, by flipping $X$ and modifying reward functions accordingly, the considered class of problems also encompasses the case driven by spectrally negative \lev processes.
Finally, the explicit nature of our results allows for accessible numerical studies and sensitivity analyses.

\vspace{1mm}
{\bf Structure of the paper.}
The rest of the paper is organised as follows. In Section \ref{problem}, we provide a mathematical formulation of the game between the two players with continuous and periodic 
{exercise opportunities}, which can be understood for a general class of stochastic processes $X$ (not only L\'evy models).
In Section \ref{sec:SP}, we review the fluctuation theory of \lev processes and the scale function. In particular, in Section \ref{fluct_new}, we also develop new identities required for expressing the expected rewards under threshold strategies, which can also be applied to study other optimal stopping problems and games under such 
{an asymmetry of exercise opportunities}; 
these could be natural directions for future research.  
In Section \ref{section_threshold_strategies}, we obtain a Nash equilibrium in a weak sense, by restricting the set of strategies to those of threshold-type. 
In Section \ref{section_optimality_general}, 
we strengthen the optimality and show that the one obtained in Section \ref{section_threshold_strategies} is actually 
a Nash equilibrium in the strong sense (with strategy sets given by general sets of stopping times). We conclude the paper with numerical results on strategy optimality, comparative statics and the value of 
{available exercise opportunities} in Section \ref{section_numerics}. Several long or technical proofs are deferred to the Appendices \ref{fluct_proof}--\ref{verification_split_proof}.

\section{Problem formulation} \label{problem}


Let $X$ be a real-valued Markov process and $N$ be an independent Poisson process with rate $\lambda > 0$, defined on a probability space $(\Omega,\mathcal{H},\p)$.
Then, denote by $(T^{(n)})_{n \in \N}$ the sequence of jump times of $N$ with inter-arrival times given by independent exponential random variables with parameter $\lambda$ (we use the convention $\N := \{1,2,\ldots \}$). 
We denote by $\p_x$ the law of $X$ started at some fixed $x \in \R$ and by $\mathbb{E}_x$ the corresponding expectation.

We consider the following two players.  
The first one is player $C$, who can observe and stop $X$ without delay, thanks to their continuous access to $X$.
The filtration $\mathbb{F}= (\mathcal{F}_t)_{t\geq 0}$ of player $C$'s exercise opportunities is therefore the natural filtration of $X$ given by $\F_t := \sigma(X_s\,|\, 0 \leq s \leq t)$, so that player $C$ can choose any stopping time with respect to $\mathbb{F}$.
The second one is player $P$, who can stop $X$ only periodically, at the Poisson observation times $(T^{(n)})_{n\in\N}$. 
We define ${\mathbb{G}} := ({\mathcal{G}}_n)_{n \geq 1}$ for $\G_n := \sigma(T^{(k)}, X_{T^{(k)}} \,|\, 1 \leq k \leq n)$, so that player P can stop at $T^{(M)}$ for any $\mathbb G$--stopping time $M$ (see Remark \ref{remark_filtration} for an equivalent formulation).
We further note that, player $C$ knows of the existence of competition with player $P$, knows the opponent's rate of 
{exercise opportunities} is $\lambda>0$, but cannot know the actual (random) times $(T^{(n)})_{n\in\N}$ of 
these exercise opportunities (they are not part of $\mathbb{F}$).
On the other hand, player $P$ also knows of the competition arising from the existence of player $C$.

\begin{remark} \label{remark_filtration}
Let $\tilde{\mathbb{G}} := (\tilde{\mathcal{G}}_n)_{n \geq 1}$ for $\tilde{\mathcal{G}}_n := \mathcal{H}_{T^{(n)}}$, $n \geq 1$, where  $(\mathcal{H}_t)_{t \geq 0}$ defines the filtration generated by the two-dimensional process $(X,N)$. 
Notice that $\mathcal{G}_n \subset \tilde{\mathcal{G}}_n$ for all $n \geq 1$. In particular, ${\mathbb{G}}$ represents a restricted information set on $X$ which is updated only when an exercise opportunity arises, while $\tilde{\mathbb{G}}$ represents full information on $X$ (as for player $C$) despite player $P$'s restricted periodic exercise opportunities. 
The considered problem is equivalent to the version when $\mathbb{G}$ is replaced with $\tilde{\mathbb{G}}$. 
The mathematical analysis in this paper applies to both information scenarios for player $P$, hence the choice of a particular framework relies solely on the application under consideration. 
\end{remark}

The aforementioned players compete against each other in the following non-zero-sum game (NZSG) of optimal stopping 
of the same asset, where the first one to stop receives an associated reward. 
To be more precise, 
player $C$ (resp.,\ $P$) aims at maximising a discounted reward function $f_c:\mathbb{R} \to \mathbb{R}$ (resp.,\ $f_p:\mathbb{R} \to \mathbb{R}$) 
by stopping the game before player $P$ (resp.,\ $C$), 
otherwise receives nothing. 
Both players discount their future gains with a constant discount rate $q>0$. 
Even though they are both after the same asset, the additional 
{exercise opportunities} provided to player $C$ (as opposed to player $P$) yield an additional fee for player $C$, if and when successfully stopping before~player~$P$.  

A pair of stopping times $(\tau, \sigma)$ in this game consists of  $\tau \in \mathcal{T}_c$ and $\sigma \in \mathcal{T}_p$,
where $\mathcal{T}_c$ is the set of $\mathbb{F}$-stopping times and  $\mathcal{T}_p := \{ T^{(M)}: \textrm{$M$ is a $\mathbb G$--stopping time}\}$.
This means that while player $C$ can stop in the ``usual'' way, player $P$ can stop only at the Poisson observation times.
%
Each player aims at maximising their expected discounted reward functions given by 
\begin{align*} 
V_c(\tau, \sigma; x) := \mathbb{E}_x \Big[ e^{-q \tau} f_c(X_{\tau}) 1_{\{ \tau <\sigma\}} \Big] 
\quad \text{and} \quad 
V_p(\tau, \sigma; x) := \mathbb{E}_x \Big[ e^{-q \sigma} f_p(X_{\sigma}) 1_{\{ \sigma < \tau\}} \Big]. 
\end{align*}
The main aim of this paper is therefore 
to obtain,  for each $x \in \R$,  a Nash equilibrium $(\tau^*,\sigma^*) \in \mathcal{T}_c \times \mathcal{T}_p$ such that
\begin{align} \label{Nasheq}
\begin{split}
&V_c(\tau^*,\sigma^*; x) \geq V_c(\tau,\sigma^*; x),  \quad \forall\; \tau  \in \mathcal{T}_c, \\
&V_p(\tau^*,\sigma^*; x) \geq V_p( \tau^*,\sigma; x), \quad \forall\; \sigma \in \mathcal{T}_p.
\end{split} 
\end{align}

\section{Fluctuation identities}
\label{sec:SP}


Throughout this paper, we focus on the case when $X=\{X_t:t\geq 0\}$ is
a spectrally positive L\'evy process   whose \emph{Laplace exponent} is given by
\begin{align}
\psi(s)  := \log \E \left[ e^{-s X_1} \right] =  \gamma s +\frac{1}{2}\nu^2 s^2 + \int_{(0,\infty)} (e^{-s z}-1 + s z 1_{\{0 < z < 1\}}) \Pi (\diff z), \quad s \geq 0, \label{laplace_spectrally_positive}
\end{align}
where $\Pi$ is a \lev measure on $(0,\infty)$ that satisfies the integrability condition $\int_{(0,\infty)} (1 \wedge z^2) \Pi(\diff z) < \infty$. 
{Note that, the process $X$} has paths of bounded variation if and only if $\nu = 0$ and $\int_{( 0,1)}z\, \Pi(\diff z) < \infty$; in this case, we write \eqref{laplace_spectrally_positive} as
\begin{align*}
\psi(s)   =  \mu  s +\int_{(0,\infty)} (e^{-s z}-1 ) \Pi (\diff z), \quad s \geq 0, 
\quad \text{where} \quad \mu := \gamma + \int_{( 0,1)}z\, \Pi(\diff z).
\end{align*}
We exclude the case in which $X$ is a subordinator (i.e., $X$ has monotone paths a.s.). This implies that $\mu > 0$ when $X$ is of bounded variation. 

In the sequel, we will prove that, for a large class  of reward functions $f_p$ and $f_c$ satisfying only certain mild assumptions, a pair of threshold strategies leads to the Nash equilibrium of our NZSG. 
In particular, player $C$'s optimal strategy will be to stop at the first down-crossing time of some level, while player $P$'s optimal strategy will be to stop at the first Poisson time at which the process is below some other level. To this end,
we further denote, for $b \in \R$, the random times 
\begin{align} \label{hittingtimes} 
\tau_{b}^- :=\inf\{t>0:X_t<b \}  \in \mathcal{T}_c 
\quad \textrm{and} \quad  
T_b^- := \inf \{ T^{(n)}: X_{T^{(n)}} < b \} \in \mathcal{T}_p,	
\end{align}
where we recall $(T^{(n)})_{n\in\N}$ are the jump times of an independent Poisson process with rate $\lambda$. 
%
The objective of this section is to thus derive the expressions, 
for $x \in \R$ and $a \leq l$, of the values of these threshold strategies
\begin{align} \label{vf_0PC}
v_c(x;a,l) := \E_x \Big[  e^{-q \tau_{a}^-}f_c(X_{\tau_{a}^-}) 1_{\{\tau_{a}^-<T_l^-\}} \Big],
\quad \text{and} \quad
v_p(x;a,l) := \E_x \Big[ e^{-qT_l^-}f_p(X_{T_l^-}) 1_{\{T_l^-< \tau_{a}^-\}} \Big] 
\end{align}
in terms of the scale function of $X$. Note that $T_l^- \neq \tau_{a}^-$ a.s.\ thanks to the independence between $X$ and $N$.


\begin{remark} \label{remark_when_a_geq_l}
Notice that, for all threshold strategies with choices of $l \leq a$ 
and all $x \in \R$ in \eqref{vf_0PC}, we have $\tau_a^- < T_l^-$ a.s., hence 
$$v_c(x;a,l) = v^o_c(x;a) := \E_x \Big[  e^{-q \tau_{a}^-}f_c(X_{\tau_{a}^-}) 1_{\{ \tau_{a}^- < \infty \}}\Big] 
\quad \text{and} \quad v_p(x;a,l) = 0$$
which boils down to a one-player maximisation problem for player $C$, since player $P$ does not participate in the game under such a choice of $l \leq a$. 
\end{remark}

\subsection{Scale functions} 
We denote by $W^{({q})}:\R \to [0, \infty)$ the $q$-scale function corresponding to $X$ for $q > 0$.  
It takes value zero on the negative half-line, while on the positive half-line it is the unique continuous and strictly increasing function defined by
\begin{align*} 
\int_0^\infty e^{-\theta x} W^{({q})}(x) \, \diff x 
&= \frac 1 {\psi(\theta)-q}, \quad \theta > \Phi({q}),
\quad \text{where} \quad 
\Phi(q) := \sup \{ s \geq 0: \psi(s) = q\}.
\end{align*} 
\begin{remark} \label{remark_scale_function_properties}
Some known properties of the scale function $W^{(q)}$ are summarised as follows:
\begin{enumerate}
\item[(i)] The scale function $W^{(q)}$ is differentiable a.e. In particular, if $X$ is of unbounded variation or the \lev measure $\Pi$ is atomless, it is known that $W^{(q)}$ is $C^1(\R \backslash \{0\})$; see, e.g.,\ \cite[Theorem 3]{Chan2011}.
\item[(ii)] As in Lemma 3.1 of \cite{KKR}, we have
\begin{align*}
\begin{split}
W^{(q)} (0) &= \left\{ \begin{array}{ll} 0 & \textrm{if $X$ is of unbounded
variation,} \\ \mu^{-1} & \textrm{if $X$ is of bounded variation.}
\end{array} \right.  
\end{split}
\end{align*}
\end{enumerate}
\end{remark}

We also define for $r>0$,
\begin{align*} 
Z^{(r)}(x; \theta) &:= e^{\theta x} \left( 1 + (r- \psi(\theta )) \int_0^x e^{-\theta u} \, W^{(r)}(u) \du\right), 
\quad x \in \R, \, \theta \geq 0,
\end{align*}
which further yields 
\begin{align} 
\label{ZqlPq}
Z^{(q+\lambda)}(x; \Phi(q)) 
&=e^{\Phi(q) x} \left( 1 + \lambda \int_0^x e^{-\Phi(q) u} \, W^{(q+\lambda)}(u) \du 
\right), \quad x \in \mathbb{R} ,
\\
\label{Z_Phi_der}
Z^{(q+\lambda)\prime}(x; \Phi(q) ) &= \Phi(q) Z^{(q+\lambda)}(x; \Phi(q) ) + \lambda  W^{(q+\lambda)}(x), \quad x \in \R \backslash \{0\}.
\end{align}
The following related result is proved in Appendix \ref{proof_W/Z}.
\begin{lemma} \label{W/Z}
The mapping $u \mapsto W^{(q+\lambda)}(u) / Z^{(q+\lambda)}(u;\Phi(q))$ is increasing on $(0,\infty)$.
\end{lemma}


The scale functions are closely linked with several known fluctuation identities that will be used for the derivation of  \eqref{vf_0PC}.
%
By Theorem 3.12 in \cite{K}, we have 
\begin{align}
\E_x [e^{-q \tau_0^-}; \tau_0^-<\infty] 
= e^{-\Phi(q) x}, \quad x \geq 0. 
\label{upcrossing_identity}
\end{align}
Let $\tau_{b}^+ := \inf\{t>0:X_t > b \}$ for $b \in \R$.
By identity (8.11) of \cite{K}, we have 
\begin{align}
\label{upcr}
\E_x \left[e^{-q \tau_0^-}; \tau_b^+> \tau_0^- \right] = \frac {W^{(q)}(b-x)} {W^{(q)}(b)}, \quad x,b \geq 0,
\end{align}
and, as in identity (15) of \cite{AKP} or (5) of \cite{Albrecher}, we also have
\begin{align}
\E_{x}\bigg[e^{- q \tau_b^+ - \theta (X_{\tau_b^+}-b)};\tau_{0}^- > \tau_b^+ \bigg] 
= Z^{(q)}(b-x;\theta)-\frac{Z^{(q)}(b;\theta)}{W^{(q)}(b)}W^{(q)}(b-x), \quad x,b, \theta \geq 0. \label{laplace_upper}
\end{align}
	
Using Theorem 2.7.(i) in \cite{KKR}, for any locally bounded measurable function $f$, constants $b > a$ and $x \geq a$, the resolvents are given by
	\begin{align}\label{resolvent_density_0}
	&\E_x\bigg[ \int_0^{\tau_{a}^-\wedge \tau_b^+} 
	e^{- q s} f(X_s) \diff s \bigg] =\int_0^{b-a} f(b-u) \bigg\{ \frac{W^{(q)}(b-x)}{W^{(q)}(b-a)} W^{(q)} (b-a-u) - W^{(q)} (b-x-u) \bigg\} \diff u. 
	\end{align}

\subsection{Computation of $v_c(x;a,l)$ and $v_p(x;a,l)$ in \eqref{vf_0PC}.}
\label{fluct_new}

We firstly fix $\lambda > 0$ and define
\begin{align} \label{olW}
\mathscr{W}^{(q, \lambda)}_b (x) 
:= W^{(q)}(x+b) + \lambda \int_0^x W^{(q+\lambda)}(x-u) \, W^{(q)}(u+b) \du 
\quad \forall\; x, b\in\R,
\end{align}
where, in particular, 
\begin{align} \label{olW0}
\mathscr{W}^{(q, \lambda)}_b (x) = W^{(q)}(x+b) , \quad \forall\; x \leq 0 
\quad \text{and} \quad b \in \R.	
\end{align}
The proof of the following result is given in Appendix \ref{proof_prop_cost_cont_player}. 

\begin{lemma}\label{Prop_cost_cont_player}
For $b\geq l\geq a$ and any locally bounded measurable function $f_c$ on $\R$, we have 
\begin{align}\label{iden_h_SP}
\E_x \left[e^{-q \tau_a^-} f_c(X_{\tau_a^-}) 1_{\{\tau_a^-< T_l^- \wedge \tau_b^+\}} \right] = \left\{  \begin{array}{ll}f_c(a) \, 
\frac{\mathscr{W}^{(q, \lambda)}_{b-l}(l-x)}{\mathscr{W}^{(q, \lambda)}_{b-l}(l-a)}, & x > a, \\ f_c(x), & x \leq a.
\end{array} \right.
\end{align}
\end{lemma}


By taking the limit as $b \uparrow \infty$ in \eqref{iden_h_SP}, we obtain the desired expression for $v_c(x;a,l)$ defined in \eqref{vf_0PC}. The proof of the following is deferred to Appendix  \ref{proof_limit_cont_SP}. 


\begin{proposition}\label{limit_cont_SP}
For $l\geq a$ and any locally bounded measurable function $f_c$ on $\R$, the function $v_c(x;a,l)$ from \eqref{vf_0PC} is given by 
\begin{align} \label{vf_1C}
v_c(x;a,l)
= \begin{cases}
f_c(a) \, \dfrac{Z^{(q+\lambda)}(l-x;\Phi(q))}{Z^{(q+\lambda)}(l-a;\Phi(q))} 
\,, & \text{for } x > a \,, \\
f_c(x) \,, & \text{for } x \leq a\,. 
\end{cases}
\end{align}
\end{proposition}

	By \eqref{ZqlPq} together with Proposition \ref{limit_cont_SP}, for $x\geq l$, we particularly have  
\begin{align*} 
v_c(x;a,l) &= f_c(a) \, \frac{e^{\Phi(q)(l-x)}}{Z^{(q+\lambda)}(l-a;\Phi(q))} 
= e^{\Phi(q)(l-x)} \, v_c(l;a,l) \,,
\end{align*}
where 
\begin{align} 
 \label{vCl}
v_c(l;a,l) &=\frac{ f_c(a)}{Z^{(q+\lambda)}(l-a;\Phi(q))} \,.
\end{align}

Consider now $f_p$ to be any locally bounded measurable function on $\R$.
Then, we also define
 \begin{align} \label{Gamma}
 \Gamma(x;l) &:= \int_0^{l-x} f_p(l-u) W^{(q+\lambda)}(l-x-u) \du = \int_0^{l-x} f_p(u+x) W^{(q+\lambda)}(u) \du, \quad \forall\; x \leq l.
  \end{align}

The proof of the following result is given in Appendix \ref{proof_prop_cost_per_player}. 
\begin{lemma}\label{Prop_cost_per_player}
For $b\geq l\geq a$ and any locally bounded measurable function $f_p$ on $\R$, we have  
\begin{align}\label{iden_g_SP}
\E_x\left[e^{-q T_l^-} f_p(X_{T_l^-}) 1_{\{T_l^-< \tau_{a}^- \wedge \tau_b^+\}}\right] 
 = \left\{  \begin{array}{ll} \lambda  \Big( \frac{\mathscr{W}^{(q, \lambda)}_{b-l}(l-x)}{\mathscr{W}^{(q, \lambda)}_{b-l}(l-a)}  \Gamma(a;l) 
  - \Gamma(x;l) \Big), 
 & \text{for } x > a, \\ 0, & \text{for } x \leq a.  \end{array} \right.
\end{align}
\end{lemma}


By taking the limit as $b \uparrow \infty$ in \eqref{iden_g_SP}, we obtain the desired expression for $v_p(x;a,l)$ 
defined in \eqref{vf_0PC}.  
The proof of the following is deferred to Appendix  \ref{proof_limit_math_W}.
\begin{proposition}\label{limit_math_W}
For $l \geq a$ and any locally bounded measurable function $f_p$ on $\R$, the function $v_p(x;a,l)$ from \eqref{vf_0PC} is given by 
\begin{align} \label{vf_1P}
\begin{split}
v_p(x;a,l) 
 &= \left\{  \begin{array}{ll} \lambda \Big( \dfrac{Z^{(q+\lambda)}(l-x;\Phi(q))}{Z^{(q+\lambda)}(l-a;\Phi(q))} 
 \Gamma(a;  l) - \Gamma(x; l) \Big), & \text{for } x > a, \\ 0, & \text{for } x \leq a.  \end{array} \right.
\end{split}
\end{align}

\end{proposition}

	By Proposition \ref{limit_math_W}, for $x \geq l$, we particularly have  
\begin{align}\label{vf_simpleP}
v_p(x;a,l) &= \lambda \, \frac{e^{\Phi(q)(l-x)}}{Z^{(q+\lambda)}(l-a;\Phi(q))} 
 \Gamma(a;l)
= e^{\Phi(q)(l-x)} \, v_p(l;a,l),
\end{align}
where 
\begin{align} \label{vPl}
v_p(l;a,l) &= 
 \frac{\lambda}{Z^{(q+\lambda)}(l-a;\Phi(q))} 
 \Gamma(a;l).
\end{align}

\begin{remark}
One possible extension of our game is to consider the case when player $C$ pays a running cost for the additional exercise opportunities provided. 
In particular, if we consider $R:\R\mapsto [0,\infty)$ to be the instantaneous running cost, then the candidate value function $\tilde{v}_c(x;a,l)$ for player $C$ will be given by
\begin{align*}
\widetilde{v}_c(x;a,l)=\mathbb{E}_x\left[e^{-q\tau_a^-}f_c(X_{\tau_a^-})1_{\{\tau_a^-<T_l^-\}} - \int_0^{\tau_a^-\wedge T_l^-}e^{-qt}R(X_t)\diff t \right],
\end{align*}
for all $x \geq a$.
Then, using identity (3.13) in Theorem 3.1 of \cite{LLWX}, we can obtain that 
\begin{align*}
\widetilde{v}_c(x;a,l)
= \,&f_c(a)\frac{Z^{(q+\lambda)}(l-x,\Phi(q))}{Z^{(q+\lambda)}(l-a,\Phi(q))} 
-\int_a^{l}R(u)\left\{\frac{Z^{(q+\lambda)}(l-x,\Phi(q))}{Z^{(q+\lambda)}(l-a,\Phi(q))}W^{(q+\lambda)}(u-a)-W^{(q+\lambda)}(u-x)\right\}\diff u\notag\\
&-\int_l^{\infty}R(u)\left\{\frac{Z^{(q+\lambda)}(l-x,\Phi(q))}{Z^{(q+\lambda)}(l-a,\Phi(q))}\mathcal{W}_{u-l}^{(q,\lambda)}(l-a)-\mathcal{W}^{(q,\lambda)}_{u-l}(l-x)\right\}\diff u, \qquad x\geq a,
\end{align*}
and $\widetilde{v}_c(x;a,l)=f_c(x)$ otherwise. 
This can be used instead of $v_c$ in the subsequent analysis.
\end{remark}

\section{Optimality over threshold strategies}  \label{section_threshold_strategies}

In this section, we consider a version of the game where admissible strategies are restricted to be of threshold-type (this will be strengthened in Section \ref{section_optimality_general}). 
Hence, the value functions of the two players take the form of \eqref{vf_0PC} and the objective \eqref{Nasheq} becomes to find a Nash equilibrium $(a^*, l^*) \in \R^2$ satisfying simultaneously the following two equations:
\begin{align} \label{nash_threshold}
\begin{split}
v_c(x;a^*,l^*)  &= \max_{a \in \R} v_c(x;a,l^*), \\
v_p(x;a^*,l^*)  &= \max_{l \in \R} v_p(x;a^*,l). 
\end{split}
\end{align}
Although the barriers $(a, l)$ are allowed to depend on $x$, the values of $(a^*, l^*)$ that we will obtain are invariant of $x$.



For the rest of the paper, we make the following assumption on the reward functions.

\begin{assump} \label{Ass}
The reward functions $f_c(\cdot)$ and $f_p(\cdot)$ satisfy the following properties:  
\begin{enumerate}
\item[(i)]  We have $f_c(\cdot) \leq f_p(\cdot)$ 
on $\R$. 
\item[(ii)] For $i \in\{c,p\}$, the function $f_i(\cdot)$ is strictly decreasing, continuously differentiable and concave on $\R$ and admits 
a constant 
 \begin{align} \label{about_x_upper}
\ol x_i \in \R \qquad \textrm{such that } \qquad f_i(x)  \left\{ \begin{array}{ll} > 0, & x < \ol x_i,  \\ 
< 0, & x > \ol x_i. \end{array} \right.
 \end{align}
%
\end{enumerate}
\end{assump}



It is worth noting that Assumption \ref{Ass}.(i) reflects either a symmetric reward structure of $f_c \equiv f_p$ (still with asymmetric opportunities), or an asymmetric reward structure of $f_c < f_p$. 
In the main application of firm acquisitions (see the discussion in our motivation section in the Introduction), the reward $f_c$ for the already diversified firm $C$ -- representing the benefits obtained upon acquisition minus the market value of the target firm, which is the cost of acquisition -- is lower that the reward $f_p$ of firm $P$ due to the synergy effect (see, e.g.\ \cite{Zhou} for more details).
Assumption \ref{Ass}.(ii) is a natural condition widely applicable. 
The decreasing reward functions reflect the game's ``optimal purchasing'' nature, given that when purchasing a real asset, the buyer increases their reward when the asset price to be paid decreases. 
For instance, as it is standard to model a firm's value as an exponential Lévy process, and given that this value represents the cost of acquisition, the reward functions $f_\cdot($X$)$ become clearly decreasing and concave in terms of the Lévy process $X$. 
If instead the firm's value is modelled by $X$ itself, then the reward functions $f_\cdot($X$)$ become decreasing and linear, which also fits Assumption \ref{Ass}.(ii) (see Section \ref{section_numerics} for a couple of case studies).

\begin{remark} \label{dual_case}
By considering the dual process $-X$, the results in this paper hold also for the case driven by a spectrally negative \lev process and strictly increasing, continuously differentiable and concave $f_i$, for $i\in\{c,p\}$, which can also expand even further the application span of this paper.
\end{remark}

In view of Assumption \ref{Ass}, we automatically get that
\begin{align}
\overline{x}_c \leq \overline{x}_p. 
\label{x_bar_inequality}
\end{align}
The values $\ol x_c$ and $\ol x_p$ 
will act as upper bounds for the optimal barriers for players $C$ and $P$; this is intuitive because it is obviously suboptimal to stop when the reward is negative. 

\subsection{Benchmark case: Single-player setting}
\label{lambda=0}

We begin our analysis with the consideration of the special case when $\lambda=0$, i.e. player $P$ can never stop, as a benchmark. 
This involves only player $C$ whose expected reward under a threshold strategy $\tau_a^-$ is derived by \eqref{upcrossing_identity} and is given by 
(cf. Remark \ref{remark_when_a_geq_l}) 
\begin{align*} 
v_c^o (x;a) := \E_x \Big[  e^{-q \tau_{a}^-}f_c(X_{\tau_{a}^-}) 1_{\{ \tau_{a}^- < \infty \}}\Big] = \left\{ \begin{array}{ll} f_c(x) &\textrm{for }  x \leq a, \\  e^{\Phi(q) (a-x)} f_c(a) &\textrm{for } x > a. \end{array} \right.
\end{align*}
Straightforward differentiation gives
\begin{align} \label{v_o_derivative}
\frac \partial {\partial a}v_c^o(x;a)  = \left\{ \begin{array}{ll} 0 &\textrm{for }  x < a, \\  e^{\Phi(q) (a-x)}h_c^o(a) &\textrm{for } x > a, \end{array} \right.
\end{align}
where we define
\begin{align}
h_c^o(x) 
:= \Phi(q) f_c(x) + f_c'(x), \quad x \in \R. \label{h_c_o}
\end{align}
By Assumption \ref{Ass}, the function $h_c^o(\cdot)$ is continuous, monotonically decreasing and in particular satisfies $h_c^o(\ol x_c) = f'_c(\ol x_c) < 0$. Hence, there exists 
\begin{align} \label{a_underbar}
\underline{a} \in [-\infty, \ol x_c) 
\qquad \textrm{such that for $x\in\R$,} \qquad  h_c^o(x) \left\{ \begin{array}{ll} 
> 0, & x <  \underline{a},  \\  
< 0, & x >  \underline{a}. \end{array} \right.
 \end{align}
\noindent 
Using Theorem 2.2 of \cite{LZ}, we can conclude that $\tau_{\underline{a}}^-$ is in fact the maximiser over all stopping times when $\underline{a}> -\infty$.
Indeed, in light of Remark \ref{dual_case}, the assumptions imposed in \cite{LZ} (where they consider the spectrally negative case) are satisfied by the properties of $h_c^o$ defined in \eqref{h_c_o}. Hence, we have 
\begin{align} 
v_c^o (x;\underline{a}) = \max_{a\in \R} v_c^o (x;a) = \sup_{\tau \in\mathcal{T}_c }  \E_x \Big[  e^{-q \tau}f_c(X_{\tau}) 1_{\{\tau < \infty \}} \Big], \quad x \in \R. \label{v_inf_max}
\end{align}
Instead, when $\underline{a} = -\infty$, an optimal stopping time does not exist.  

Identity \eqref{v_inf_max} also provides the solution to the degenerate case discussed in Remark \ref{remark_when_a_geq_l}.

\subsection{
Preliminary results}

In view of Section \ref{lambda=0}, we make the following standing assumption in the rest of the paper, which essentially rules out the case when player $C$ should optimally never stop (when there is no opponent), as this is not interesting in terms of applications.
\begin{assump} \label{assump_a_underbar_finite}
We assume that $\underline{a}$ defined in \eqref{a_underbar} satisfies $\underline{a} > -\infty$.
\end{assump}


Analogous to $h_c^o(\cdot)$ as in  \eqref{h_c_o}, we define 
the following continuous functions, for all $x \in \mathbb{R}$, 
\begin{align} \label{def_h_p_c}
\begin{split}
h_c(x) &:= (\Phi(q)+\lambda W^{(q+\lambda)}(0)) \, f_c(x) + f_c'(x) 
= h_c^o(x) + \lambda W^{(q+\lambda)}(0) \, f_c(x), \\ 
h_p(x) &:= \Phi(q) f_p(x) + f'_p(x).
\end{split}
\end{align} 
For $i \in\{c,p\}$, 
thanks to Assumption \ref{Ass} and $h_i(\ol x_i) = f'_i(\ol x_i) < 0$, there exist 
 \begin{align} \label{x_p_and_h_p}
\underline{x}_i \in [-\infty, \overline{x}_i) \qquad \textrm{such that for $x\in\R$,} \qquad  h_i(x) \left\{ \begin{array}{ll} 
> 0, & x <  \underline{x}_i,  \\  
< 0, & x >  \underline{x}_i. \end{array} \right.
 \end{align}



\begin{remark} \label{remark_h_c_o}
It is straightforward to see from \eqref{def_h_p_c} that the function $h_c(\cdot)$ coincides with $h_c^o(\cdot)$ and $\underline{x}_c = \underline{a}$ if and only if $W^{(q+\lambda)}(0) = 0$ or if and only if $X$ is of unbounded variation (see Remark \ref{remark_scale_function_properties}.(ii)).
\end{remark}

While $\underline{x}_p$ may be equal to $-\infty$,  Assumption \ref{assump_a_underbar_finite} guarantees the finiteness of $\underline{x}_c$ as shown in the following result. This will be important in showing the existence of a Nash equilibrium, since $\underline{x}_c$ will act as a lower bound for $(a^*, l^*)$ (see Lemma \ref{lem:exist_a} below). The proof is given in Appendix \ref{proof_lemma_x_c_finite}.

\begin{lemma} \label{lemma_x_c_finite}
Recall $\underline{a}$ and $\underline{x}_c$ defined in \eqref{a_underbar} and \eqref{x_p_and_h_p}, respectively.  We have 
$-\infty < \underline{a} \leq \underline{x}_c$.
\end{lemma}



Given the above result in Lemma \ref{lemma_x_c_finite} and the observations in Remark \ref{remark_when_a_geq_l}, we aim in the following result at restricting our focus on a strict subset of 
$\R^2$ 
for the selection of $(a^*, l^*)$ leading to the maximisation of \eqref{nash_threshold}.
The proof is given in Appendix \ref{proof_remark_domain_to_focus}.


\begin{lemma}  \label{remark_domain_to_focus}
For $x \in \R$, the problem in \eqref{nash_threshold} satisfies:
\begin{enumerate}
\item[(i)] For any $l \in \R$, we have 
$\max_{a \in \R} v_c(x;a,l) = \max_{a \in (-\infty, x \wedge \overline{x}_c]} v_c(x;a,l)$;

\item[(i)'] If $l \geq \underline{a}$ in {\rm (i)}, then we have
$\max_{a \in \R} v_c(x;a,l) = \max_{a \in (-\infty, x \wedge l \wedge \overline{x}_c]} v_c(x;a,l)$;

\item[(ii)] For any $a \in \R$, we have
$\max_{l \in \R} v_p(x;a,l) = \max_{l \in [a, 
\overline{x}_p]} v_p(x;a,l)$.
\end{enumerate}
\end{lemma}

Given that a Nash equilibrium $(a^*, l^*) \in \R^2$ must satisfy both equations in \eqref{nash_threshold} simultaneously, the latter system can be written in light of Lemma \ref{remark_domain_to_focus}.(i)--(ii) in the form of
\begin{align} \label{nash_threshold_new}
\begin{split}
v_c(x;a^*,l^*)  &= \max_{a \in (-\infty, x \wedge l^* \wedge \overline{x}_c]} v_c(x;a,l^*), \\
v_p(x;a^*,l^*)  &= \max_{l \in [a^*, \overline{x}_p]} v_p(x;a^*,l). 
\end{split}
\end{align}
%
%


\subsection{First-order conditions.}
\label{1storder}

In this section we will characterise the candidate (optimal) thresholds $(a^*,l^*)$ by means of using the first-order conditions for the candidate value functions in \eqref{vf_0PC}, given by \eqref{vf_1C} and \eqref{vf_1P}.

To this end, 
for any fixed $l\in\R$, we define for $a \in (-\infty,l]$
the function 
\begin{align} \label{Jal}
\begin{split}
I(a;l) 
&:= f_c'(a) + \big( \Phi(q) \, Z^{(q+\lambda)}(l-a;\Phi(q)) + \lambda \, W^{(q+\lambda)}(l-a) \big) \, v_c(l;a,l) \\
&= f_c'(a) + \Phi(q) \, f_c(a) + \lambda \, W^{(q+\lambda)}(l-a) \, v_c(l;a,l),
\end{split}
\end{align}
where the second equality holds by \eqref{vCl}. 
In particular, for 
$a=l\in\R$, we have from \eqref{vCl} that $v_c(l;l,l) = f_c(l)$, hence by \eqref{def_h_p_c} 
we get
\begin{align} \label{J_l_l}
I(l;l) = h_c(l), \quad \forall\; l \in \R.
\end{align}
For any fixed $a\in\R$, we also define for $l \in [a,\infty)$ the function 
\begin{align} \label{Ila}
J(l;a)
&:= Z^{(q+\lambda)}(l-a;\Phi(q)) \, \big( f_p(l)-v_p(l;a,l) \big) 
= f_p(l) \, Z^{(q+\lambda)}(l-a;\Phi(q)) 
- \lambda \Gamma(a;l),
\end{align}
where the second equality holds by \eqref{vPl}. In particular, taking into account \eqref{ZqlPq} and  \eqref{Gamma}, we get
\begin{align} \label{I_a_a}
J(a;a) = f_p(a), \quad \forall\; a \in \R.
\end{align}

Below, we compute the partial derivatives of \eqref{vf_1C} and \eqref{vf_1P} with respect to the threshold under control of each player; the proof is given in Appendix \ref{proof_lemma_derivatives}.

\begin{lemma} \label{lemma_derivatives}
Consider $v_c$ and $v_p$ in \eqref{vf_1C} and \eqref{vf_1P}, respectively. Then, we have:
\begin{enumerate}
\item[(i)] For $a < l \wedge x$,
\begin{align*}
\frac{\partial}{\partial a} v_c(x;a,l) 
&= \frac{Z^{(q+\lambda)}(l-x;\Phi(q))}{Z^{(q+\lambda)}(l-a;\Phi(q))} I(a;l).
\end{align*}

\item[(ii)] For $x \geq a$ and $l > a$ such that $l\neq x$,
\begin{align*}
\frac \partial {\partial l} v_p(x;a,l) 
&= \lambda \bigg( \dfrac{Z^{(q+\lambda)}(l-x;\Phi(q))}{Z^{(q+\lambda)}(l-a;\Phi(q))}W^{(q+\lambda)}(l-a) - W^{(q+\lambda)}(l-x) \bigg) \big( f_p(l) - v_p(l;a,l)  \big) \\
&\stackrel{\normalfont{\eqref{Ila}}}{=} \frac{\lambda}{Z^{(q+\lambda)}(l-a;\Phi(q))} \bigg( \dfrac{Z^{(q+\lambda)}(l-x;\Phi(q))}{Z^{(q+\lambda)}(l-a;\Phi(q))}W^{(q+\lambda)}(l-a) - W^{(q+\lambda)}(l-x) \bigg) J(l;a). 
\end{align*}
\end{enumerate}
\end{lemma}



Using an appropriate modification of the identity in \eqref{laplace_upper} for $l \geq a$, we observe that
\[
\dfrac{Z^{(q+\lambda)}(l-x;\Phi(q))}{Z^{(q+\lambda)}(l-a;\Phi(q))}W^{(q+\lambda)}(l-a) - W^{(q+\lambda)}(l-x) > 0, \quad \forall\; x \geq a,
\]
{and hence} 
we conclude from Lemma \ref{lemma_derivatives}.(ii) and the positivity of $Z^{(q+\lambda)}$ from \eqref{Z_Phi_der} that
\begin{align} \label{sign_matches}
\text{sign} \Big(\frac \partial {\partial l} v_p(x;a,l) \Big)
&= \text{sign} \Big( f_p(l) - v_p(l;a,l)  \Big) 
= \text{sign} \Big( J(l;a) \Big), 
\quad \forall\; x \geq a, \; l > a. 
\end{align}
We can therefore extract from Lemma \ref{lemma_derivatives} the following two necessary conditions for the optimality of threshold strategies:
\begin{enumerate}
\item The first-order condition $\frac{\partial}{\partial a} v_c(x;a,l) =0$ for $x\geq a$, required for the optimality (best response to any given $l$) of the candidate threshold $a$, implies that the following condition should hold:
\begin{align}\label{cond_aC}
\mathbf{C}_a: I(a; l)
= 0 \,.
\end{align}

\item The first-order condition $\frac{\partial}{\partial l}v_p(x;a,l)=0$ for $x\geq a$, required for the optimality (best response to any given $a$) of the candidate threshold $l$, implies that the following condition should hold:
\begin{align}\label{cond_lP}
\mathbf{C}_l: f_p(l) = v_p(l;a,l) \quad \Leftrightarrow \quad J(l;a) = 0, 
\end{align}
\end{enumerate}
Overall, the candidate (equilibrium) threshold pair $(a^*,l^*)$ should satisfy both conditions \eqref{cond_aC} and \eqref{cond_lP}, or equivalently satisfy the system of equations $I(a^*,l^*)=J(l^*,a^*)=0$.
The study of the functions $I$ and $J$ will thus be fundamental in the forthcoming analysis.

\subsection{Existence and uniqueness of the 
best response threshold strategies $a^*$ and $l^*$.}
\label{exist}

Now, we will check conditions for the existence and uniqueness 
of the candidate (optimal) thresholds $a^*$ and $l^*$ as best responses to arbitrary choices of thresholds by the opponent player. 

%
%

\subsubsection{Best response for Player C}
\label{exist_a}

Recall that, in order for the candidate (equilibrium) threshold pair $(a^*,l^*)$ to satisfy condition $\mathbf{C}_a$ given by \eqref{cond_aC}, we must have $I(a^*;l^*)=0$.
Hence, we aim at proving the existence of a unique solution to the equation $I(\cdot;l)=0$ for an appropriate range of values of $l$. We begin with the most general case of $l \in\R$, by focusing only on player $C$ (cf.\ Lemma \ref{remark_domain_to_focus}.(i)--(i)'), even though for optimality we will later only require $l^* \leq \overline{x}_p$ (cf.\ \eqref{nash_threshold_new} for $v_p$). 
In view of \eqref{nash_threshold_new} for $v_c$, we therefore search for solutions to  $I(\cdot;l)=0$ on $(-\infty, \ol x_c \wedge l]$. 
Refer to Figure \ref{figure_C_l_a}.(i) in Section \ref{section_numerics} for sample plots of the function $a \mapsto I(a; l)$.


\begin{lemma} \label{lem:exist_a}
Recall that $\underline{x}_c > -\infty$ (cf.\ Lemma \ref{lemma_x_c_finite}) and the function $I(\cdot;l)$ defined in \eqref{Jal}.
\begin{enumerate}
\item[(I)] Suppose that $l <  \underline{x}_c$. Then, there does not exist $a \leq l$ such that $I(a;l) = 0$.

\item[(II)] 
Suppose that $l \geq \underline{x}_c$. Then, there exists on $(-\infty, \ol x_c \wedge l]$, 
a unique root $\tilde{a}(l)$ such that $I(\tilde{a}(l);l) = 0$, i.e. satisfying \eqref{cond_aC}. In addition, we have 
\begin{enumerate}
\item[(i)]
$\tilde{a}(l) \in [\ul x_c, \, \ol x_c \wedge l)$ 
and $h_c(\tilde{a}(l)) \leq 0$. 
In particular,
$\tilde{a}(\underline{x}_c) = \underline{x}_c.$

\item[(ii)] 
$I(a;l) > 0$ for all $a < \tilde{a}(l)$ and $I(a;l) < 0$ for all $a > \tilde{a}(l)$. 

\item[(iii)] 
for all $x \in \mathbb{R}$, that 
$v_c(x;\tilde{a}(l),l) = \max_{a \in (-\infty, x \wedge l \wedge \overline{x}_c]} v_c(x;a,l)$  
which is further upgraded to 
$$v_c(x;\tilde{a}(l),l)  = \max_{a \in \R} v_c(x;a,l).$$
\end{enumerate}
\end{enumerate}
\end{lemma}
\begin{proof}
%

Suppose $a \leq \underline{x}_c $ and $a < l$.  Then, we have by \eqref{x_p_and_h_p} that $h_c(a) \geq 0$, or equivalently by \eqref{def_h_p_c} that $f_c'(a) + \Phi(q) \, f_c(a) \geq - \lambda W^{(q+\lambda)}(0) \, f_c(a)$.
Taking this into account together with the definition \eqref{Jal} of $I(\cdot;l)$ and the equation \eqref{vCl}, we get  
\begin{align} \label{I(a)>0} 
I(a;l) 
&\geq \lambda \left( W^{(q+\lambda)}(l-a) \, v_c(l;a,l) - W^{(q+\lambda)}(0) \, f_c(a) \right) \notag\\
&= \lambda \, f_c(a) \left( \frac{W^{(q+\lambda)}(l-a)}{Z^{(q+\lambda)}(l-a;\Phi(q))} - W^{(q+\lambda)}(0) \right) > 0,
\end{align}
where the latter inequality follows from Lemma \ref{W/Z}, the fact that $Z^{(q+\lambda)}(0;\Phi(q)) = 1$ and because $f_c(a) > 0$ thanks to $a  \leq  
\underline{x}_c  < \overline{x}_c$ and \eqref{about_x_upper}.

{\it Part }(I). For $l < \underline{x}_c$, it is straightforward to see from the above result that $I(a;l)  > 0$ for all $a < l <  \underline{x}_c$, {while we also know that $I(l;l) = h_c(l) > 0$ for all $l <  \underline{x}_c$}. 

{\it Part }(II). 
{\it Step 1: Existence.} For $l = \underline{x}_c$, \eqref{J_l_l} gives (recalling $\underline{x}_c > -\infty$ by Lemma \ref{lemma_x_c_finite})
\begin{align} \label{J_at_x_c_under}
I(\underline{x}_c; \underline{x}_c) = h_c(\underline{x}_c) = 0,
\end{align}
thus existence is straightforward. 
For $l > \underline{x}_c$, 
we have from \eqref{I(a)>0} that $I(\underline{x}_c;l) > 
0$. 
On the other hand, it follows from \eqref{J_l_l} and \eqref{x_p_and_h_p} with $l >  \underline{x}_c$, that 
\begin{align*} 
I(l;l) = h_c(l) < 0 \,,
\end{align*}
and when $l>\overline{x}_c$, we also have from \eqref{Jal} and \eqref{vCl}  that
\begin{align*} 
I(\overline{x}_c;l) &= f_c'(\overline{x}_c) < 0.
\end{align*}
 Now thanks to the continuity of $I(\cdot; l)$, there must exist at least one $a$ such that $I(a;l) = 0$ on 
 $[\underline{x}_c, l \wedge \overline{x}_c]$.

{\it Step 2: Uniqueness.} 
By differentiating \eqref{Jal} and using Lemma \ref{lemma_derivatives}.(i), we obtain that 
\begin{align}
\label{uni_der_J_2}
\begin{split}
\frac{\partial}{\partial a}I(a+;l) 
= f_c''(a+) + \Phi(q) f_c'(a) - \lambda\, {W^{(q+\lambda)}}'((l-a)-) \, v_c(l;a,l) + \lambda \, \frac{W^{(q+\lambda)}(l-a)}{Z^{(q+\lambda)}(l-a;\Phi(q))} \, I(a;l).
\end{split}
\end{align}
Given that the sign of \eqref{uni_der_J_2} is unclear, we define, for all $a\in(-\infty, l \wedge \overline{x}_c]$, the function
\begin{align}\label{def_overline_J}
\overline{I}(a;l):=\exp\bigg\{\lambda\int_a^l\frac{W^{(q+\lambda)}(l-u)}{Z^{(q+\lambda)}(l-u;\Phi(q))}\diff u\bigg\}I(a;l).
\end{align}
The first derivative of $\ol I(\cdot;l)$ can be obtained by using \eqref{uni_der_J_2}, namely, we have for all $a\leq l \wedge \overline{x}_c$ that 
\begin{align*}
\frac{\partial}{\partial a} \overline{I}(a+;l) = 
\exp\bigg\{\lambda\int_a^l\frac{W^{(q+\lambda)}(l-u)}{Z^{(q+\lambda)}(l-u;\Phi(q))} \diff u\bigg\} \left( f_c''(a+) + \Phi(q) f_c'(a) - \lambda \, {W^{(q+\lambda)}}'((l-a)-) \, v_c(l;a,l) \right) 
\end{align*}
which is negative due to 
Assumption \ref{Ass} and the positivity of ${W^{(q+\lambda)}}'((l-a)-)$ and $v_c(l;a,l)$  thanks to $a < \overline{x}_c$.
Hence, for any fixed $l\in\R$, the mapping $a\mapsto \overline{I}(a;l)$ is 
decreasing on $(-\infty, l \wedge \overline{x}_c)$, which yields $\ol I(\cdot;l) = 0$ has at most one solution  on $(-\infty, l \wedge \overline{x}_c)$. 
In view of its definition in \eqref{def_overline_J}, it is straightforward to see that also $I(\cdot;l)=0$ has at most one solution 
 on $(-\infty, l \wedge \overline{x}_c)$.
Hence, the solution 
to the equation $I(\tilde{a}(l);l) = 0$, established in {\it Step 1} on $[\ul x_c, \,\ol x_c \wedge l)$ is unique and denoted by $\tilde{a}(l)$. 

{\it Step 3: Proof of part} (i). This follows from the inequalities obtained in {\it step 1}.

{\it Step 4: Proof of part} (ii). This follows directly from the results in {\it steps 1} and {\it 2}.

{\it Step 5: Proof of part} (iii). Lemma \ref{lemma_derivatives}.(i) and the above results imply that $\tilde{a}(l)$ satisfies
$$v_c(x;\tilde{a}(l),l) = \max_{a \in (-\infty, x \wedge l \wedge \overline{x}_c]} v_c(x;a,l).$$ 
Since $l \geq \underline{x}_c \geq \underline{a}$, Lemma \ref{remark_domain_to_focus}.(i)' gives $v_c(x;\tilde{a}(l),l)  = \max_{a \in \R} v_c(x;a,l)$, which completes the proof. 
\end{proof}

Below we present the continuity and monotonicity of the threshold $\tilde{a}(l)$ with respect to the arbitrarily chosen -- until this stage of analysis -- threshold 
$l \in [\ul x_c,\infty)$, according to Lemma \ref{lem:exist_a}.(II).
\begin{lemma}
\label{aincr}
The function $l \mapsto \tilde{a}(l)$, which is defined in Lemma \ref{lem:exist_a}.(II), is continuous and strictly increasing in $l$ on {$(\ul x_c, \infty)$. }
\end{lemma}
\begin{proof}
To show the continuity,
we argue by contradiction, assuming that 
it fails to be continuous at some $l^\dagger$, so that there exist two sequences $(l_n^-)$ and $(l_n^+)$ converging to $l^\dagger$ with $\lim_{n \to \infty}\tilde{a}(l_n^-) \neq \lim_{n \to \infty}\tilde{a}(l_n^+)$.
Then, noting that $(a,l) \mapsto I(a;l)$ is continuous, we have
\[
I(\lim_{n \to \infty} \tilde{a}(l_n^-); l^\dagger) = \lim_{n \to \infty}I(\tilde{a}(l_n^-); l_n^-) = 0 =  \lim_{n \to \infty}I(\tilde{a}(l_n^+); l_n^+) =  I(\lim_{n \to \infty} \tilde{a}(l_n^+); l^\dagger).
\]
This contradicts with the uniqueness of the root  $\tilde{a}(l^\dagger)$ of $I(\cdot; l^\dagger) = 0$ as established in 
Lemma \ref{lem:exist_a}.

Now to show the monotonicity, we 
combine the definition of $I(a;l)$ in \eqref{Jal} with the expression in \eqref{vCl} and we calculate the partial derivative of $I(a;l)$ with respect to $l$, for all $a < \overline{x}_c$ (so that $f_c(a) > 0$), given by
\begin{align}  \label{partialJl}
\frac{\partial}{\partial l} I(a; l) 
&= \lambda \, \frac{\partial}{\partial l} \left( W^{(q+\lambda)}(l-a) \, v_c(l;a,l) \right) = \lambda \, f_c(a) \, \frac{\partial}{\partial l} \bigg( \frac{W^{(q+\lambda)}(l-a)}{Z^{(q+\lambda)}(l-a;\Phi(q))} \bigg) > 0, 
\end{align}
for a.e. $l > a$; the latter inequality follows from Lemma \ref{W/Z}.
%
Then, we argue again by contradiction, assuming
that there exists $l^\dagger \in 
(\ul x_c, \infty)$ and 
$\delta > 0$, so that 
$\tilde{a}(l^\dagger) \geq \tilde{a}(l^\dagger + \delta)$. 
Since $I(\tilde{a}(l^\dagger); l^\dagger) = I(\tilde{a}(l^\dagger+ \delta) ; l^\dagger+\delta) = 0$, we have
\[
I(\tilde{a}(l^\dagger+ \delta) ; l^\dagger+\delta) - I(\tilde{a}(l^\dagger) ; l^\dagger+\delta) = I(\tilde{a}(l^\dagger) ; l^\dagger) - I(\tilde{a}(l^\dagger) ; l^\dagger+\delta) < 0, 
\]
where the last inequality holds by \eqref{partialJl} (and because $\tilde{a}(l^\dagger) < \overline{x}_c$ by Lemma \ref{lem:exist_a}.(II).(i)). 
Hence, again by $I(\tilde{a}(l^\dagger+ \delta) ; l^\dagger+\delta)=0$, we get
\[
I(\tilde{a}(l^\dagger) ; l^\dagger+\delta) > 0.
\]
Now the assumption $\tilde{a}(l^\dagger) \geq \tilde{a}(l^\dagger+ \delta)$ contradicts with the fact that $I(a; l^\dagger+\delta) \leq 0$ for $a \geq \tilde{a}(l^\dagger+\delta)$ as in Lemma \ref{lem:exist_a}.(II).(ii), which completes the proof.
\end{proof}

\subsubsection{Best response for Player P}
\label{exist_l}


By Proposition \ref{limit_math_W} and the expressions in \eqref{vPl}
and the definition \eqref{Ila} of $J$,
we get, for $a\leq x$, that
\begin{align} \label{v_p_other_form}
\begin{split}
v_p(x;a,l) 
&=Z^{(q+\lambda)}(l-x;\Phi(q)) v_p(l;a,l)
- \lambda \Gamma(x;l) \\
&= Z^{(q+\lambda)}(l-x;\Phi(q)) \, \big( v_p(l;a,l) - f_p(l) \big) + J(l;x). 
\end{split}
\end{align}

 
Recall that, in order for the candidate equilibrium threshold pair $(a^*,l^*)$ to satisfy condition $\mathbf{C}_l$ given by \eqref{cond_lP}, we must have $J(l^*;a^*)=0$. Hence, we aim at proving the existence of a unique solution to the equation $J(\cdot;a)=0$ for an appropriate range of values of $a$.
We begin with the most general case of $a \in (-\infty, \overline{x}_p)$, by focusing only on player $P$ (cf.\ Lemma \ref{remark_domain_to_focus}.(ii)), even though for optimality we will later only require $a^* \leq \overline{x}_c$ (cf.\ \eqref{nash_threshold_new} for $v_c$). 
In view of \eqref{nash_threshold_new} for $v_p$, we therefore search for solutions to $J(\cdot;a)=0$ on $[a, \overline{x}_p]$. 
Refer to Figure \ref{figure_C_l_a}.(ii) in Section \ref{section_numerics} 
for sample plots of the function $l \mapsto J(l; a)$.
 


\begin{lemma} \label{lem:exist_l}
Fix $a < \overline{x}_p$ 
and recall the function $J(\cdot;a)$ defined in \eqref{Ila}.
\begin{enumerate}
\item[(i)] There exists on $[a, \overline{x}_p]$, 
a unique root $\tilde{l}(a)$ such that $J(\tilde{l}(a);a) = 0$. 
In addition, this satisfies
\begin{equation}\label{cond_l_optP}
 \tilde{l}(a) \in (\underline{x}_p \vee a, \bar{x}_p) \qquad \textrm{and} \qquad h_p(\tilde{l}(a) ) < 0, \text{ where $h_p$ is defined in \eqref{def_h_p_c}.}
\end{equation}

\item[(ii)] 
We have $J(l;a) > 0$, for all $l < \tilde{l}(a)$, and $J(l;a) < 0$, for all $l > \tilde{l}(a)$. 

\item[(iii)] 
For $x \in \R$, 
we have 
$v_p(x;a,\tilde{l}(a)) = \max_{l \in [a, \overline{x}_p]} v_p(x;a,l)$     
which is further upgraded to 
$$v_p(x;a,\tilde{l}(a)) = \max_{l \in \R} v_p(x;a,l). $$
\end{enumerate}
\end{lemma}
\begin{proof}
We first prove parts (i) and (ii) together in the first two steps and part (iii) in the third step.

{\it Step 1.} 
Thanks to $a < \overline{x}_p$, \eqref{Gamma} and \eqref{about_x_upper}, we have 
\begin{align} \label{l_end_points}
J(a;a) = f_p(a) > 0 \quad \textrm{and} \quad J( \overline{x}_p; a) 
= - \lambda \Gamma (a; \overline{x}_p) < 0.
\end{align} This together with the continuity of $l \mapsto J(l; a)$ shows that there exists at least one $l 
\in (a, \overline{x}_p)$ such that $J(l;a) 
= 0$.

{\it Step 2.}  
By taking the partial derivative of \eqref{Gamma},  
for all $a < l$, which gives
$\frac \partial {\partial l} \Gamma(a;l)  
= f_p(l)  W^{(q+\lambda)}(l-a)$ 
and using \eqref{Z_Phi_der}, we obtain
\begin{align*} 
\begin{split}
\frac \partial {\partial l}J(l;a)
&= f_p'(l) \, Z^{(q+\lambda)}(l-a;\Phi(q)) + f_p(l) \, Z^{(q+\lambda)\prime}(l-a;\Phi(q)) 
- \lambda \frac \partial {\partial l} \Gamma(a;l) 
= h_p(l) Z^{(q+\lambda)}(l-a; \Phi(q) ).
\end{split}
\end{align*}
Using \eqref{x_p_and_h_p} and the facts that $Z^{(q+\lambda)}(\cdot; \Phi(q) )$ is uniformly positive and $h_p(\cdot)$ is continuous, we have
\begin{align} \label{l_and_x_p}
\frac \partial {\partial l} J(l;a)\left\{ \begin{array}{ll} 
> 0, & l <  \underline{x}_p,  \\ 
< 0, & l >  \underline{x}_p. \end{array} \right.
 \end{align}
Combining this (when $\underline{x}_p = -\infty$, $J(\cdot;a)$ is monotonically decreasing on $(-\infty, \overline{x}_p]$) with \eqref{l_end_points}, the solution in {\it step 1} is unique and we denote it by $\tilde{l}(a)$. 
We also obtain the claims in \eqref{cond_l_optP} -- thus completing part (i) -- as well as the claim in part (ii).

{\it Step 3.} 
Combining part (ii) with \eqref{sign_matches}, we conclude that $\tilde{l}(a)$ is indeed the maximiser over $[a, \overline{x}_p]$. Thus, Lemma \ref{remark_domain_to_focus}.(ii) completes the proof.
\end{proof}




For the threshold pair $(a,\tilde{l}(a))$, we get in light of condition $\mathbf{C}_l$ from \eqref{cond_lP} and the expression \eqref{v_p_other_form} of $v_p$, that 
\begin{align}
v_p(x;a,\tilde{l}(a)) =  J(\tilde{l}(a);x), \quad x \geq a, \label{v_p_l}
\end{align}
where in particular, recalling once again the condition $\mathbf{C}_l$, it is straightforward to confirm that $v_p(a;a,\tilde{l}(a)) =  J(\tilde{l}(a);a) = 0$. 
This expression will be useful both in the proof of the forthcoming result as well as later for strengthening the results in Section \ref{section_optimality_general}.

Similar to Lemma \ref{aincr}, we now present the continuity and monotonicity of the threshold $\tilde{l}(a)$ with respect to the arbitrarily chosen -- until this stage of analysis -- threshold 
$a \in (-\infty, \ol x_p)$, according to Lemma \ref{lem:exist_l}.
 
\begin{lemma}
\label{lincr}
The function $a \mapsto \tilde{l}(a)$, which is defined in Lemma \ref{lem:exist_l}.(i), is continuous and strictly increasing in $a$ on $(-\infty, \overline{x}_p)$. 
\end{lemma}
\begin{proof}
To show the continuity, we argue by contradiction, assuming that it fails to be continuous at some $a^\dagger < \overline{x}_p$, so that there exist two sequences $(a_n^-)$ and $(a_n^+)$ converging to $a^\dagger$ 
with $\lim_{n \to \infty}\tilde{l}(a_n^-) \neq \lim_{n \to \infty}\tilde{l}(a_n^+)$.
Then, noting that $(l,a) \mapsto J(l;a)$ is continuous, we have
\[
J(\lim_{n \to \infty} \tilde{l}(a_n^-); a^\dagger) = \lim_{n \to \infty}J(\tilde{l}(a_n^-); a_n^-) = 0 =  \lim_{n \to \infty}J(\tilde{l}(a_n^+); a_n^+) =  J(\lim_{n \to \infty} \tilde{l}(a_n^+); a^\dagger ).
\]
This contradicts with the uniqueness of the root  $\tilde{l}(a^\dagger)$ of $J(\cdot; a^\dagger ) = 0$ established in 
Lemma  \ref{lem:exist_l}.



%

To show the monotonicity, 
we again argue by contradiction, assuming that there is $a^\dagger \hspace{-1mm}\in \hspace{-1mm}(-\infty, \ol x_p)$ and sufficiently small $\delta > 0$ , so that $a^\dagger  + \delta < \ol x_p$ and $\tilde{l}(a^\dagger  + \delta) \leq \tilde{l}(a^\dagger )$. 
Since $J(\tilde{l}(a^\dagger ); a^\dagger ) = J(\tilde{l}(a^\dagger + \delta); a^\dagger +\delta) = 0$, we have
\[
 J(\tilde{l}(a^\dagger ); a^\dagger ) - J(\tilde{l}(a^\dagger ); a^\dagger +\delta) = J(\tilde{l}(a^\dagger + \delta); a^\dagger +\delta) - J(\tilde{l}(a^\dagger ); a^\dagger +\delta) \geq 0 
\]
where the last inequality holds by \eqref{l_and_x_p} (recall that $\tilde{l}(a) > \underline{x}_p$ for all $a \in [a^\dagger , a^\dagger  + \delta]$ by \eqref{cond_l_optP}) and by our assumption that $ \tilde{l}(a^\dagger  + \delta) \leq \tilde{l}(a^\dagger )$. Using once again the fact that $J(\tilde{l}(a^\dagger ); a^\dagger )=0$, we obtain
\[
J(\tilde{l}(a^\dagger ); a^\dagger +\delta) \leq 0.
\]
This contradicts with the fact that \eqref{v_p_l} is strictly positive 
given that $\tilde{l}(a^\dagger) < \overline{x}_p$.  
This completes the proof. 
\end{proof}

%
%


\subsection{Construction of Nash equilibrium.}

In the previous section, we established that, 
for any threshold $l$ chosen by player $P$ from an appropriate domain, player $C$ chooses a {\it unique best response $\tilde{a}(l)$} such that \eqref{cond_aC} holds (cf.\ Lemma \ref{lem:exist_a}), i.e. $I(\tilde{a}(l);l) = 0$. 
Moreover, 
for any threshold $a$ chosen by player $C$ from an appropriate domain, player $P$ chooses a {\it unique best response $\tilde{l}(a)$} such that \eqref{cond_lP} holds (cf.\ Lemma \ref{lem:exist_l}), i.e. $J(\tilde{l}(a);a) = 0$.  
The uniqueness of $\tilde{a}(\cdot)$ and $\tilde{l}(\cdot)$ and their continuity (cf.\ Lemmata \ref{aincr} and \ref{lincr}) will be used in the proofs of Propositions \ref{prop_equivalence_opt} and \ref{existNash}, which prove that there always exists a Nash equilibrium. 
Sample plots of the functions $a \mapsto I(a; \tilde{l}(a))$ and $l \mapsto J(l; \tilde{a}(l))$ are shown in Figure \ref{figure_J_I_2} in Section \ref{section_numerics}.

In this subsection, we aim at analysing a fixed point $(a^*,l^*)$ satisfying
\begin{align}
l^*=\tilde{l}(a^*) \quad \text{and} \quad a^*=\tilde{a}(l^*) \,, \label{l_a_saddle_pt}
\end{align}
which is equivalent to proving that the associated stopping times to these threshold strategies are the best responses to each other.
This condition can be shown to be equivalent to the equilibrium relation \eqref{nash_threshold}. 
 
\begin{proposition} \label{prop_equivalence_opt}
A pair of barriers $(a^*,l^*)$ satisfies \eqref{l_a_saddle_pt} if and only if $(a^*,l^*)$ satisfies
\eqref{nash_threshold}. 
\end{proposition}
\begin{proof}
We prove the sufficiency and necessity separately in the following two steps.

{\it Step 1.} Suppose $(a^*,l^*)$  satisfies \eqref{l_a_saddle_pt}. Then, by Lemmata \ref{lem:exist_a}.(II).(iii) and \ref{lem:exist_l}.(iii), $(a^*,l^*)$ solves also  \eqref{nash_threshold}.

{\it Step 2.} Suppose $(a^*,l^*)$  satisfies  \eqref{nash_threshold}. 
Since the partial derivatives of $v_c$ and $v_p$ with respect to $a$ and $l$, respectively, were shown to be continuous in Lemma  \ref{lemma_derivatives},  $(a^*,l^*)$ must satisfy the first-order conditions in \eqref{cond_aC} and \eqref{cond_lP}, which are equivalent to $J(l^*;a^*) = I(a^*; l^*) = 0$. 
Now, by Lemma \ref{lem:exist_a}.(II), the equality $I(a^*; l^*) = 0$ requires 
that $l^* \geq \underline{x}_c$ and by the uniqueness of the root $\tilde{a}(\cdot)$ we must have $a^* = \tilde{a}(l^*) \in [\underline{x}_c, \overline{x}_c \wedge l^*)$. 
Furthermore, given that $a^* < \overline{x}_c < \overline{x}_p$,  Lemma \ref{lem:exist_l} implies that the equality $J(l^*;a^*) = 0$ guarantees (again by the uniqueness) that $l^* = \tilde{l} (a^*)$. Hence $(a^*,l^*)$ must satisfy the relationships in \eqref{l_a_saddle_pt}. 
\end{proof}

We now show that the root $l^*$ of $J(\cdot; \tilde{a}(\cdot)) = 0$ exists and together with the corresponding best response $a^*$ form a Nash equilibrium in the sense of \eqref{l_a_saddle_pt} (or equivalently \eqref{nash_threshold}, in light of Proposition \ref{prop_equivalence_opt}).

\begin{proposition} \label{existNash}
Recall the definition \eqref{I_a_a} of $J$ and the construction of $\tilde{a}$ in Lemma \ref{lem:exist_a}.(II).  
\begin{enumerate}
\item[(i)] There exists a root $l^*$ to the equation $J(\cdot; \tilde{a}(\cdot)) = 0$, such that $l^* \in (\ul x_c, \, \ol x_p)$. 
\item[(ii)] For any root $l^*$ in (i) and $a^* :=\tilde{a}(l^*) \in (\ul x_c, \, \ol x_c \wedge l^*)$, the pair $(a^*, l^*)$ satisfies \eqref{l_a_saddle_pt}. 
\end{enumerate}
Hence, there always exists a pair $(a^*, l^*)$ that forms a Nash equilibrium of threshold strategies in the sense of \eqref{nash_threshold}.
\end{proposition}

\begin{proof}
We prove the two parts separately. 

{\it Proof of part} (i). 
By Lemma \ref{lem:exist_a}.(II).(i) 
and \eqref{I_a_a} and because $\underline{x}_c < \overline{x}_p$ due to \eqref{x_bar_inequality} and \eqref{x_p_and_h_p}, we get
\begin{equation} \label{Iulxc}
J(\ul x_c; \tilde{a}(\ul x_c)) = J(\ul x_c; \ul x_c) = f_p(\ul x_c) > 0.
\end{equation} 
Furthermore, since $f_p(\ol x_p)=0$ and $\ol x_p \geq \ol x_c > \tilde{a}(\ol x_p)$  by \eqref{x_bar_inequality} and Lemma \ref{lem:exist_a}.(II).(i), \eqref{Ila} also yields that 
\begin{align} \label{Iolxp}
\begin{split}
J(\ol x_p; \tilde{a}(\ol x_p)) = - \lambda \Gamma(\tilde{a}(\overline{x}_p); \overline{x}_p)
= - \lambda \int_{ \tilde{a}(\ol x_p)}^{\ol x_p}f_p(u) \, W^{(q+\lambda)}(u- \tilde{a}(\ol x_p)) \du < 0 \,.
\end{split}
\end{align}
Using \eqref{Iulxc}--\eqref{Iolxp} together with the fact that the function $l \mapsto J(l; \tilde{a}(l))$ is continuous on $(\ul x_c, \,\ol x_p)$ (due to the continuity of  $l \mapsto \tilde{a}(l)$ from Lemma \ref{aincr}), we complete the proof of this part .

{\it Proof of part} (ii). 
Using $l^* > \ul x_c$ from part (i) to define $a^* := \tilde{a}(l^*)$, we conclude from Lemma \ref{lem:exist_a} that $a^* \in (\ul x_c, \, \ol x_c \wedge l^*)$.
Thanks to the uniqueness of $\tilde l(\cdot)$ 
from Lemma \ref{lem:exist_l}, we have  $l^* = \tilde{l}(a^*)$ if and  only if $J(l^*; a^*) = 0$. This is indeed true, due to the ways $l^*$ was derived and $a^* := \tilde{a}(l^*)$ was defined, 
which imply $J(l^*; \tilde{a}(l^*)) = 0$. 
%
%
\end{proof}

Collecting all aforementioned results together, we are in position to provide sufficient conditions for the uniqueness of the Nash equilibrium in threshold strategies. 
\begin{proposition} \label{unique}
The Nash equilibrium in the sense of \eqref{nash_threshold} constructed in Proposition \ref{existNash} (i.e.\ the pair of $(a^*,l^*)$ satisfying the two equalities in  \eqref{nash_threshold}), is unique, if and only if the equation 
$J(\cdot; \tilde{a}(\cdot)) = 0$ admits a unique solution in $(\ul x_c, \, \ol x_p)$. 
\end{proposition}
\begin{proof} 
We prove the desired result in the following two steps.

{\it Step 1.}
Combining the equivalence of \eqref{nash_threshold} and  \eqref{l_a_saddle_pt} proved in Proposition \ref{prop_equivalence_opt} with the Nash equilibrium existence result obtained in Proposition \ref{existNash}, we conclude that $J(l^*; \tilde{a}(l^*)) = 0$ (with $a^* := \tilde{a}(l^*))$ is a necessary condition for $(a^*,l^*)$ to be a Nash equilibrium amongst threshold-type strategies. 
Hence, the ones obtained in Proposition \ref{existNash} are the only Nash equilibria. 

{\it Step 2.}
In light of Step 1, we see that the uniqueness of Nash equilibrium in the sense of \eqref{nash_threshold}, is equivalent to the uniqueness of the solution $l^* \in (\ul x_c, \, \ol x_p)$ to the equation $J(\cdot; \tilde{a}(\cdot)) = 0$ (cf.\ Proposition \ref{existNash}.(i).  
\end{proof}

It is also apparent from Proposition \ref{unique}, that if the model and problem formulation are such that there exist multiple solutions to the equation 
$J(\cdot; \tilde{a}(\cdot)) = 0$, i.e. multiple $l^*$ satisfying Proposition \ref{existNash}.(i), we can construct multiple associated Nash equilibria $(a^*,l^*)$ as in Proposition \ref{existNash}.(ii).
In such a scenario, our subsequent analysis provides a way to construct the unique Nash equilibrium that is Pareto-superior to any other Nash equilibrium.
That is, by assuming that both players are rational and intelligent enough, we can discard other (Pareto-dominated) equilibria, because all agents are strictly better-off if they switch to this unique Pareto-superior equilibrium pair of strategies.
This can be seen as an alternative way of achieving a version of ``uniqueness'' even in this case of potential multiple Nash equilibria.

We show in the following proposition, using also the monotonicity of $\tilde{a}(\cdot)$ and $\tilde{l}(\cdot)$ (cf.\ Lemmata \ref{aincr} and \ref{lincr}), that by choosing the smallest (threshold) root in Proposition \ref{existNash}.(i), we can construct  the unique Nash equilibrium that is Pareto-superior to any other Nash equilibrium.

\begin{proposition} \label{prop_Pareto}
Define the thresholds 
\begin{equation} \label{3l*}
l^*_{min} := \min \{ l \in (\ul x_c, \, \ol x_p) \,:\, J(l; \tilde{a}(l)) = 0 \}
\qquad \text{and} \qquad
a^*_{min} := \tilde{a}(l^*_{min}).
\end{equation}
The pair $(a^*_{min},l^*_{min})$ is the unique Nash equilibrium that is Pareto-superior to any other Nash equilibrium $(a^*,l^*)$ satisfying \eqref{l_a_saddle_pt} (or equivalently \eqref{nash_threshold}). In other words,
\begin{align*}
v_i(x;a^*,l^*) \leq v_i(x;a^*_{min},l^*_{min}) , \qquad \text{for both  $i \in\{c,p\}$ and all $x\in\R$}.
\end{align*}
In particular, if $x > a^*_{min}$, then the above inequality is strict.
\end{proposition}
\begin{proof}
We prove the desired claims in the following four steps.

{\it Step 1.} 
Firstly, we note that $l^*_{min}$ is the minimum root
in Proposition \ref{existNash}.(i), thus the pair $(a^*_{min},l^*_{min})$ is well-defined.
Then, it follows by Proposition \ref{existNash}.(ii) that $(a^*_{min},l^*_{min})$ is a Nash equilibrium in the sense of \eqref{l_a_saddle_pt} and \eqref{nash_threshold}.

{\it Step 2.} 
Suppose there exists $(a^*_o,l^*_o)$ satisfying  \eqref{l_a_saddle_pt}, i.e.\ $l^*_o = \tilde{l}(a^*_o)$ and $a^*_o = \tilde{a}(l^*_o)$), different from $(a^*_{min},l^*_{min})$.  
This implies that $J(l^*_o; \tilde{a}(l^*_o)) = 0$, 
while $a^*_o = \tilde{a}(l^*_o)$ implies 
$I(a^*_o; l^*_o) = 0$, hence we know from Lemma \ref{lem:exist_a}.(II).(i) that $l_o^* > \underline{x}_c$. 
Due to Proposition \ref{prop_equivalence_opt} and the equivalence of 
\eqref{nash_threshold} and \eqref{nash_threshold_new}, we must also have $l_o^* < \overline{x}_p$. 
Combining these with the definition of $l^*_{min}$ in \eqref{3l*} 
and the fact that $\tilde{a}(\cdot)$ is unique due to Lemma \ref{lem:exist_a}.(II), we conclude that   $l^*_{min} < l^*_o$.
%
%
Using this together with the monotonicity of the function $\tilde{a}(\cdot)$ in Lemma \ref{aincr}, we get that
$$ a_o^* = \tilde{a}(l_o^*) > \tilde{a}(l^*_{min}) = a^*_{min} \,.$$

{\it Step 3.}
It is straightforward to see that $l^*_{min} < l_o^*$ (thus $T^-_{l^*_o} \leq T^-_{l^*_{min}}$ a.s.) implies that 
$\{\tau_{a^*_o}^- <T^-_{l_o^*}\} \subseteq \{\tau_{a^*_o}^- < T^-_{l^*_{min}}\}$ a.s.,
which implies for $x\in\R$ (recalling that $a^*_o < \bar{x}_c$ from Lemma \ref{lem:exist_a}.(II), thus $f_c(X_{\tau_{a^*_o}^-}) > 0$ a.s. by \eqref{about_x_upper}), that
\begin{align*}
v_c(x;a^*_o,l^*_o) = \E_x \Big[  e^{-q \tau_{a^*_o}^-}f_c(X_{\tau_{a^*_o}^-}) 1_{\{\tau_{a^*_o}^-<T_{l^*_o}^-\}} \Big] \leq  \E_x \Big[  e^{-q \tau_{a^*_o}^-}f_c(X_{\tau_{a^*_o}^-}) 1_{\{\tau_{a^*_o}^-<T_{l^*_{min}}^-\}} \Big] \leq v_c(x;a^*_{min},l^*_{min}),
\end{align*}
where the last inequality holds due to  $a^*_{min}$ being the best response to $l^*_{min}$ and is strict when $x > a^*_{min}$.
%
%
%
%

{\it Step 4.} Similarly, $a^*_{min} < a_o^*$ (thus $\tau_{a^*_o}^- \leq  \tau_{a^*_{min}}^-$ a.s.) implies that $\{T_{l_o^*}^- < \tau_{a_o^*}^- \} \subseteq \{T^-_{l_o^*}< \tau_{a^*_{min}}^- \}$ a.s., 
hence for $x\in\R$ (recalling that $l^*_o < \bar{x}_p$ from Lemma \ref{lem:exist_l}.(i), thus $f_p(X_{T^-_{l_o^*}}) > 0$ a.s. by \eqref{about_x_upper}), we obtain 
$$
v_p(x;a^*_o,l^*_o) =  \E_x \Big[e^{-q T_{l_o^*}^-} f_p(X_{T^-_{l_o^*}}) 1_{\{T_{l_o^*}^- < \tau_{a_o^*}^- \}}\Big] \leq
\E_x \Big[e^{-q T_{l_o^*}^-} f_p(X_{T^-_{l_o^*}}) 1_{\{T_{l_o^*}^- < \tau_{a^*_{min}}^- \}}\Big] \leq  v_p(x;a^*_{min},l^*_{min}),
$$
where the last inequality holds due to  $l^*_{min}$  being the best response to $a^*_{min}$ and is strict when $x > a^*_{min}$.
%
%
%
\end{proof}

\subsection{Properties of $v_c(\cdot;a^*,l^*)$ and $v_p(\cdot;a^*,l^*)$} 
Before concluding this section, we obtain some properties of $v_c(\cdot;a^*,l^*)$ and $v_p(\cdot;a^*,l^*)$ for $(a^*, l^*)$ satisfying \eqref{l_a_saddle_pt}.

Observe that, by \eqref{cond_lP}, \eqref{Ila} and due to $J(l^*;a^*) = 0$, we obtain 
\begin{align*} 
\begin{split}
\frac {\lambda \Gamma(a^*;l^*) } {Z^{(q+\lambda)}(l^*-a^*;\Phi(q)) }
&= f_p(l^*) = v_p(l^*;a^*,l^*).
\end{split}
\end{align*}
Substituting this in \eqref{vf_1P} for $v_p(x;a^*,l^*)$, we get  
\begin{align}
v_p(x;a^*,l^*)=  Z^{(q+\lambda)}(l^*-x;\Phi(q))v_p(l^*;a^*,l^*)- \lambda  \Gamma(x;l^*), \quad \text{for } x \geq a^*. \label{v_p_opt}
 \end{align}
This alternative expression will be used both for proving the next result (see Appendix \ref{proof_smooth}) as well as in Section \ref{section_optimality_general}.

\begin{proposition}[Smoothness \& Convexity] \label{smooth}
Recall $v_c(\cdot;a^*,l^*)$, $v_p(\cdot;a^*,l^*)$ defined in \eqref{vf_0PC} and satisfying \eqref{nash_threshold}. \\
 {\rm (I)} Regarding the function $v_c(\cdot;a^*,l^*)$, we have the following:
\begin{itemize}
	\item[(i)] $v_c(\cdot;a^*,l^*)$   is continuous on $\R$ and $C^2$ (resp.,\ $C^1$) on 
	$(a^*, \infty) \backslash \{ l^*\}$ when $X$ is of unbounded (resp.,\ bounded) variation.
	\item[(ii)] $v_c(\cdot;a^*,l^*)$ is continuously differentiable at $a^*$.
	\item[(iii)] $v_c(\cdot;a^*,l^*)$ is continuously differentiable at $l^*$, only when $X$ is of unbounded variation.
	\item[(iv)] $v_c(\cdot;a^*,l^*)$ is decreasing and convex on $(a^*, \infty)$.
\end{itemize}
{\rm (II)} Regarding the function $v_p(\cdot;a^*,l^*)$, we have the following:
\begin{itemize}
\item[(i)] $v_p(\cdot;a^*,l^*)$ is continuous on $\R$ and twice continuously differentiable on $\mathbb{R}\backslash\{a^*, l^*\}$. 
	\item[(ii)] $v_p(\cdot;a^*,l^*)$ is continuously differentiable at $l^*$.
	\item[(iii)] $v_p(\cdot;a^*,l^*)$ is twice continuously differentiable at $l^*$, only when $X$ is of unbounded variation.
\end{itemize}
\end{proposition}

\section{Optimality over all stopping times} \label{section_optimality_general}

In Proposition \ref{existNash}, we showed the existence of the solutions $(a^*, l^*)$ to \eqref{cond_aC} and \eqref{cond_lP} and that $(a^*, l^*)$ is a Nash equilibrium in the sense of \eqref{nash_threshold} where strategies are restricted to be of threshold-type. In Proposition \ref{prop_Pareto}, we showed that by choosing $(a^*_{min},l^*_{min})$ as in  \eqref{3l*}, we can construct a Nash equilibrium that is Pareto-superior to other $(a^*, l^*)$ satisfying  \eqref{l_a_saddle_pt} or equivalently \eqref{nash_threshold} (cf.\ Proposition \ref{prop_equivalence_opt}).


In this section, we strengthen the results by considering a larger set of admissible strategies.
%
%
%
%
%
For $x \in \R$, we aim at showing that the pair of strategies $(\tau_{a^*}^-,T_{l^*}^-)$ is a Nash equilibrium as in Section \ref{problem}, when the strategy sets of both players are unrestricted (most general ones possible); this is formally stated as 
\begin{align} \label{NEsys} 
\begin{cases} 
V_c( \tau_{a^*}^-,T_{l^*}^-; x) \geq V_c( \tau,T_{l^*}^-; x),  \quad \forall\; \tau  \in \mathcal{T}_c, \\
V_p(\tau_{a^*}^-, T_{l^*}^-; x) \geq V_p( \tau_{a^*}^-,\sigma; x), \quad \forall\; \sigma \in \mathcal{T}_p.
 \end{cases} 
\end{align}
In view of the definitions of the values $V_c$ and $V_p$ in Section \ref{problem} and the definitions \eqref{vf_0PC}, we see that 
$V_c(\tau_{a^*}^-,T_{l^*}^-; x) \equiv v_c(x; a^*, l^*)$ and 
$V_p(\tau_{a^*}^-,T_{l^*}^-; x) \equiv v_p(x; a^*, l^*)$.
Hence, proving the coupled system of inequalities in \eqref{NEsys} is equivalent to proving the coupled system of equalities, for all $x \in \R$, given by
\begin{align} \label{opt_for_CP}
\begin{cases}
v_c(x; a^*, l^*) = \sup_{\tau \in  \mathcal{T}_c}  V_c(\tau, T_{l^*}^-; x), \\
v_p(x; a^*, l^*) = \sup_{\sigma \in  \mathcal{T}_p} V_p( \tau_{a^*}^-,\sigma; x).
\end{cases} 
\end{align}
Notice that the optimal stopping problems on the right-hand sides of \eqref{opt_for_CP} form a coupled system, where the coupling comes from their random time horizons $T_{l^*}^-$ and $\tau_{a^*}^-$ (recall their definitions in \eqref{hittingtimes}), which are controlled by their opponents. 
To be more precise, the threshold $l^*$ involved in the random time horizon $T_{l^*}^-$ for player $C$, is chosen by player $P$ according to $l^* = \tilde{l}(a^*)$ (cf.\ \eqref{l_a_saddle_pt}) while aiming at solving their own optimal stopping problem with random time horizon $\tau_{a^*}^-$, and vice versa, thus creating this closed-loop coupling.  
This dependence in the optimal stopping problems on the right-hand sides of \eqref{opt_for_CP} can be formally expressed as
\begin{align} \label{opt_for_C_}
\sup_{\tau \in  \mathcal{T}_c}  V_c(\tau, T_{l^*}^-; x) 
&= \sup_{\tau \in  \mathcal{T}_c}  \mathbb{E}_x \Big[ e^{-q \tau} f_c(X_{\tau}) 1_{\{ \tau < T_{l^*}^- \}} \Big], 
\;\qquad \text{where } l^* = \tilde{l}(a^*), 
\\ \label{opt_for_P_}
\sup_{\sigma \in  \mathcal{T}_p} V_p( \tau_{a^*}^-,\sigma; x)
&= \sup_{\sigma \in  \mathcal{T}_p} \mathbb{E}_x \Big[ e^{-q \sigma} f_p(X_{\sigma}) 1_{\{ \sigma < \tau_{a^*}^-\}} \Big], 
\qquad \text{where } a^* = \tilde{a}(l^*),
\end{align}
with the mappings $\tilde{a}$ and $\tilde{l}$ given by Lemma \ref{lem:exist_a}.(II) and Lemma \ref{lem:exist_l}, respectively.
The coupled system of equalities in \eqref{opt_for_CP} can be therefore proved by showing that the solution to the system \eqref{opt_for_C_}--\eqref{opt_for_P_} is given by the pair of value functions $v_c(x; a^*, l^*)$ and $v_p(x; a^*, l^*)$, respectively, when the pair $(a^*, l^*)$ is obtained in Proposition \ref{existNash}. 
 
A standard methodology employed in optimal stopping theory is the use of variational inequalities to verify the optimality of candidate strategies and value functions. 
To this end, we define the infinitesimal generator $\mathcal{L}$ acting on sufficiently smooth functions $w(\cdot)$ as follows:
\begin{align*}
\mathcal{L} w(x) := \frac{\nu^2}{2} w''(x) - \gamma w'(x) + \int_{(0, \infty)}[w(x+z) - w(x) - w'(x) z\mathbf{1}_{\{z<1\}}]\Pi(\mathrm{d}z).
\end{align*}

\begin{remark}
A standard analytical verification theorem for Player $C$, i.e. the first part of \eqref{opt_for_CP}, would be given (proof is omitted) by verifying the following conditions, for $(a^*, l^*)$ obtained in Proposition \ref{existNash}:
\begin{itemize}
\item[(i)] $(\mathcal{L} - q) v_c(x;a^*,l^*)=0$, \,for $x > l^*$ \,,
\item[(ii)] $(\mathcal{L} - q) v_c(x;a^*, l^*) - \lambda v_c(x;a^*, l^*) = 0$, \,for $x \in (a^*, l^*)$ \,,
\item[(iii)] 
$(\mathcal{L} - q) v_c(x;a^*, l^*) - \lambda v_c(x;a^*, l^*) \leq 0$, \,for $x<a^*$ \,,
\item[(iv)] $v_c(x;a^*, l^*) \geq f_c(x)$, \,for $x > a^*$ \,,
\item[(v)] $v_c(x;a^*, l^*) = f_c(x)$, \,for $x \leq a^*$ \,.
\end{itemize}
However, in the general setup of reward functions $f_c(\cdot)$ (see Assumption \ref{Ass}) and L\'evy models considered in this paper, this  method is non-feasible. 
To be more precise, condition (iii)\footnote{Such a condition was part of Assumption 2.6 in \cite{DeAFM} and was used to prove the existence of Nash equilibrium in threshold strategies for diffusion models (i.e. with jump measure $\Pi \equiv 0$) in a game of symmetric continuous exercise opportunities (i.e. two players of type $C$).} required for verification, is 
equivalent to
\begin{multline*}
\frac{\nu^2}{2} f_c''(x) - \gamma f_c'(x) + \int_{(0, \infty)} \hspace{-0pt}\big\{ v_c(x+z;a^*,l^*) - f_c(x) - f_c'(x) z \mathbf{1}_{\{z<1\}} \big\} \Pi(\mathrm{d}z) - (q+\lambda) f_c(x) \leq 0, \quad \forall x < a^*. 
\end{multline*}
Due to the presence of jumps in our model, verifying the above condition is in general non-feasible -- additional assumptions are required on the model parameters $\gamma$, $\nu$, $q$, the frequency of periodic exercise opportunities $\lambda$, the jump measure $\Pi$ and reward functions $f_c$, which are also hard to verify themselves. 
\end{remark}

In order to maintain the original general setting of our paper, without relying on further assumptions, we propose an amalgamated methodology that will involve: 
(a) our already obtained results on Nash equilibria in threshold strategies (Propositions \ref{unique} and \ref{prop_Pareto}); 
(b) our results on the regularity of candidate value functions (Proposition \ref{smooth}); 
(c) a reformulation of problem \eqref{opt_for_C_} -- involved in the first part of the system \eqref{opt_for_CP} -- to one with stochastic path-dependent discounting; 
(d) the introduction of an {average problem approach} (developed in \cite{LZ}, \cite{RZ1} and \cite{Surya} for optimal stopping problems) to non-zero-sum games of optimal stopping with asymmetric exercise opportunities; 
(e) a combination of the above with the variational inequalities for problem \eqref{opt_for_P_} -- involved in the second part of the system \eqref{opt_for_CP} -- which provide sufficient conditions for optimality in the next result\footnote{A similar result was obtained by \cite{DW2002} in an infinite horizon optimal stopping problem for a continuous model -- instead, we deal with a random time horizon (controlled by the opponent) and L\'evy models with positive jumps -- we thus adapt the result accordingly.}. Its proof is deferred to Appendix \ref{verification_split_proof}. 

\begin{lemma}[Verification lemma for \eqref{opt_for_P_}] \label{verification_split}
{Recall that $v_p$ satisfies all properties proved in Proposition \ref{smooth}.(II) and additionally} suppose that
\begin{itemize}
\item[(i)] $(\mathcal{L} - q) v_p(x; a^*, l^*) =0$, \,for $x\geq l^*$ \,,
\item[(ii)] $(\mathcal{L} - q) v_p(x; a^*, l^*) - \lambda \, (v_p(x; a^*, l^*) - f_p(x)) = 0$, \,for $x \in (a^*,l^*]$ \,,
\item[(iii)] $v_p(x; a^*, l^*) \geq f_p(x)$, \,for $x \geq l^*$ \,,
\item[(iv)] $v_p(x; a^*, l^*) \leq f_p(x)$, \,for $x\in [a^*,l^*]$ \,,
\item[(v)] $v_p(x; a^*, l^*) = 0$, \,for $x \leq a^*$ \,.
\end{itemize}
Then, $v_p(x; a^*, l^*)$ is the value function of \eqref{opt_for_P_}. 
\end{lemma}

The main result of the paper is presented below. 

\begin{theorem} With $(a^*, l^*)$ obtained in Proposition \ref{existNash},
we have for all $x \in \R$, that the coupled system of equalities in \eqref{opt_for_CP} holds true, which is equivalent to \eqref{NEsys}, 
hence $(\tau_{a^*}^-,T_{l^*}^-)$ is a Nash equilibrium in the sense of \eqref{Nasheq}.
\end{theorem}
\begin{proof}
Recall firstly that the coupled system of equalities in \eqref{opt_for_CP} can be proved by showing that the solutions to \eqref{opt_for_C_} and \eqref{opt_for_P_} are given by $v_c(x; a^*, l^*)$ and $v_p(x; a^*, l^*)$, respectively, when the pair $(a^*, l^*)$ is obtained in Proposition \ref{existNash}. 
We thus split the proof in the following two steps. 

\vspace{3pt}
{\it Step I. Proof that $v_c(x; a^*, l^*)$ is the solution to \eqref{opt_for_C_}.}
We begin by denoting the problem \eqref{opt_for_C_} by $u(x)$.
We thus aim at proving that 
\begin{equation} \label{proveOS1}
u(x) = 
v_c(x; a^*,l^*) .
\end{equation}
Using the definition \eqref{hittingtimes} of $T^-_{l^*}$ and the independence of the Poisson process $N$ and L\'evy process $X$, we can rewrite the current optimal stopping problem with random time-horizon $T^-_{l^*}$, in the form of
\begin{equation} \label{auxOS2}
u(x) 
= \sup_{\tau\in\mathcal{T}_c}  \E_x \Big[  e^{- A^{X}_\tau} f_c(X_{\tau}) 1_{\{\tau<\infty\}} \Big] ,
\end{equation}
where the latter is a perpetual optimal stopping problem with stochastic discounting given by the occupation time 
$$
A^{X}_t := q t + \lambda \int_{0}^{t} 1_{\{X_u < l^*\}} \du ,
\quad \forall \; t \geq 0 ,
$$
(see \cite{LZ} and \cite{RZ1} for other optimal stopping problems with occupation time discounting under one-dimensional L\'evy models, and \cite{RZ2} under two-dimensional ones).
It is easy to see that $A^{X}_{\cdot}$ is a continuous additive functional.

In order to use the results of \cite{LZ}, where they consider the spectrally negative case,  we define
the process $Y_t := -X_t$, for all $t\geq 0$, which is a spectrally negative L\'evy process starting from $Y_0 =-x$.  
We then define 
\begin{equation} \label{f} 
\widehat{f}_c(y) := f_c(-y) \quad \text{for all $y\in\R$},
\end{equation}
and we observe from Assumption \ref{Ass}$.(ii)$ that $\widehat{f}_c(\cdot)$ is a strictly increasing, continuously differentiable and concave function on $\R$, such that $\widehat{f}_c(y) > 0$ if and only if  $y > -\overline{x}_c$. 
Hence, the optimal stopping problem \eqref{auxOS2} can be further rewritten in terms of the process $Y$, with  $\widehat{\E}_y := \E_{-y}$ and $x = -y$, taking the form
\begin{align} \label{auxOS3}
u(x) 
&= \sup_{\tau\in\mathcal{T}_c}  \widehat{\E}_{y} \Big[ e^{- A^{-Y}_\tau} f_c(-Y_{\tau}) 1_{\{\tau<\infty\}} \Big] 
= \sup_{\tau\in\mathcal{T}_c}  \widehat{\E}_y \Big[ e^{- \widehat A^{Y}_\tau} \widehat{f}_c(Y_{\tau}) 1_{\{\tau<\infty\}} \Big] =: \widehat u(y) \,,
\end{align}
where we notice that 
$$
A^{-Y}_t = \widehat A^{Y}_t := q t + \lambda \int_{0}^{t} 1_{\{Y_u > -l^*\}} \du ,
\quad \forall \; t \geq 0.
$$
Then, we define the left-inverse $\zeta$ of $\widehat A^Y$ at an independent exponential time $\mathbf{e}$ with unit mean by 
$$
\zeta\equiv (\widehat{A}^Y)^{-1}(\mathbf{e}):=\inf\{t>0: \widehat{A}_t^Y>\mathbf{e}\}
$$
and the running maximum process  of $Y$ by $\overline{Y}_t := \sup_{0\leq u\leq t} Y_u$. 
Using similar arguments to \cite[Section 4.1]{RZ1}, we obtain for $y \leq -a$, that 
\begin{align*}
\widehat{\p}_y( \overline{Y}_{\zeta}>-a)
=\widehat{\E}_y\Big[\exp(- \widehat A^Y_{\widehat \tau_{-a}^+})\, 1_{\{\widehat \tau_{-a}^+<\infty\}}\Big] 
&= \E_{-y}\Big[\exp(- A^X_{\tau_{a}^-})\, 1_{\{\tau_a^-<\infty\}}\Big] \\
&= \E_{-y}\Big[  e^{-q \tau_{a}^-} 1_{\{\tau_{a}^-<T_{l^*}^-\}} \Big]
= \dfrac{Z^{(q+\lambda)}(l^*+y;\Phi(q))}{Z^{(q+\lambda)}(l^*-a;\Phi(q))}\,,
\end{align*}
where the second equality follows from the definition of $Y$, which yields
\begin{equation} \label{taus}
\widehat \tau_{-a}^+:=\inf \{ t > 0: Y_t > -a\} = \inf \{ t > 0: X_t < a\} \equiv \tau_a^- \,,
\end{equation}
while the last equality follows from \eqref{vf_1C} (for $f_c(\cdot)\equiv 1$) in Proposition \ref{limit_cont_SP}.  
Therefore, using \eqref{Z_Phi_der}, we can define the ``hazard rate'' function $\Lambda(\cdot)$ on $\R$ such that
\begin{align} \label{Lambda}
\Lambda(z) := - \frac{1}{\widehat{\p}_y(\overline{Y}_\zeta>z)} \, 
\frac{\partial}{\partial z} \left( \widehat{\p}_y(\overline{Y}_\zeta > z) \right)
=  \Phi(q) + \lambda \, \frac{W^{(q+\lambda)}(l^*+z)}{Z^{(q+\lambda)}(l^*+z;\Phi(q))},
\end{align}
which holds for all $y \leq z$.
It is clear that $\int_x^\infty \Lambda(z) \diff z = \infty$ and hence satisfies \cite[Assumption 2.2]{LZ}.
Now define the function $\widehat{h}(\cdot)$ on $\R$, which is given by (cf.\  \eqref{f})
\begin{equation} \label{h}
\widehat{h}(y) := \widehat{f}_c(y) - \frac{\widehat{f}_c'(y)}{\Lambda(y)} 
= f_c(-y) + \frac{f_c'(-y)}{\Lambda(y)}. 
\end{equation}
Due to the definition \eqref{Jal} of $I(\cdot;l^*)$ together with \eqref{vCl} and the expression \eqref{Lambda} of $\Lambda$, we further have 
\begin{equation*}
\widehat{h}(y)
= (\Lambda(y))^{-1} \bigg( f_c(-y) \Big( \Phi(q)  + \lambda \frac { W^{(q+\lambda)}(l^*+y)} {Z^{(q+\lambda)}(l^*+y;\Phi(q))} \Big) + f_c'(-y) \bigg)
=  \, \frac {I(-y;l^*)  }  {\Lambda(y)}
\end{equation*}
which implies via the positivity of $\Lambda(\cdot)$ on $\R$ and the results in Lemma \ref{lem:exist_a}.(II).$(ii)$ for $a^*=\tilde a(l^*)$ that
$$
\widehat{h}(y) \begin{cases}
> 0, &\text{for } -y < a^* \Leftrightarrow y > -a^*, \\
\leq 0, &\text{for } -y \geq a^* \Leftrightarrow y \leq -a^*.
\end{cases}
$$
Finally, by taking the first derivative in \eqref{h}, we obtain by straightforward calculations that 
$$
\widehat{h}'(y) = - f_c'(-y) - \frac{f_c''(-y)}{\Lambda(y)} - \frac{f_c'(-y) \Lambda'(y)}{\Lambda^2(y)} > 0 , \quad \text{for a.e. } y\in\R. 
$$
The positivity is
due to Assumption \ref{Ass}$.(ii)$  and the fact that the combination of the  definition \eqref{Lambda} of $\Lambda(\cdot)$ together with Lemma \ref{W/Z} implies that 
\begin{equation*} 
y \mapsto \Lambda(y) = \Phi(q) + \lambda \, \frac{W^{(q+\lambda)}(l^*+y)}{Z^{(q+\lambda)}(l^*+y;\Phi(q))} 
\quad \text{is increasing on } \R.
\end{equation*} 
This shows that \cite[Assumption 2.3]{LZ} holds true (with $x^\star$ replaced with $-a^*$).

Now from \cite[Theorem 2.2]{LZ} (see also \cite[Section 4.1--4.2]{RZ1} for a similar result), which states under \cite[Assumptions 2.2 and 2.3]{LZ} that the root of $\widehat{h}(\cdot) = 0$ (in our case $-a^*$) gives the optimal strategy. In other words, 
the optimal stopping time for $\widehat u$ in \eqref{auxOS3} is given by $\widehat \tau_{-a^*}^+$. 
This yields in view of \eqref{taus} that the optimal stopping time for $u$ is given by $\tau_{a^*}^-$ and consequently \eqref{proveOS1} holds true, which completes the proof of this step.

\vspace{3pt}
{\it Step II. Proof that $v_p(x; a^*, l^*)$ is the solution to \eqref{opt_for_P_}.}
By the smoothness obtained in Proposition \ref{smooth}.(II), and because $v_p(x; a^*, l^*)$ is bounded in $x$ by $\max_{a^* \leq y \leq l^*} f_p(y) = f_p(a^*) \leq f_p(\ul x_c)<\infty$, due to the monotonicity of $f_p$ in Assumption \ref{Ass}.(ii), the admissible interval for $a^*=\tilde a(l^*)$ in Lemma \ref{lem:exist_a}.(II).(i), and the finiteness of $\ul x_c$ in Lemma \ref{lemma_x_c_finite}, we conclude that $\mathcal{L}  v_p(x; a^*, l^*)$ is defined for all $x \in \R \backslash \{ a^*\}$.

%

The remainder of the proof thus requires the verification of the five conditions assumed in Lemma \ref{verification_split} to eventually conclude the desired optimality.

{\it Part} (i).
This follows from the expression \eqref{vf_simpleP} of $v_p$ and the direct computation of  
\begin{align*}
(\mathcal{L} - q) e^{\Phi(q)(l^*-x)} = 0. 
\end{align*} 

{\it Part} (ii). By  (4.20) in \cite{PY} and Lemma 4.5 in \cite{EY}, respectively, we get 
\begin{align*}
&\big(\mathcal{L} - (q+\lambda)\big) Z^{(q+\lambda)}(l^*-x;\Phi(q)) = 0 
\qquad \text{and} \qquad
\big(\mathcal{L} - (q+\lambda)\big) \Gamma(x; l^*) 
= f_p(x). \nonumber
\end{align*}
Applying these equations to the expression \eqref{vf_1P} of $v_p$, we obtain the desired  
$\big(\mathcal{L} - (q+\lambda)\big) v_p(x; a^*, l^*) = - \lambda f_p(x)$.

{\it Parts} (iii) {\it and} (iv). 
Using the expression \eqref{v_p_opt} of $v_p$, we define
\begin{align} \label{RP}
R_p(x) := v_p(x;a^*,l^*) - f_p(x) 
= Z^{(q+\lambda)}(l^*-x;\Phi(q)) \, f_p(l^*) - \lambda
\Gamma(x; l^*) - f_p(x), \quad x \geq a^*.
\end{align}
In particular, we have by \eqref{vf_1P}
$$R_p(a^*) = v_p(a^*;a^*,l^*) - f_p(a^*) = - f_p(a^*) < 0 \,,$$
and by condition $\mathbf{C}_l$ in \eqref{cond_lP}, we have 
\begin{equation} \label{RPl}
R_p(l^*) = 0 \,.
\end{equation} 

First, suppose that $x \geq l^*$ aiming for the proof of (iii). 
By \eqref{vf_simpleP}, we get the simplified expression
\begin{align}\label{fun_hP}
R_p(x) 
= v_p(x;a^*,l^*) - f_p(x) 
= e^{\Phi(q)(l^*-x)} \, v_p(l^*;a^*,l^*) - f_p(x), \quad x \geq l^*.
\end{align}
By calculating its first and second derivatives, we get
\begin{align*}
R_p'(x) &= - \Phi(q) \, e^{\Phi(q)(l^*-x)} \, v_p(l^*;a^*,l^*) - f'_p(x) \,, \\
R_p''(x+) &= \Phi^2(q) \, e^{\Phi(q)(l^*-x)} v_p(l^*;a^*,l^*) - f''_p(x+) > 0 \,,
\end{align*}
where the last inequality follows from Assumption \ref{Ass}.(ii) and because $v_p(l^*;a^*,l^*) > 0$. 
We can thus immediately see that $R_p(\cdot)$ is convex on $(l^*, \infty)$, 
while the inequality \eqref{cond_l_optP} for the equilibrium threshold pair $(l,a)=(l^*,a^*)$ and the definition \eqref{def_h_p_c} of $h_p$, imply that
\begin{equation*}
R_p'(l^*) = - \Phi(q) \, v_p(l^*;a^*,l^*) - f'_p(l^*) = - \Phi(q) \, f_p(l^*) -  f'_p(l^*) = -h_p(l^*) > 0. 
\end{equation*}
Therefore, combining the above inequality with the convexity of $R_p(\cdot)$ we get that 
$R_p'(x) > 0$ for all $x > l^*$.  By this and \eqref{RPl}, 
we have that $R_p(x) \geq 0$ for all $x \geq l^*$. 
In all, the definition \eqref{fun_hP} yields $v_p(x;a^*,l^*) \geq f_p(x)$ for all $x \geq l^*$, proving that part $\rm (iii)$ indeed holds true. 

Suppose now that $a^* \leq x < l^*$ aiming for the proof of (iv). 
Differentiating \eqref{RP} and using \eqref{def_h_p_c}, we get
\begin{align*}
R_p'(x) 
&= - \Phi(q) Z^{(q+\lambda)}(l^*-x;\Phi(q)) \, f_p(l^*) - \lambda\int_0^{l^*-x} f_p'(u+x) \, W^{(q+\lambda)}(u) \du - f_p'(x) \\
&= - \Phi(q) R_p(x) - \lambda\int_0^{l^*-x} h_p(u+x) \, W^{(q+\lambda)}(u) \du - h_p(x) \,.
\end{align*}
Then, we define $\ol R_p(x) := e^{\Phi(q) x} R_p(x)$ and observe that 
\begin{align*}
\ol R_p'(x) 
=  e^{\Phi(q)x}\bigg[ - \lambda\int_0^{l^*-x} h_p(u+x) \, W^{(q+\lambda)}(u) \du - h_p(x) \bigg].
\end{align*}
In view of  \eqref{x_p_and_h_p}, we have for all $x \in (\ul x_p, l^*)$, which is a well-defined interval (cf.\ Lemma \ref{lem:exist_l}), that $h_p(x) < 0$. Therefore, $\ol R_p'(x)>0$ for all $x \in (\ul x_p, l^*)$. 
Combining this monotonicity with the fact that $\ol R_p(l^*)=0$ (see \eqref{RPl} and the definition of $\ol R_p$), we conclude that 
\begin{equation*} 
R_p(x) < 0 \,, \quad \text{for all } x \in [\ul x_p, l^*) \,.
\end{equation*} 

For $x \leq \underline{x}_p$,
using the probabilistic expression of $v_p$ in the definition \eqref{vf_0PC}, we observe that 
\begin{align*} 
v_p(x;a^*,l^*) 
&= \E_x \Big[ e^{-qT_{l^*}^-}f_p(X_{T_{l^*}^-}) 1_{\{T_{l^*}^-<\tau_{a^*}^-\}} \Big] 
\leq \E_x \Big[ e^{-qT_{l^*}^-}f_p(X_{T_{l^*}^-}) 1_{\{T_{l^*}^-<\infty\}} \Big] \notag\\ 
&\leq \sup_{\tau \in \mathcal{T}_c} \E_x \Big[ e^{-q\tau}f_p(X_{\tau}) 1_{\{\tau<\infty\}} \Big] 
= \E_x \Big[ e^{-q \tau^-_{\ul x_p}}f_p(X_{\tau^-_{\ul x_p}}) 1_{\{\tau^-_{\ul x_p}<\infty\}} \Big] \,,
\end{align*}
where the last optimality holds similarly to Section \ref{lambda=0} 
(by using $f_p$ instead of $f_c$ and $h_p$ from \eqref{def_h_p_c} instead of $h^o_c$ in \eqref{h_c_o} and using \eqref{x_p_and_h_p}).  Hence,
\begin{equation*} 
v_p(x;a^*,l^*) \leq  \E_x \Big[ e^{-q \tau^-_{\ul x_p}}f_p(X_{\tau^-_{\ul x_p}}) 1_{\{\tau^-_{\ul x_p}<\infty\}} \Big]  =  f_p(x) 
\,, \quad \text{for all } x \in [a^* \vee \ul x_p, \ul x_p].
\end{equation*} 

{\it Part} (v). This follows directly from the definition of $v_p(x;a^*,l^*)$ in \eqref{vf_0PC}, and completes the proof. 
\end{proof}

\section{Numerical results} \label{section_numerics}

In this section, we illustrate the analytical results focusing on case studies with put option type rewards for both players. 
These are in line with our main application of firm acquisitions in the Introduction, where the two players, an established, large and well-diversified firm $C$ and a smaller, less diversified firm $P$, are after acquiring a target firm at some time $t$, by paying its market value $e^{X_t}$.
Suppose that $f_i(x) = K_i - e^x$, for $i \in\{c,p\}$, for some fixed $K_p \geq K_c$.\footnote{ Note that $(K_i - e^x)^+$ is usually used to describe the payoff of a financial put option, but these two problem formulations are equivalent. 
This is because it is never optimal to stop when the payoff is negative -- as shown in Lemma  \ref{remark_domain_to_focus}, the optimal thresholds must lie in $(-\infty, \bar{x}_i)$, where the payoffs are strictly positive, thus $(K_i - e^x)^+ = K_i - e^x$ for all $x \in (-\infty, \bar{x}_i)$.}
The constants $K_i$ could reflect  
the benefits of acquisition for each firm, and the property $K_p \geq K_c$ reflects the fact that the benefits of acquisition could be lower for the already-diversified firm $C$ compared to $P$, due to the synergy effect (see, e.g.\ \cite{Zhou} for more details).
In the sequel, we examine both unequal $K_p > K_c$ and equal $K_p = K_c$ benefits of acquisition in our case studies.
In these case studies, Assumption \ref{Ass} is therefore satisfied with $\overline{x}_i = \log K_i$, while  
Assumption \ref{assump_a_underbar_finite} holds true with $\underline{a} = \log \frac {\Phi(q)K_c} {1+ \Phi(q)}$. We also have in view of \eqref{x_p_and_h_p} that
\begin{align*}
\underline{x}_p &= \log \frac {K_p \Phi(q)} {1+\Phi(q)} = \bar{x}_p +  \log \frac {\Phi(q)} {1+\Phi(q)}, \\ 
 \underline{x}_c &= \log \frac {(\Phi(q)+\lambda W^{(q+\lambda)}(0)) K_c} {1+ \Phi(q)+\lambda W^{(q+\lambda)}(0)} = \overline{x}_c + \log \frac {\Phi(q)+\lambda W^{(q+\lambda)}(0)} {1+ \Phi(q)+\lambda W^{(q+\lambda)}(0)}.
\end{align*}
%
%
%
As an underlying asset price $e^X$,
we consider the case of $X$ that is of the form
\begin{equation*}
 X_t = X_0 - \mu t+ \nu B_t + \sum_{n=1}^{M_t} Z_n, \quad 0\le t <\infty, 
\end{equation*}
where $\mu>0$ and $\nu\geq 0$ are constants, $B=( B_t: t\ge 0)$ is a standard Brownian motion, $M=(M_t: t\ge 0 )$ is a Poisson process with arrival rate $\alpha$, and  $Z = ( Z_n: n = 1,2,\ldots )$ is an i.i.d.\ sequence of exponential random variables with parameter $\beta$. The processes $B$, $M$, and $Z$ are assumed mutually independent. 
In this model, the scale functions admit explicit expressions that can be found, e.g. in \cite{Egami_Yamazaki_2010_2, KKR}.

In all forthcoming analyses, we use the parameter values $\nu = 0.2$, $\alpha = 1$, $\beta = 2$, $q = 0.05$, $\mu = 0.313$,  
and unless stated otherwise, we also set $\lambda = 1$.

%
%

\subsection{Optimality: Asymmetric rewards case ($K_p = 60$, $K_c = 50$)} 
 
As discussed in Section \ref{1storder}, 
the optimal barriers corresponding to the Nash equilibrium $(a^*,l^*)$ are those satisfying simultaneously the conditions $\mathbf{C}_a$ and $\mathbf{C}_l$, given by \eqref{cond_aC}--\eqref{cond_lP} and expressed as the roots of $a \mapsto I(a;l)$ in \eqref{Jal} and $l \mapsto J(l;a)$ in \eqref{Ila}, respectively. 

In Figure \ref{figure_C_l_a}.(i), we plot $a \mapsto I(a;l)$ for five values of $l$ equally spaced between $\underline{x}_c$ and $\overline{x}_p$ (cf.\ Proposition \ref{existNash}). 
It is observed that $I(\cdot;l)$ starts from positive values and ends at non-positive values -- admitting a unique root $\tilde{a}(l) \in [\underline{x}_c, \overline{x}_c \wedge l]$. 
This illustrates Lemma \ref{lem:exist_a}.(II).(i)--(ii).
In Figure \ref{figure_C_l_a}.(ii), we plot $l \mapsto J(l;a)$ for five values of $a$ equally spaced between $\underline{x}_c$ and $\overline{x}_c$ (cf.\ Proposition \ref{existNash}). 
It is observed that $J(\cdot;a)$ starts from positive values and ends at negative values -- admitting a unique root $\tilde{l}(a) \in [\underline{x}_p, \ol{x}_p]$. This illustrates Lemma \ref{lem:exist_l}.(i)--(ii). 

\begin{figure}[htbp]
\vspace{-5mm}
\begin{center}
\begin{minipage}{1.0\textwidth}
\centering
\begin{tabular}{cc}
 \includegraphics[scale=0.5]{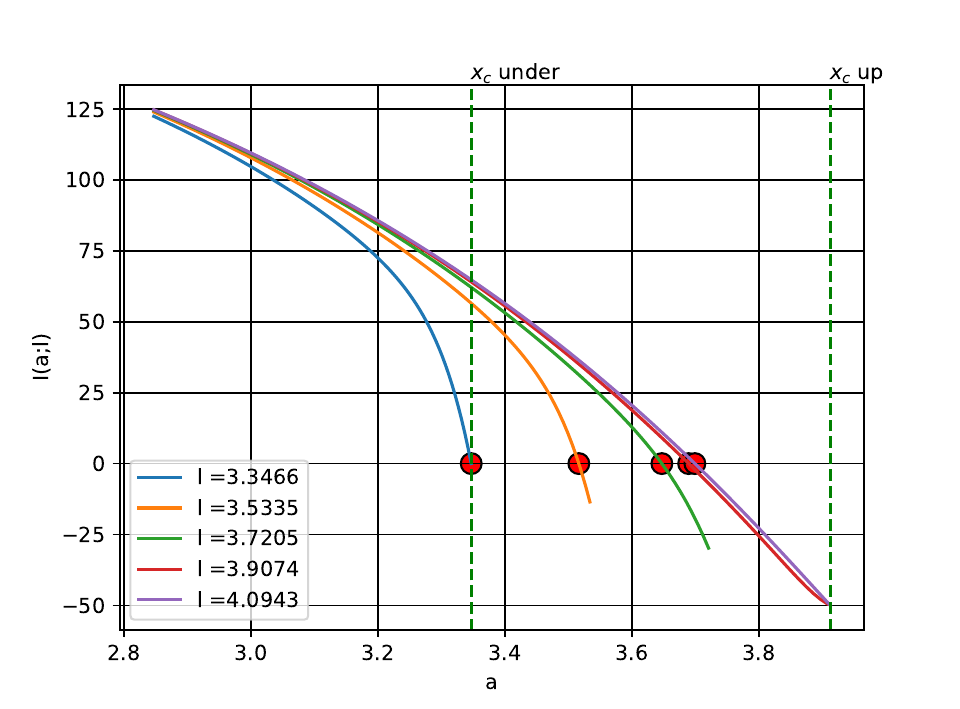} & \includegraphics[scale=0.5]{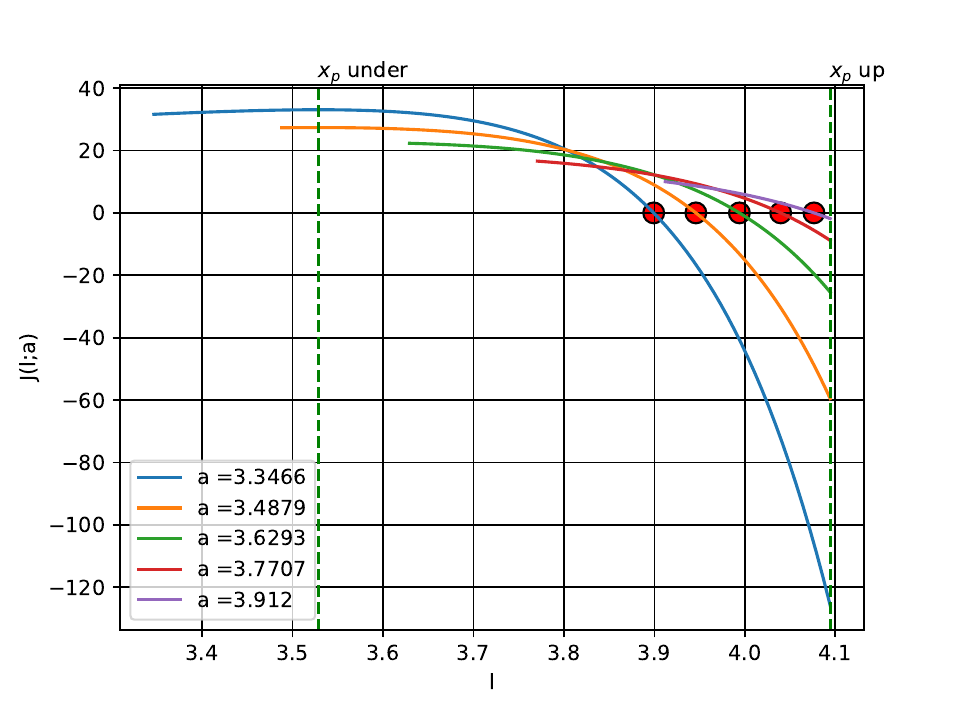}  \\
(i) $a \mapsto I(a;l)$ & (ii)  $l \mapsto J(l;a)$
 \end{tabular}
\end{minipage}
\end{center}
\vspace{-4mm}
\caption{(i) Plot of $a \mapsto I(a;l)$  on $[{\underline{x}_c-0.5}, l \wedge \overline{x}_c]$ for $l = \underline{x}_c, \ldots, \overline{x}_p$. 
The roots of $I(\cdot; l) = 0$ are indicated by circles and the vertical dotted lines correspond to $a = \underline{x}_c,  \overline{x}_c$.
(ii) Plot of $l \mapsto J(l;a)$ on $[a, \overline{x}_p]$ for $a = \underline{x}_c, \ldots, \overline{x}_c$. The roots of $J(\cdot; a) = 0$ are indicated by circles and the vertical dotted lines correspond to $l = \underline{x}_p, \overline{x}_p$.
} \label{figure_C_l_a}
\end{figure}

The roots $\tilde{a}(l)$ and $\tilde{l}(a)$ can be computed via classical bisection. 
In Proposition \ref{existNash}, we showed the existence of $l^*$ such that $J(l^*;\tilde{a}(l^*))=0$, hence $(\tilde{a}(l^*), l^*)$ 
becomes a Nash equilibrium. 
In fact, as shown in Figure~\ref{figure_J_I_2}.(ii), the mapping $l \mapsto J(l; \tilde{a}(l))$ is monotone, hence the value $l^*$ (and therefore  also $a^*:=\tilde{a}(l^*)$) is unique in this case study.
For completeness, we also plot $a \mapsto I(a; \tilde{l}(a))$ in Figure \ref{figure_J_I_2}.(i). 
It is also confirmed to be monotone and the unique root $a^*$ such that $I(a^*;\tilde{l}(a^*))=0$ leads to the unique Nash equilibrium $(a^*, \tilde{l}(a^*)) \equiv (\tilde{a}(l^*), l^*)$.

\begin{figure}[htbp]
\vspace{-4mm}
\begin{center}
\begin{minipage}{1.0\textwidth}
\centering
\begin{tabular}{cc}
 \includegraphics[scale=0.5]{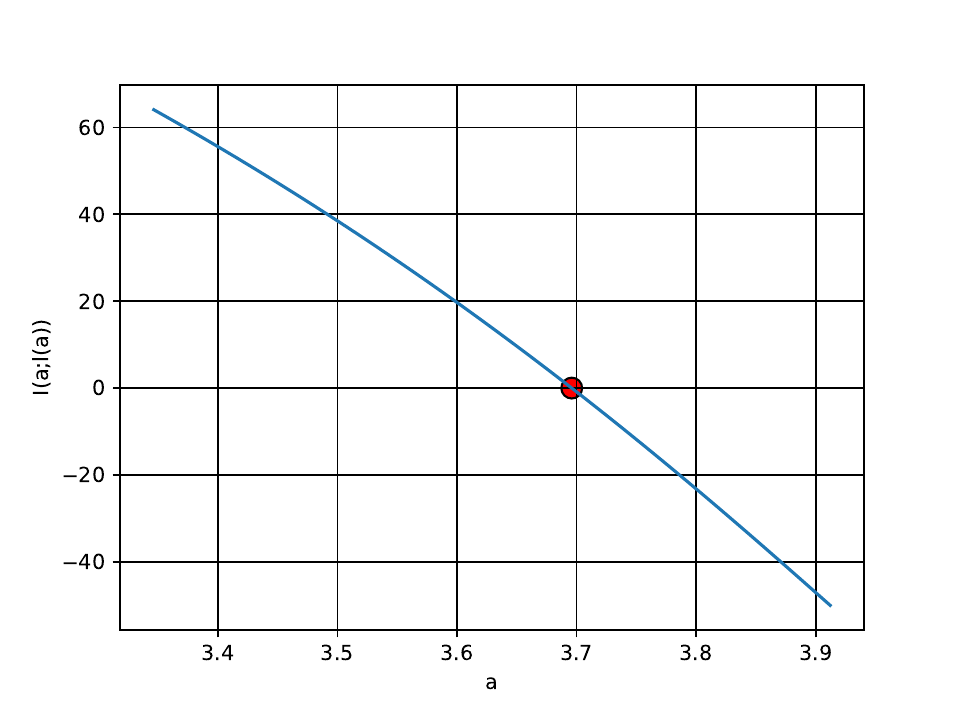} &  \includegraphics[scale=0.5]{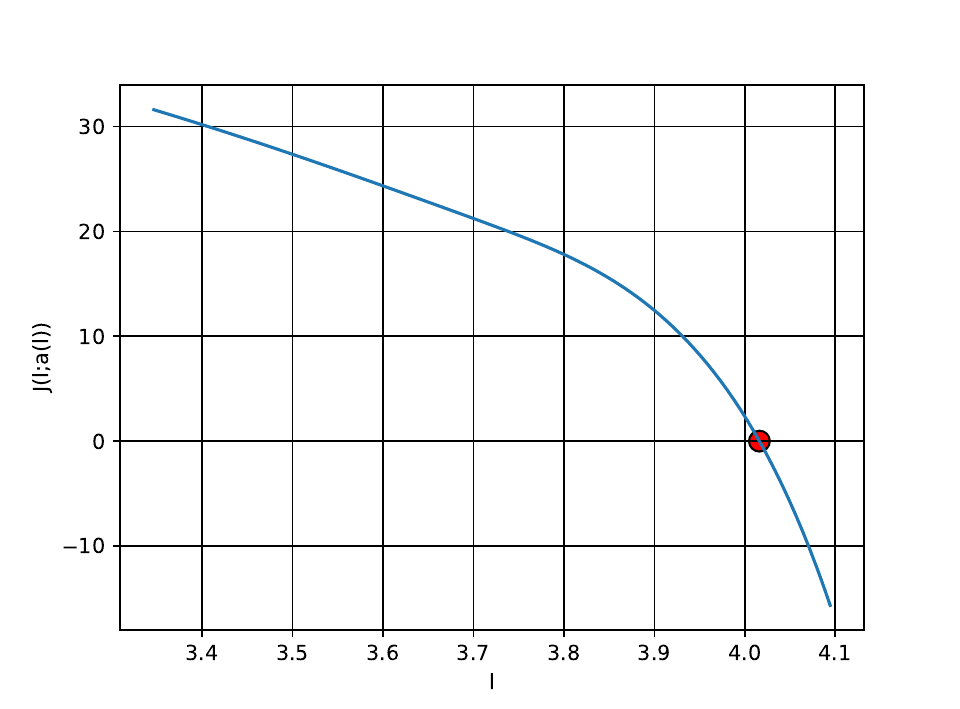} \\
(i)  $a \mapsto I(a;\tilde{l}(a))$ & (ii) $l \mapsto J(l;\tilde{a}(l))$  
 \end{tabular}
\end{minipage}
\end{center}
\vspace{-4mm}
\caption{(i) Plot of $a \mapsto I(a;\tilde{l}(a))$ on $[\underline{x}_c, \overline{x}_c]$.  The root of $I(\cdot;\tilde{l}(\cdot)) = 0$ is indicated by the circle.
(ii) Plot of $l \mapsto J(l;\tilde{a}(l))$ on  $[\underline{x}_c, \overline{x}_p]$. The root of $J(\cdot;\tilde{a}(\cdot)) = 0$ is indicated by the circle. 
} \label{figure_J_I_2}
\end{figure}

%

With $(a^*,l^*)$ obtained by the above procedures, we then compute the value functions for both players $C$ and $P$ and illustrate their optimality.  To this end, we compare them with those for different {(suboptimal)} choices {of $(a,l^*)$ and $(a^*,l)$}. 
In Figure \ref{figure_optimality}, we illustrate the results, showing that the value function indeed dominates those with wrong barrier selections for both players. 
For player $P$, we also plot the reward function $e^x \mapsto K_p- e^x$, to illustrate that the optimal barrier $l^*$ (in the plot, $e^{l^*}$) is the only barrier at which 
the value function and reward function coincide {(cf.\ optimality condition $\mathbf{C}_l$ in \eqref{cond_lP})}.



\begin{figure}[htbp]
\vspace{-5mm}
\begin{center}
\begin{minipage}{1.0\textwidth}
\centering
\begin{tabular}{cc}
 \includegraphics[scale=0.5]{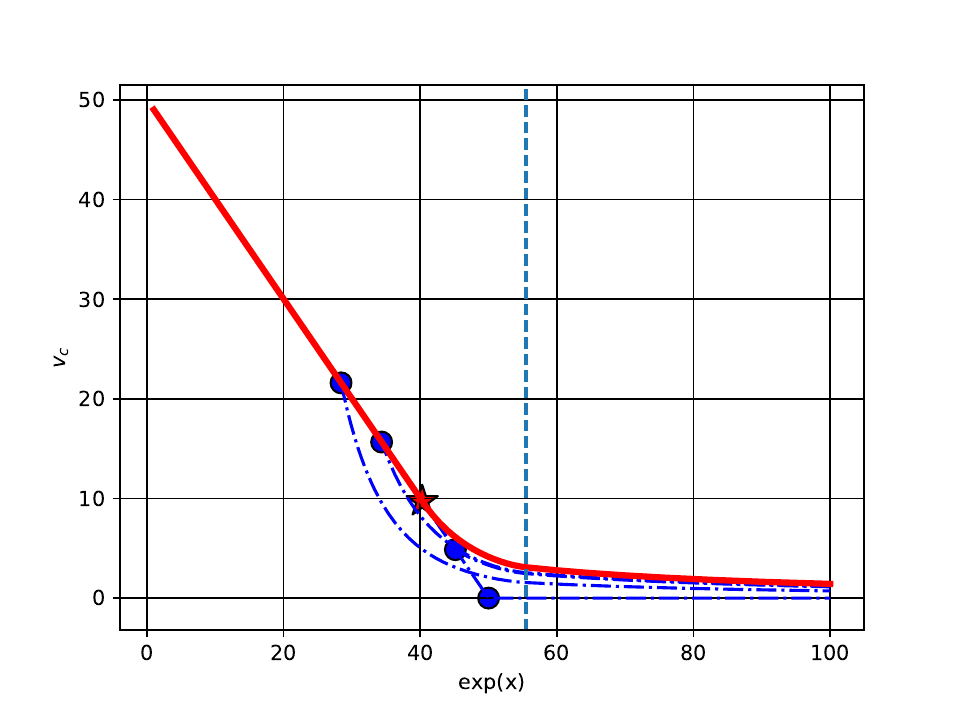} & \includegraphics[scale=0.5]{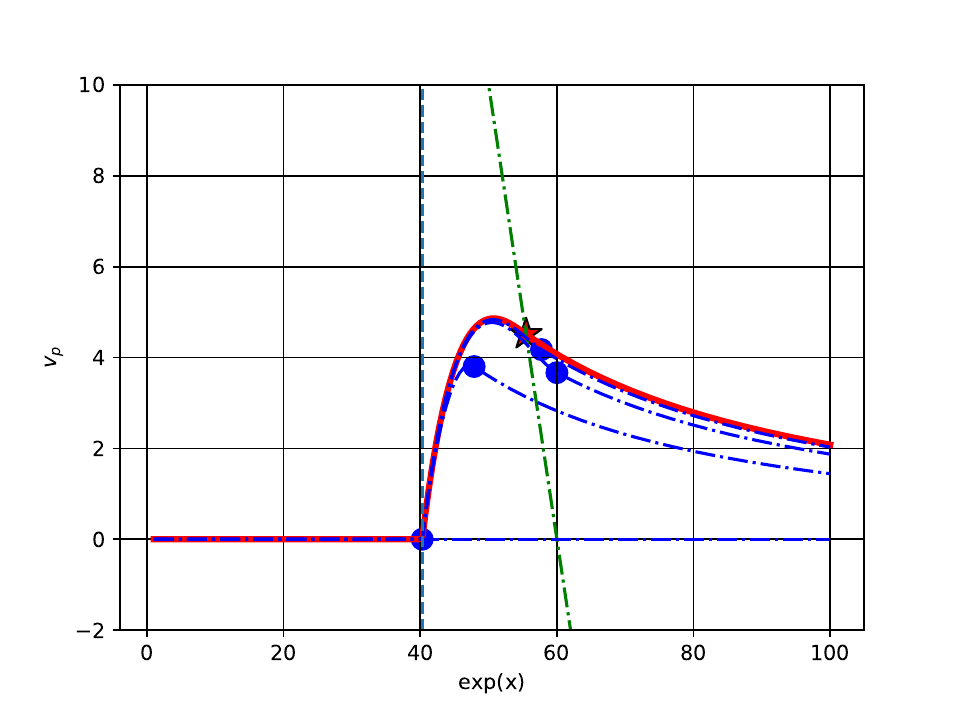}  \\
(i) $e^x \mapsto v_c(x; a,l^*)$ & 
(ii) $e^x \mapsto  v_p(x; a^*,l)$ 
\end{tabular}
\end{minipage}
\vspace{-4mm}
\caption{(i) Plots of $e^x \mapsto v_c(x; a^*,l^*)$ in red, along with $e^x \mapsto v_c(x; a,l^*)$ in dotted blue for $e^a = e^{\underline{x}_c}, (e^{\underline{x}_c} + e^{a^*})/2, ( e^{a^*} + K_c)/2, K_c$. The points at $a$ and $a^*$ are indicated by circles and a star, respectively. The value at $l^*$ is indicated by the dotted vertical line. 
(ii) Plots of $e^x \mapsto v_p(x; a^*,l^*)$ in red, along with $v_p(x; a^*,l)$ in dotted blue for $e^l = e^{a^*}, (e^{a^*} + e^{l^*})/2, (K_p+ e^{l^*})/2, K_p$. The points at $l$ and $l^*$ are indicated by circles and a star, respectively. The value at $a^*$ is indicated by the dotted vertical line and the green line depicts the (reward) mapping $e^x\mapsto K_p-e^x$.  
} \label{figure_optimality}
\end{center}
\end{figure}

\subsection{Optimality: Symmetric rewards case ($K_c=K_p = 60$)} 

In order to illustrate that our results hold for symmetric reward functions $f_c \equiv f_p$, we repeat the experiment with all other parameters unchanged. 
Figure \ref{plot_symmetric} plots the same functions as those in Figures \ref{figure_C_l_a}, \ref{figure_J_I_2} and \ref{figure_optimality} when $K_c$ was chosen to be $50$. 
In particular, the optimality of the selected thresholds is illustrated in the figures at the bottom. 
It is interesting to note that the shapes of these functions remain the same. 
This is due to the fact that, although the rewards are the same for the two players, the game is still asymmetric in the sense that player $C$ can choose to stop anytime, while player $P$ can stop only at Poisson arrival times.

\begin{figure}[htbp]
\vspace{-4mm}
\begin{center}
\begin{minipage}{1.0\textwidth}
\centering
\begin{tabular}{cc}
 \includegraphics[scale=0.5]{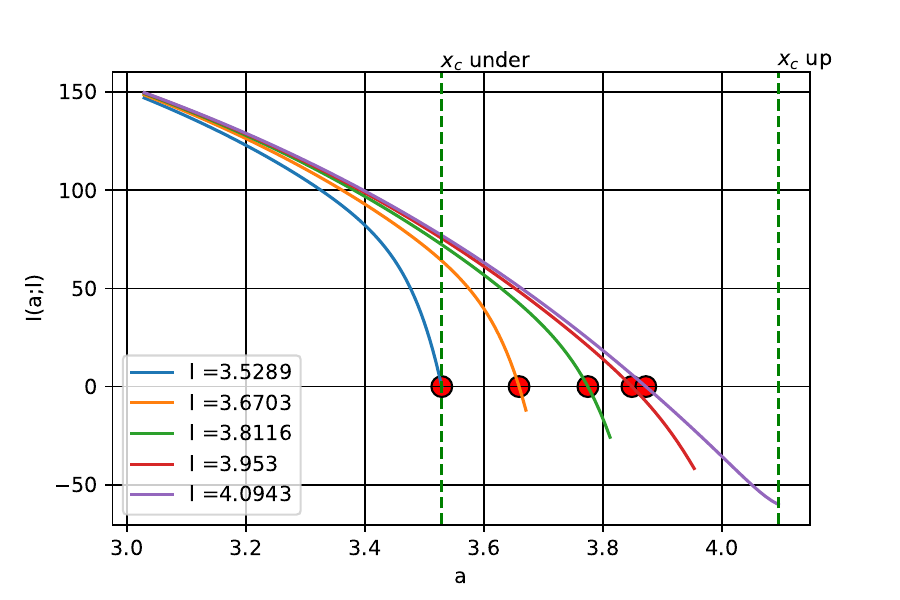} &  \includegraphics[scale=0.5]{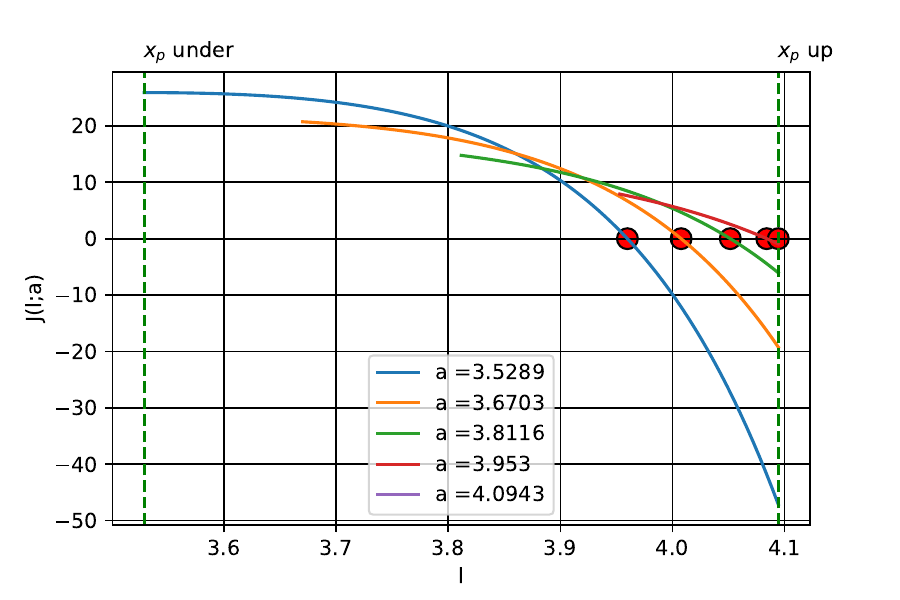} \\
$a \mapsto I(a;l)$ &   $l \mapsto J(l;a)$ \\
  \includegraphics[scale=0.5]{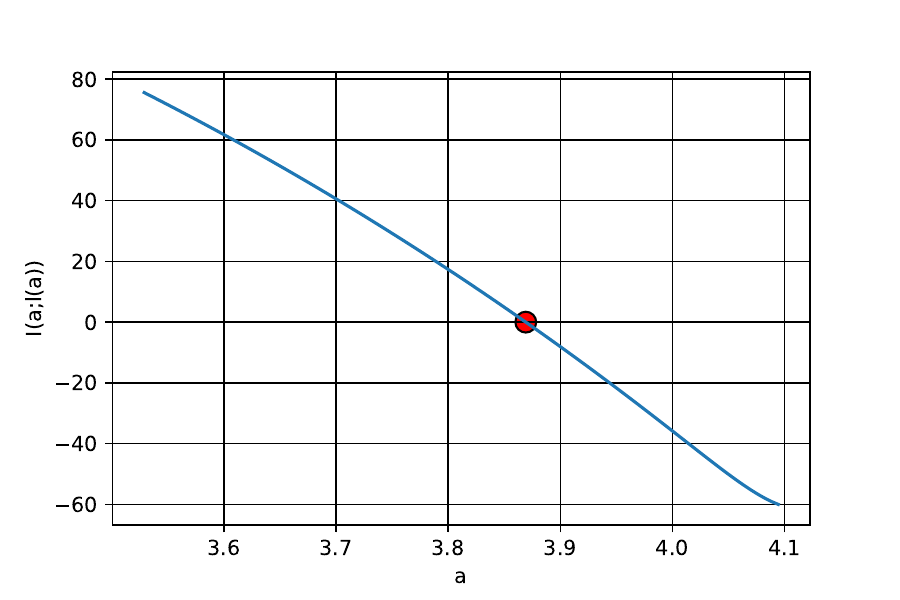} &  \includegraphics[scale=0.5]{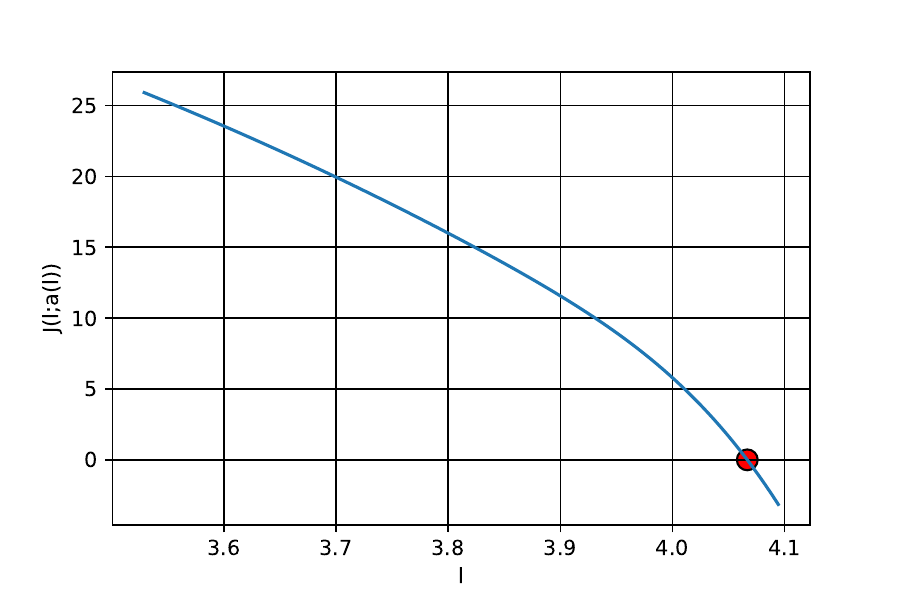} \\
  $a \mapsto I(a;\tilde{l}(a))$ &  $l \mapsto J(l;\tilde{a}(l))$  \\
   \includegraphics[scale=0.5]{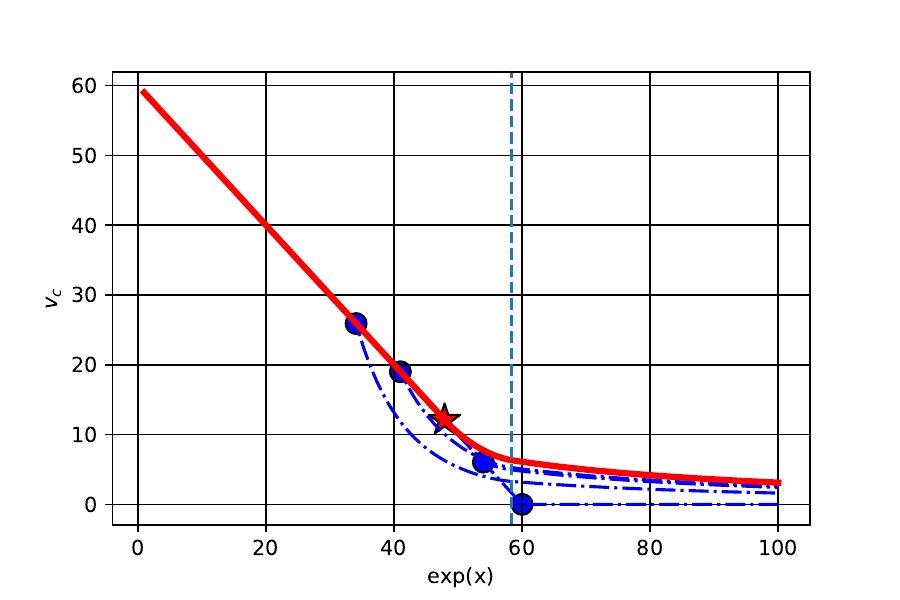} &  \includegraphics[scale=0.5]{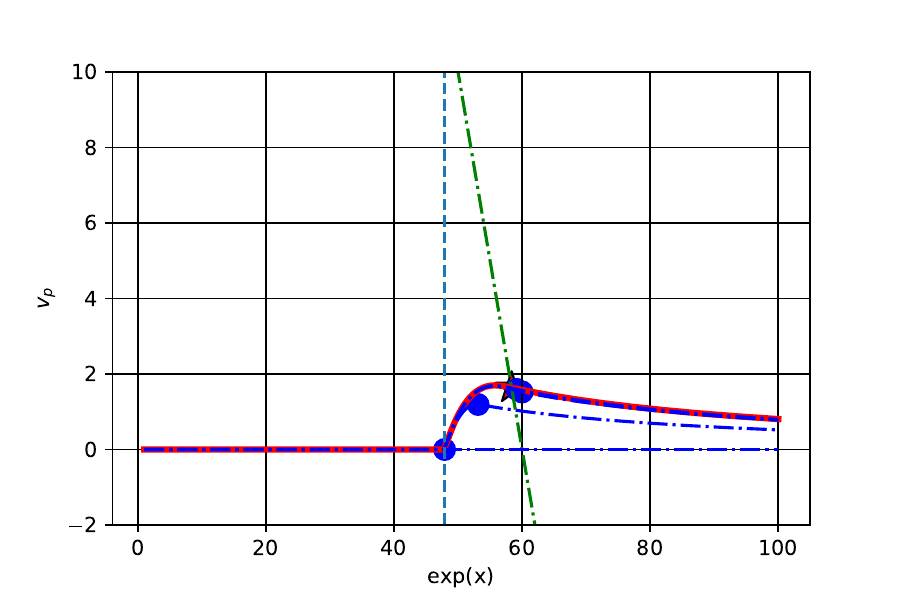}  \\
    $e^x \mapsto v_c(x; a,l^*)$ & 
 $e^x \mapsto  v_p(x; a^*,l)$ 
 \end{tabular}
\end{minipage}
\end{center}
\vspace{-4mm}
\caption{Symmetric rewards case with $K_c=K_p = 60$. 
The panels in the first, second and third rows plot the same functions as those in Figure \ref{figure_C_l_a}, \ref{figure_J_I_2} and \ref{figure_optimality}, respectively, when $K_c = 50$.} \label{plot_symmetric}
\end{figure}

\subsection{Sensitivity with respect to $\lambda$: Asymmetric rewards case ($K_p = 60$, $K_c = 50$)}

We shall now analyse how the equilibrium strategies of both players change with respect to the 
{the rate of exercise opportunities} $\lambda$ of player $P$.  
In Figure \ref{plot_lambda}.(i)--(ii), we plot the equilibrium value function of each player for $\lambda$ ranging from $0.1$ to $500$. We see that $ v_c(x; a^*,l^*)$ is monotonically decreasing in $\lambda$ for all asset values  $e^x$, {whereas $ v_p(x; a^*,l^*)$ 
seems monotonically increasing for large values of $e^x$, but the monotonicity is non-conclusive when the asset value $e^x$ is low, due to the involvement of player $C$}.

To complete the picture, we plot in Figure \ref{plot_lambda}.(iii)  the barriers $a^*$ and $l^*$ as functions of $\lambda$. 
As $\lambda$ increases, the threshold $a^*$ 
seems to converge to some value close to, but slightly smaller than $\overline{x}_c = \log K_c = \log 50$, so that it yields a positive reward at stopping, while $l^*$ converges to some value larger than $\overline{x}_c$.\footnote{We tested this for several other selections of parameters and obtained a similar asymptotic behaviour.} 
When $\lambda$ is very large, the 
{exercise opportunities} of player $P$ are almost as frequent as those of player $C$. 
However, player $C$'s advantageous right to always stop first when $ a^*=l^*$, no matter how large $\lambda$ is, leads player $P$ to select $l^*$ strictly larger than $a^*$ (otherwise Player $P$'s reward will always be zero).
Moreover, if player $P$ chooses $l < \overline{x}_c$ 
and has a very high 
{frequency of exercise opportunities}, then player $C$ will try to increase $a$ to $l$ to stop before player $P$ -- in response to this, player $P$ needs to increase $l$. 
In this way, the selection of barriers $l < \overline{x}_c$ 
cannot be part of an equilibrium strategy, justifying the aforementioned asymptotic behaviour. 

\begin{figure}[htbp]
\vspace{-5mm}
\begin{center}
\begin{minipage}{1.0\textwidth}
\centering
\begin{tabular}{cc}
 \includegraphics[scale=0.5]{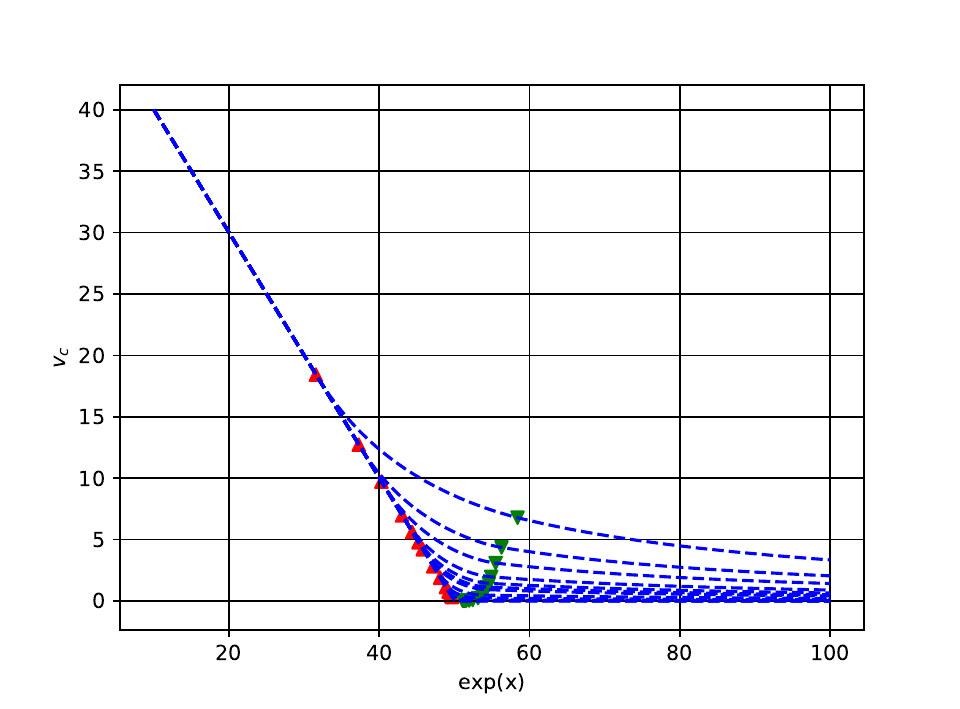} & \includegraphics[scale=0.5]{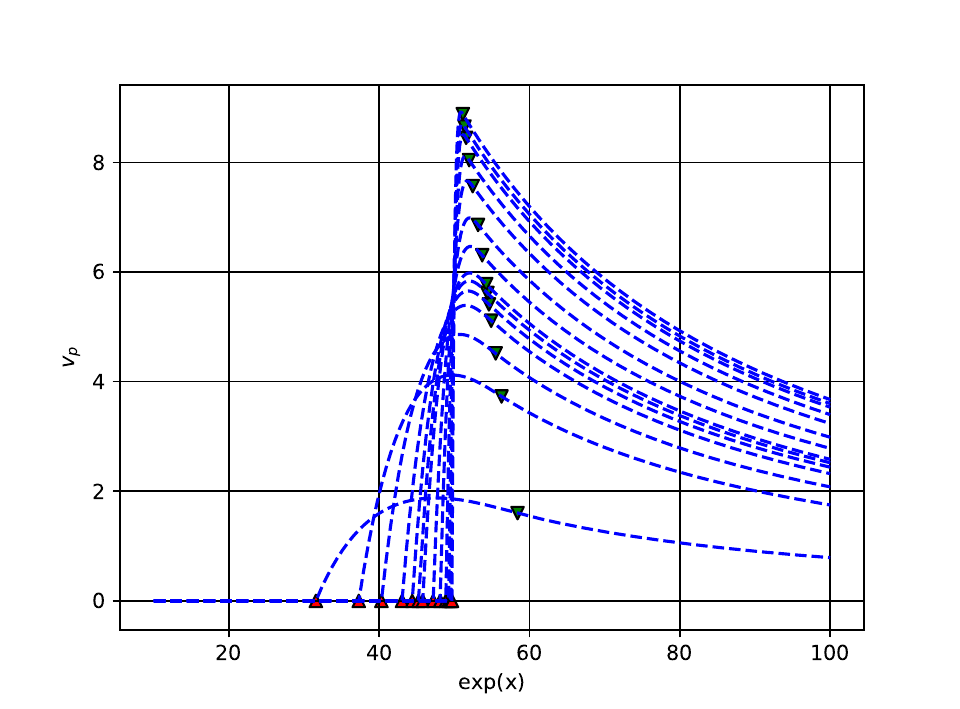}  \\
(i) $e^x \mapsto v_c(x; a^*,l^*)$ & 
(ii)  $e^x \mapsto v_p(x; a^*,l^*)$ 
\end{tabular}
\begin{tabular}{c}
\includegraphics[scale=0.5]{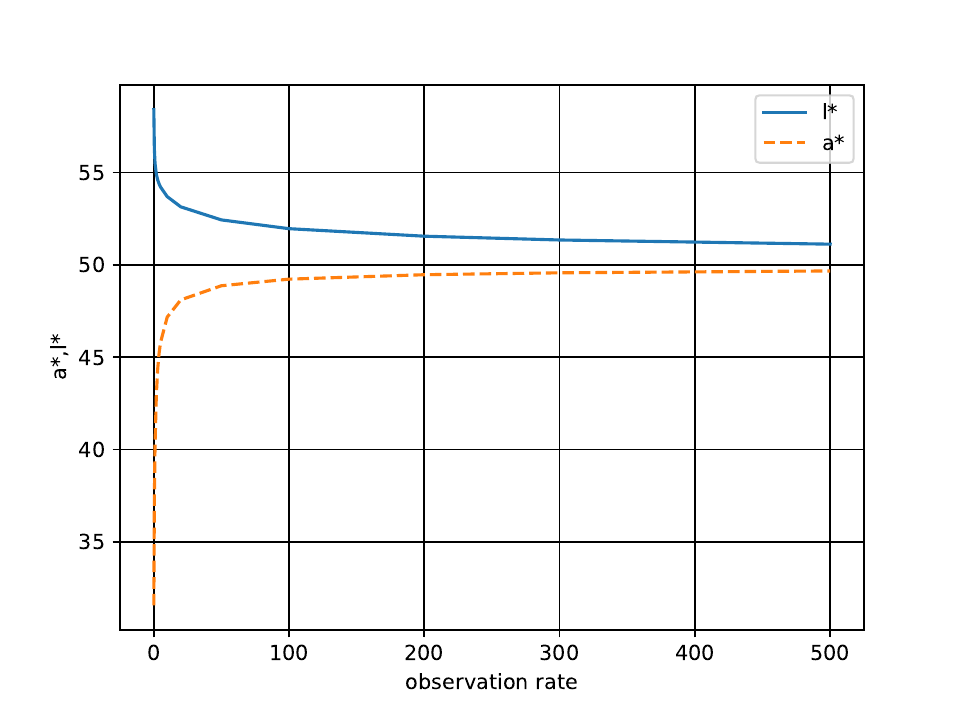} \\
(iii) $\lambda \mapsto a^*, l^*$
\end{tabular}
\end{minipage}
\vspace{-4mm}
\caption{The value functions (i) $e^x \mapsto v_c(x; a^*,l^*)$ and (ii) $e^x \mapsto v_p(x; a^*,l^*)$, for $\lambda = 0.1$, $0.5$, $1$, $2$, $3$, $4$, $5,10,20,50,100,200,300,500$. {Red (resp., Green) triangles represent $a^*$ (resp., $l^*$) for each $\lambda$.}
(iii) The barriers $l^*$ and $a^*$ as functions of $\lambda$.
} \label{plot_lambda}
\end{center}
\end{figure}

\subsection{Value of available exercise opportunities: Asymmetric rewards case (fixed $K_p = 60$)} 
\label{Sec:ValueOpp}

We finally define the \emph{value of {additional exercise opportunities}} to be the amount $\delta := f_p(x) - f_c(x) = K_p-K_c$ 
such that the value functions of both players coincide $v_c(x;a^*,l^*) = v_p(x;a^*,l^*)$.  
In our numerical results, the difference $v_c(x;a^*,l^*) - v_p(x;a^*,l^*)$ is monotone in $K_c$, hence we obtain via the bisection method the unique zero that leads to the desired $\delta$ value. 
In order to analyse how this $\delta$ changes with respect to the starting value $e^x$ and player $P$'s observation rate $\lambda$, we plot $\delta$ in Figure \ref{figure_fair_fee} 
as a function of $e^x$ (when $\lambda = 1$) and $\lambda$ (when $e^x = K_p = 60$). 
It is observed that the value of 
{additional exercise opportunities} decreases both in $e^x$ and $\lambda$, however it does not seem to converge 
to zero as $\lambda \to \infty$, because player $C$ still has the right to stop before player $P$ no matter how large $\lambda$ is.

\begin{figure}[htbp]
\begin{center}
\begin{minipage}{1.0\textwidth}
\centering
\begin{tabular}{cc}
 \includegraphics[scale=0.5]{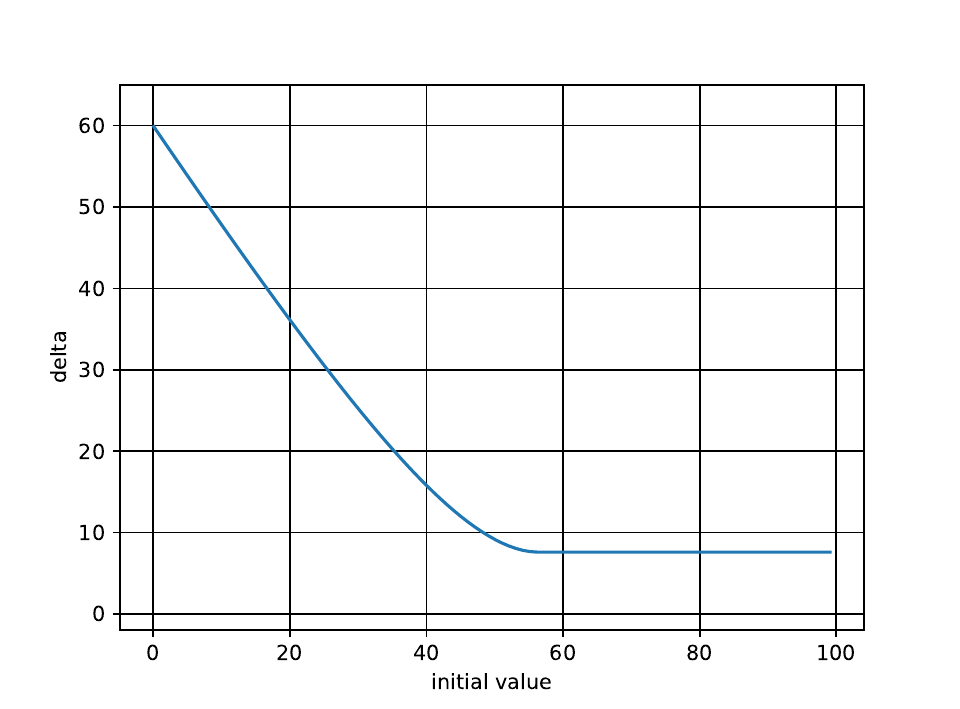} & \includegraphics[scale=0.5]{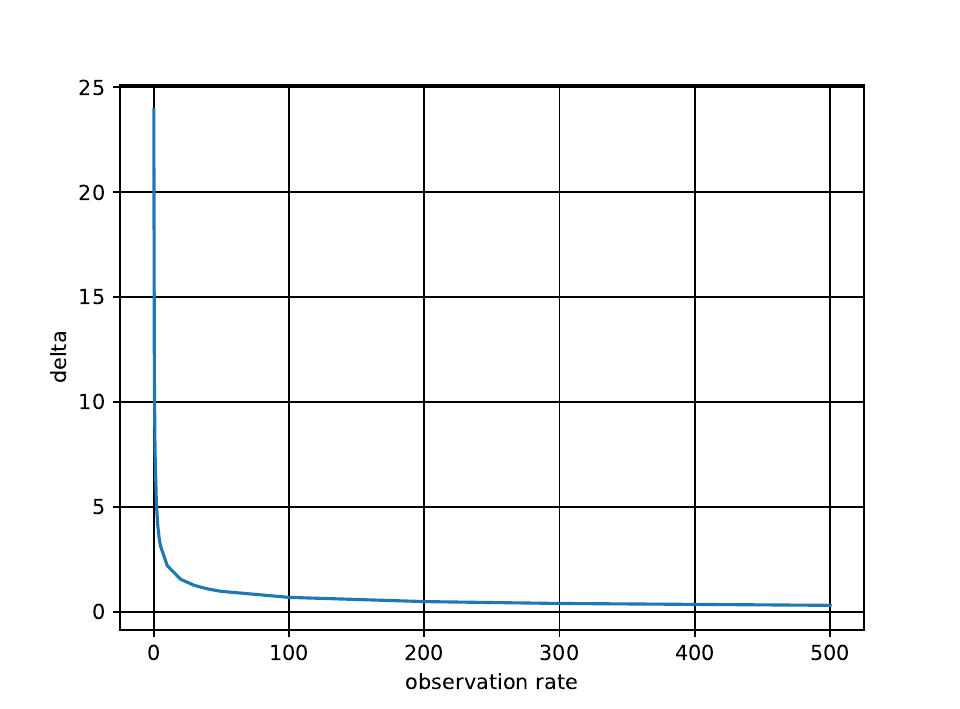}  \\
(i) $e^x \mapsto \delta$ 
& (ii) $\lambda \mapsto \delta$ 
 \end{tabular}
\end{minipage}
\vspace{-4mm}
\caption{
Value of additional 
{exercise opportunities} $\delta := K_p - K_c$, such that $v_c(x; a^*, l^*) = v_p(x; a^*, l^*)$, as a function of (i) the initial asset value $e^x$ and (ii) the observation rate $\lambda$.
} \label{figure_fair_fee}
\end{center}
\end{figure}

\section*{Acknowledgement}

The authors thank the Area Editor, Associate Editor, and anonymous referees for the careful reading of the paper and constructive comments and suggestions. They are also grateful to Haejun Jeon for valuable discussions on the potential application of the considered problem.

\begin{appendix}

\section{Proofs of fluctuation identities}
\label{fluct_proof}


We first obtain the following two results, which will be used in the proofs of Lemmata \ref{Prop_cost_cont_player} and \ref{Prop_cost_per_player}.

The following result provides an important fluctuation identity involving the upward-stopped process, which satisfies $X_{\tau_0^+} \geq 0$, $\px$-a.s. and whose overshoots are expressed in terms of the scale functions. 
\begin{lemma} \label{flW-}
Recall the definition \eqref{olW} of $\mathscr{W}^{(q, \lambda)}_b$. 
For $b>0$, $q\geq0$ and $0 \leq x \leq a$, we have
\begin{align*} 
\E_{-x} \left[e^{-(q+\lambda)\tau_0^+} W^{(q)}(b-X_{\tau_0^+});\tau_0^+<\tau_{-a}^-\right]&=\mathscr{W}^{(q, \lambda)}_b (x)-\frac{W^{(q+\lambda)}(x)}{W^{(q+\lambda)}(a)} \, \mathscr{W}^{(q, \lambda)}_b (a).
\end{align*}
\end{lemma}
\begin{proof}
This is a direct consequence of \cite[Lemma 2.1]{LRZ} (see also \cite[Lemma 2.2]{LRZ}) and the spatial homogeneity of L\'evy processes.  
\end{proof}

\begin{lemma}\label{aux_SP}
For $c\geq0$ and $b > 0$, we have
\begin{align*}
\lambda \int_0^b W^{(q)}(u) \, W^{(q+\lambda)}(b+c-u) \diff u 
= W^{(q+\lambda)}(b+c) - \mathscr{W}^{(q,\lambda)}_b (c).
\end{align*}
\end{lemma}
\begin{proof} 

By identity (6) in \cite {LRZ}, we have for $p,q \geq0$ and $x\in\R$
\begin{align}
\label{LRZ-W}
W^{(q)}(x) - W^{(p)}(x) &= (q-p) \int_0^x W^{(p)}(x-y) W^{(q)}(y) \diff y.
\end{align}
Using the equation \eqref{LRZ-W} and the definition \eqref{olW} of $\mathscr{W}^{(q,\lambda)}_b(\cdot)$,
 we get that
\begin{align*}
&\lambda \int_0^b W^{(q)}(u) \, W^{(q+\lambda)}(b+c-u) \diff u \notag\\
&= \lambda \int_0^{b+c} W^{(q)}(u) \, W^{(q+\lambda)}(b+c-u) \diff u 
- \lambda \int_b^{b+c} W^{(q)}(u) \, W^{(q+\lambda)}(b+c-u) \diff u \notag\\
&= W^{(q+\lambda)}(b+c) - W^{(q)}(b+c) 
- \lambda \int_0^{c} W^{(q)}(u+b) W^{(q+\lambda)}(c-u) \diff u 
= W^{(q+\lambda)}(b+c) - \mathscr{W}^{(q,\lambda)}_b(c) \,,
\end{align*}
which completes the proof. 
\end{proof}

\subsection{Proof of Lemma \ref{W/Z}} \label{proof_W/Z}
By the probabilistic expression \eqref{laplace_upper},  for $0<x<y$,  we have
\begin{equation*}
Z^{(q+\lambda)}(x;\Phi(q))-\frac{Z^{(q+\lambda)}(y;\Phi(q))}{W^{(q+\lambda)}(y)}W^{(q+\lambda)}(x) > 0.
\end{equation*}
Hence,
\begin{equation*}
\frac {W^{(q+\lambda)}(x)}  {Z^{(q+\lambda)}(x;\Phi(q))}< \frac {W^{(q+\lambda)}(y)}  {Z^{(q+\lambda)}(y;\Phi(q))} \,, 
\quad \text{for all } 0<x<y.
\end{equation*}

\subsection{Proof of Lemma \ref{Prop_cost_cont_player}}\label{proof_prop_cost_cont_player}

Throughout this proof, we write $\tilde{g}(x) := \E_x \left[e^{-q \tau_a^-} f_c(X_{\tau_a^-}) 1_{\{\tau_a^-< T_{l}^- \wedge \tau_b^+\}} \right]$ for $x \in \R$. 
 It is straightforward to see by the definitions of $\tilde{g}(x)$ and $\tau_{a}^-$, that $\tilde{g}(x) = f_c(x)$ for all $x < a$.
The remainder of the proof is devoted to the case of $x \geq a$.
We prove this result in the following steps. 

{\it Step 1: Computation of $\tilde{g}(x)$ in terms of $\tilde{g}(l)$.}
On one hand, for $x \geq l$, using the strong Markov property, spatial homogeneity of L\'evy processes and \eqref{upcr}, we obtain 
\begin{align}\label{iden_h_SP_0}
\tilde{g}(x) 
&= \E_x \left[e^{-q\tau_l^-} \tilde{g}(X_{\tau_{l}^-}) ; \tau_l^-<\tau_b^+\right] 
= \tilde{g}(l) \, \frac{W^{(q)}(b-x)}{W^{(q)}(b-l)}.
\end{align}	
On the other hand, for $x \in [a,l)$, 
using again the strong Markov property, we obtain
\begin{align}\label{iden_h_SP_1}
\tilde{g}(x) = 
\E_x \left[e^{-q\tau_{a}^-} f_c(X_{\tau_{a}^-}) ; \tau_{a}^- < T^{(1)}\wedge\tau_{l}^+ \right] 
+ \E_x \left[e^{-q\tau_{l}^+} \tilde{g}(X_{\tau_{l}^+}) ; \tau_{l}^+ < T^{(1)}\wedge \tau_{a}^- \right].
\end{align}
For the first term on the right-hand side of \eqref{iden_h_SP_1}, using the spatial homogeneity of L\'evy processes and \eqref{upcr}, 
\begin{align*}
\E_x \left[e^{-q\tau_{a}^-} f_c(X_{\tau_{a}^-}) ; \tau_{a}^- < T^{(1)} \wedge\tau_{l}^+ \right] 
= \E_x \left[e^{-(q+\lambda) \tau_{a}^-} f_c(X_{\tau_{a}^-}) ; \tau_{a}^- < \tau_{l}^+ \right] 
=f_c(a) \, \frac{W^{(q+\lambda)}(l-x)}{W^{(q+\lambda)}(l-a)}.
\end{align*}
Now, for the second term on the right-hand side of \eqref{iden_h_SP_1}, we firstly see from \eqref{iden_h_SP_0}, given that $X_{\tau_{l}^+} \geq l$, $\p_x$-a.s. (due to the possibility of positive jumps), that
\begin{align*}
\tilde{g}(X_{\tau_{l}^+}) 
= \tilde{g}(l) \, \frac{ W^{(q)}\big(b -X_{\tau_{l}^+} \big)}{W^{(q)}(b-l)}, \quad  \p_x-\text{a.s.},
\end{align*}
which then gives that  
\begin{align*}
\E_x \left[e^{-q\tau_{l}^+} \tilde{g}(X_{\tau_{l}^+}) ; \tau_{l}^+ < T^{(1)}\wedge \tau_{a}^- \right] 
&= \frac{\tilde{g}(l)}{W^{(q)}(b-l)} \, \E_x \left[e^{-q\tau_{l}^+} W^{(q)}(b -X_{\tau_{l}^+}) ; \tau_{l}^+ < T^{(1)}\wedge \tau_{a}^- \right] .
\end{align*}
Combining the above with the spatial homogeneity of L\'evy processes and Lemma \ref{flW-}, we thus obtain
\begin{align*}
\E_x \left[e^{-q\tau_{l}^+} \tilde{g}(X_{\tau_{l}^+}) ; \tau_{l}^+ < T^{(1)}\wedge \tau_{a}^- \right] 
&= \frac{\tilde{g}(l)}{W^{(q)}(b-l)} \, \E_{x-l} \left[e^{-q\tau_{0}^+} W^{(q)}(b-l -X_{\tau_{0}^+}) ; \tau_{0}^+ < T^{(1)}\wedge \tau_{a-l}^- \right] \notag\\ 
&=\frac{\tilde{g}(l)}{W^{(q)}(b-l)}\bigg(\mathscr{W}^{(q, \lambda)}_{b-l}(l-x) - \frac{W^{(q+\lambda)}(l-x)}{W^{(q+\lambda)}(l-a)} \, \mathscr{W}^{(q, \lambda)}_{b-l}(l-a)\bigg).
\end{align*}

Therefore, using \eqref{olW0} and putting all the pieces together, we obtain for $x \geq a$, 
\begin{align}\label{iden_h_SP_2}
\tilde{g}(x) = 
f_c(a) \, \frac{W^{(q+\lambda)}(l-x)}{W^{(q+\lambda)}(l-a)}
+ \frac{\tilde{g}(l)}{W^{(q)}(b-l)} \bigg(\mathscr{W}^{(q, \lambda)}_{b-l}(l-x) - \frac{W^{(q+\lambda)}(l-x)}{W^{(q+\lambda)}(l-a)} \, \mathscr{W}^{(q, \lambda)}_{b-l}(l-a) \bigg).
\end{align}

{\it Step 2: Computation of $\tilde{g}(l)$.} 
We note that we can write
\begin{align*}
\tilde{g}(l) = 
\E_l \left[e^{-q\tau_{a}^-} f_c(X_{\tau_{a}^-}) ; \tau_{a}^- < T^{(1)} \wedge \tau_{b}^+ \right] 
+ \E_l \left[e^{-qT^{(1)}}\tilde{g}(X_{T^{(1)}})1_{\{X_{T^{(1)}}\geq l\}};T^{(1)}< \tau_{a}^-\wedge\tau_{b}^+\right].
\end{align*}
The first term on the right-hand side of the above equation can be expressed as
\begin{align*}
\E_l \left[e^{-q\tau_{a}^-} f_c(X_{\tau_{a}^-}) ; 
\tau_{a}^- < T^{(1)} \wedge \tau_{b}^+ \right] 
= f_c(a) \, \E_l \left[e^{-(q+\lambda)\tau_{a}^-} ; 
\tau_{a}^- < \tau_{b}^+ \right] 
=f_c(a) \, \frac{W^{(q+\lambda)}(b-l)}{W^{(q+\lambda)}(b-a)}.
\end{align*}
Using \eqref{resolvent_density_0}, the second term can be calculated by 
\begin{align*}
&\E_l \bigg[e^{-qT^{(1)}} \tilde{g}(X_{T^{(1)}}) 1_{\{X_{T^{(1)}}\geq l\}} ; T^{(1)}<\tau_{a}^-\wedge\tau^+_b\bigg] 
= \lambda \, \E_l \bigg[\int_0^{\tau_{a}^-\wedge\tau_b^+} 
e^{-(q+\lambda)s}\, \tilde{g}(X_s) \, 1_{\{X_s \geq l \}} \diff s \bigg] \notag\\
&= \lambda \, \int_0^{b-l} \tilde{g}(b-u) \bigg( 
	\frac{W^{(q+\lambda)}(b-l)}{W^{(q+\lambda)}(b-a)} \, W^{(q+\lambda)}(b-a-u) 
	- W^{(q+\lambda)}(b-l-u) \bigg) \diff u \notag\\ 
	&= \lambda \, \frac{\tilde{g}(l)}{W^{(q)}(b-l)} \int_0^{b-l} W^{(q)}(u) \bigg( 
	\frac{W^{(q+\lambda)}(b-l)}{W^{(q+\lambda)}(b-a)}  \, W^{(q+\lambda)}(b-a-u) 
	- W^{(q+\lambda)}(b-l-u) \bigg) \diff u,
	\end{align*}
	where the last equality follows from \eqref{iden_h_SP_0}. 
	Next, using 
Lemma \ref{aux_SP} and \eqref{LRZ-W} with \eqref{olW0} we further get
	\begin{align} \label{iden_tg_SP_5}
\begin{split}
	&\E_l \bigg[e^{-qT^{(1)}} \tilde{g}(X_{T^{(1)}}) \, 1_{\{X_{T^{(1)}}\geq l\}} ; T^{(1)}<\tau_{a}^-\wedge\tau^+_b \bigg] \\
	&= \frac{\tilde{g}(l)}{W^{(q)}(b-l)} \bigg( 
	\frac{W^{(q+\lambda)}(b-l)}{W^{(q+\lambda)}(b-a)} \left(W^{(q+\lambda)}(b-a) -\mathscr{W}^{(q,\lambda)}_{b-l}(l-a)\right) 
	- W^{(q+\lambda)}(b-l) + W^{(q)}(b-l) \bigg) \\
	&= \tilde{g}(l) - \tilde{g}(l) \, \frac{\mathscr{W}^{(q,\lambda)}_{b-l}(l-a) \; W^{(q+\lambda)}(b-l)}{W^{(q)}(b-l) \; W^{(q+\lambda)}(b-a)}. 
\end{split}
	\end{align}
Hence, putting all the pieces together, we get
\begin{align*}
\tilde{g}(l) = 
f_c(a) \, \frac{W^{(q+\lambda)}(b-l)}{W^{(q+\lambda)}(b-a)}
+\tilde{g}(l) - \tilde{g}(l) \, \frac{\mathscr{W}^{(q,\lambda)}_{b-l}(l-a) \; W^{(q+\lambda)}(b-l)}{W^{(q)}(b-l) \; W^{(q+\lambda)}(b-a)}.
\end{align*}
Therefore, by solving for $\tilde{g}(l)$, we finally obtain
\begin{align}\label{iden_h_SP_3}
\tilde{g}(l)=f_c(a) \, \frac{W^{(q)}(b-l)}{\mathscr{W}^{(q,\lambda)}_{b-l}(l-a)}.
\end{align}

{\it Step 3: Computation of $\tilde{g}(x)$.} 
Substituting \eqref{iden_h_SP_3} in \eqref{iden_h_SP_2} we obtain the result. 

\subsection{Proof of Proposition \ref{limit_cont_SP}} \label{proof_limit_cont_SP}
First, we note that by using Exercise 8.5.(i) in \cite{K} and recalling the definition \eqref{olW} of $\mathscr{W}^{(q,\lambda)}_{b-l}(\cdot)$, we have for $x\in\R$, the limit 
\begin{align} \label{limit_Wb}
\lim_{b\to\infty} \frac{\mathscr{W}^{(q,\lambda)}_{b-l}(x)}{W^{(q)}(b-l)} 
&= \lim_{b\to\infty} \bigg(\frac{W^{(q)}(x+b-l)}{W^{(q)}(b-l)} 
+\lambda \int_0^x W^{(q+\lambda)}(x-u) \, \frac{W^{(q)}(u+b-l)}{{W^{(q)}(b-l)}} \du \bigg) \notag\\
&= e^{\Phi(q)x} + \lambda \int_0^x e^{\Phi(q)u} \, W^{(q+\lambda)}(x-u) \du 
= Z^{(q+\lambda)}(x;\Phi(q)) \,,
\end{align}
where the last equality follows from \eqref{ZqlPq}. 	
Then, by taking $b\to\infty$ in \eqref{iden_h_SP} we get the desired result.
	
\subsection{Proof of Lemma \ref{Prop_cost_per_player}}\label{proof_prop_cost_per_player}

Throughout this proof, we write $g(x) := \E_x\left[e^{-q  T_{l}^-} f_p(X_{ T_{l}^-}) 1_{\{ T_{l}^-< \tau_{a}^- \wedge \tau_b^+\}}\right]$ for $x \in \R$. 
It is straightforward to see by the definitions of $g(x)$ and $\tau_{a}^-$, that $g(x) = 0$ for all $x \leq a$.
	Hence, the remainder of the proof is devoted to the case of $x > a$.
	We prove this result in the following steps. 
	
	{\it Step 1: Computation of $g(x)$ in terms of $g(l)$.}
	On one hand, for $x \geq l$, 
	using the strong Markov property, spatial homogeneity of L\'evy processes and \eqref{upcr}, we obtain 
	\begin{align}\label{iden_g_SP_0}
	g(x) 
	&= \E_{x-l}\left[ e^{-q \tau_0^-} g(X_{\tau_0^-}+l) ; \tau_0^-<\tau_{b-l}^+\right] 
	= g(l) \frac{W^{(q)}(b-x)}{W^{(q)}(b-l)}.
	\end{align}
	On the other hand, for $x \in [a,l)$, using again the strong Markov property, we obtain
	\begin{align}\label{iden_g_SP_1}
	g(x)
	= \E_x\left[ e^{-qT^{(1)}} f_p(X_{T^{(1)}}) ; T^{(1)}< \tau_{a}^-\wedge \tau_l^+ \right]
	+ \E_x\left[ e^{-q \tau_l^+} g(X_{\tau_l^+}) ; \tau_l^+<T^{(1)}\wedge \tau_{a}^- \right].
	\end{align}
	For the first term on the right-hand side of \eqref{iden_g_SP_1}, using the spatial homogeneity of L\'evy processes and \eqref{resolvent_density_0}, 
	\begin{align}\label{iden_g_SP_3}
	&\E_x\Big[ e^{-qT^{(1)}} f_p(X_{T^{(1)}}) ; T^{(1)}<\tau_{a}^-\wedge
	\tau_l^+ \Big]= \lambda \, \E_x \bigg[ \int_0^{\tau_{a}^-\wedge \tau_l^+} 
	e^{-(q+\lambda)s} f_p(X_s) \diff s \bigg] \notag\\
	&= \lambda \, \int_0^{l-a} f_p(l-u) \bigg( \frac{W^{(q+\lambda)}(l-x)}{W^{(q+\lambda)}(l-a)} W^{(q+\lambda)} (l-a-u) - W^{(q+\lambda)} (l-x-u) \bigg) \diff u \\
		&= \lambda  \bigg( \frac{W^{(q+\lambda)}(l-x)}{W^{(q+\lambda)}(l-a)} \Gamma(a;l) -  \Gamma(x;l) \bigg), \nonumber
	\end{align}
	where the last equality follows from \eqref{Gamma}.
	
Now, for the second term on the right-hand side of \eqref{iden_g_SP_1}, we firstly see from the spacial homogeneity of L\'evy processes that
\begin{align*}
\E_x\left[ e^{-q\tau_l^+} g(X_{\tau_l^+}) ; \tau_l^+<T^{(1)}\wedge\tau_{a}^- \right] 
&= \E_{x-l}\left[ e^{-q\tau_0^+} g(X_{\tau_0^+}+l) ; \tau_0^+<T^{(1)}\wedge\tau_{a-l}^- \right] .
\end{align*}
It then follows from \eqref{iden_g_SP_0}, given that $X_{\tau_{0}^+} + l \geq l$, $\p_x$-a.s. (due to the possibility of positive jumps), that
\begin{align*}
{g}(X_{\tau_{0}^+} + l) 
= {g}(l) \, \frac{ W^{(q)}\big(b -X_{\tau_{0}^+} - l \big)}{W^{(q)}(b-l)}, \quad  \p_x-\text{a.s.},
\end{align*}
which then gives that  
\begin{align*}
\E_x\left[ e^{-q\tau_l^+} g(X_{\tau_l^+}) ; \tau_l^+<T^{(1)}\wedge\tau_{a}^- \right] 
&= \frac{g(l)}{W^{(q)}(b-l)} \, \E_{x-l}\bigg[ e^{-q\tau_0^+} W^{(q)}(b-X_{\tau_0^+}-l) ; \tau_0^+<T^{(1)}\wedge\tau_{a-l}^- \bigg] . 
\end{align*}
Combining the above with Lemma \ref{flW-}, we thus obtain
\begin{align*}
\E_x\left[ e^{-q\tau_l^+} g(X_{\tau_l^+}) ; \tau_l^+<T^{(1)}\wedge\tau_{a}^- \right] 
&= \frac{g(l)}{W^{(q)}(b-l)} \bigg( 
\mathscr{W}^{(q, \lambda)}_{b-l}(l-x) - \frac{W^{(q+\lambda)}(l-x)}{W^{(q+\lambda)}(l-a)} \, \mathscr{W}^{(q, \lambda)}_{b-l}(l-a) \bigg).
\end{align*}
	
Therefore, by putting all the pieces together, we obtain for $x\in [a,l)$, that
	\begin{align}\label{iden_g_SP_2}
	g(x) 
	=\, 
	&\lambda  \bigg( \frac{W^{(q+\lambda)}(l-x)}{W^{(q+\lambda)}(l-a)} \Gamma(a;l) -  \Gamma(x;l) \bigg) 
+ g(l) \bigg( 
\frac{\mathscr{W}^{(q, \lambda)}_{b-l}(l-x)}{W^{(q)}(b-l)} - \frac{W^{(q+\lambda)}(l-x)}{W^{(q+\lambda)}(l-a)} \, \frac{\mathscr{W}^{(q, \lambda)}_{b-l}(l-a)}{W^{(q)}(b-l)} \bigg).
	\end{align}
	
	{\it Step 2: Computation of $g(l)$.}
	We note that
	\begin{align*}
	g(l) 
	&= \E_l \left[ e^{-qT^{(1)}} f_p(X_{T^{(1)}}) 1_{\{X_{T^{(1)}}<l\}} ; T^{(1)}<\tau_{a}^-\wedge\tau^+_b\right] + \E_l \left[ e^{-qT^{(1)}} g(X_{T^{(1)}}) 1_{\{X_{T^{(1)}}\geq l\}} ; T^{(1)}<\tau_{a}^-\wedge\tau^+_b\right].
	\end{align*}
	Modifying \eqref{iden_g_SP_3} (with $f_p(\cdot) 1_{\{\cdot < l\}}$ instead of $f_p(\cdot)$), 
and using \eqref{Gamma}, the first expectation becomes
\begin{align*} 
	\E_l \bigg[&e^{-qT^{(1)}} f_p(X_{T^{(1)}}) 1_{\{X_{T^{(1)}}<l\}} ; T^{(1)}<\tau_{a}^-\wedge\tau^+_b\bigg] 
	= \lambda \, \E_l \bigg[\int_0^{\tau_{a}^-\wedge\tau_b^+} 
	e^{-(q+\lambda)s} f_p(X_s) 1_{\{X_s<l\}} \diff s \bigg] \nonumber\\
	&=\lambda \, \frac{W^{(q+\lambda)}(b-l)}{W^{(q+\lambda)}(b-a)}\int_{b-l}^{b-a}f_p(b-u)W^{(q+\lambda)}(b-a-u)\diff u 
	= \lambda \, \frac{W^{(q+\lambda)}(b-l)}{W^{(q+\lambda)}(b-a)} \Gamma(a;l) \nonumber
	\end{align*} 
while for the final expectation, 
we simply replace $\tilde{g}$ with $g$ in \eqref{iden_tg_SP_5} to obtain 
	\begin{align*} 
	\E_l \Big[e^{-qT^{(1)}} g(X_{T^{(1)}}) \, 1_{\{X_{T^{(1)}}\geq l\}} ; T^{(1)}<\tau_{a}^-\wedge\tau^+_b \Big]
	= g(l) - g(l) \, \frac{\mathscr{W}^{(q,\lambda)}_{b-l}(l-a) \; W^{(q+\lambda)}(b-l)}{W^{(q)}(b-l) \; W^{(q+\lambda)}(b-a)}. 
	\end{align*}
	Hence, putting all the pieces together 
	in the original equation, we get
	\begin{align*}
	g(l) = \,&\lambda \, \frac{W^{(q+\lambda)}(b-l)}{W^{(q+\lambda)}(b-a)}
	\Gamma(a;l)
	 \, 
	+g(l) - g(l) \, \frac{\mathscr{W}^{(q,\lambda)}_{b-l}(l-a) \; W^{(q+\lambda)}(b-l)}{W^{(q)}(b-l) \; W^{(q+\lambda)}(b-a)}.
	\end{align*}
	Therefore, by solving for $g(l)$, we finally obtain
	\begin{align}\label{iden_g_SP_4}
	g(l) = \lambda \, \frac{W^{(q)}(b-l)}{\mathscr{W}^{(q, \lambda)}_{b-l}(l-a)} \Gamma(a;l).
	\end{align}
	
	{\it Step 3: Computation of $g(x)$.} 
	By substituting \eqref{iden_g_SP_4} in \eqref{iden_g_SP_2} from the previous two steps,  we obtain the result.

\subsection{Proof of Proposition \ref{limit_math_W}}  \label{proof_limit_math_W}
%
Taking $b\to\infty$ in \eqref{iden_g_SP} together with an application of equation \eqref{limit_Wb} leads to the desired result.

\section{Proofs of some technical results in Section \ref{section_threshold_strategies}} 
\label{technical_proof}

\subsection{Proof of Lemma \ref{lemma_x_c_finite}} 
\label{proof_lemma_x_c_finite}
By Remark \ref{remark_h_c_o}, $\underline{a}=\underline{x}_c$ when $X$ is of unbounded variation.
It thus remains 
to prove the claim only for the bounded variation case, i.e. when $W^{(q+\lambda)}(0) > 0$ (see Remark \ref{remark_scale_function_properties}.(ii)).
By Assumption \ref{Ass} and \eqref{a_underbar}, we have $\underline{a} < \overline{x}_c$ and $f_c(\underline{a}) > 0$. 
Hence, by \eqref{def_h_p_c} we have
$h_c(\underline{a}) = \lambda W^{(q+\lambda)}(0) f_c(\underline{a}) > 0$, 
implying 
in view of \eqref{x_p_and_h_p} that $\underline{a} < \underline{x}_c$ .

\subsection{Proof of Lemma \ref{remark_domain_to_focus}} \label{proof_remark_domain_to_focus}
We prove each part separately.  

{\it Proof of }(i). 
We firstly show that $\max_{a \geq x} v_c(x;a,l) \leq v_c(x;x,l)$, which is straightforward to see since $v_c(x;a,l) = v_c(x;x,l) = f_c(x)$ for all $a \geq x$. 
Then, it remains to show that $\max_{a \geq \overline{x}_c} v_c(x;a,l) \leq v_c(x;\overline{x}_c,l)$. For $a > \overline{x}_c$, 
we have $f_c(a)<0$ due to \eqref{about_x_upper}, hence it is clear that $v_c(x;a,l) \leq v_c(x;\overline{x}_c,l)$.
%
%

{\it Proof of }(i)'. 
We have from Remark \ref{remark_when_a_geq_l} that
\[
\max_{a \in \R} v_c(x;a,l) = \max_{a \leq l} v_c(x;a,l)  \vee \max_{a > l} v_c(x;a,l) =  \max_{a \leq l} v_c(x;a,l)  \vee \max_{a > l} v_c^o(x;a).
\]
Due to the assumption $l \geq \underline{a}$ and given that $a \mapsto v_c^o(x;a)$ is decreasing on $[l, \infty) \subseteq [\underline{a}, \infty)$  (see \eqref{v_o_derivative}--\eqref{a_underbar}), we get
\[
\max_{a \in \R} v_c(x;a,l) =  \max_{a \leq l} v_c(x;a,l)  \vee v_c^o(x;l) = \max_{a \leq l} v_c(x;a,l)  \vee v_c(x;l,l)  = \max_{a \leq l} v_c(x;a,l).
\]
Combining this together with (i), we complete the proof. 
%

{\it Proof of }(ii).
It suffices to show that $\max_{l \notin [a, 
\overline{x}_p]} v_p(x;a,l) \leq v_p(x;a,\overline{x}_p)$.
Indeed, on one hand, for $l < a$, 
we have by Remark \ref{remark_when_a_geq_l}, that $v_p(x;a,l) = 0 \leq v_p(x;a,\overline{x}_p)$. On the other hand, for $l > \overline{x}_p$, we have 
\begin{align*}
v_p(x;a,l) -  v_p(x;a,\overline{x}_p) = \E_x\Big[e^{-q T_l^-} f_p(X_{T_l^-}) 1_{\{T_l^-< \tau_{a}^-, T_l^- < T_{\overline{x}_p}^- \}}\Big] - \E_x\Big[e^{-q T_{\overline{x}_p}^-} f_p(X_{T_{\overline{x}_p}^-}) 1_{\{T_{\overline{x}_p}^-< \tau_{a}^-, T_{l}^- < T_{\overline{x}_p}^- \}}\Big] \leq 0
\end{align*}
where the last inequality holds because, on $\{ T_l^- < T_{\overline{x}_p}^- \}$ we have $X_{T_l^-} > \overline{x}_p$, thus $f_p(X_{T_l^-}) < 0$ while $f_p(X_{T_{\overline{x}_p}^-})\geq 0$ due to \eqref{about_x_upper}.
This completes the proof.

\subsection{Proof of Lemma \ref{lemma_derivatives}}
\label{proof_lemma_derivatives}
We prove the two parts separately. 

{\it Proof of part} (i). 
By \eqref{Z_Phi_der}, we have
\begin{align*}
&\frac{\partial}{\partial a} v_c(x;a,l) \\
&= f_c'(a) \, \frac{Z^{(q+\lambda)}(l-x;\Phi(q))}{Z^{(q+\lambda)}(l-a;\Phi(q))} 
+ f_c(a) \, Z^{(q+\lambda)}(l-x;\Phi(q))\frac{\left(\Phi(q) \, Z^{(q+\lambda)}(l-a;\Phi(q)) + \lambda \, W^{(q+\lambda)}(l-a) \right)}{(Z^{(q+\lambda)}(l-a;\Phi(q)))^2} \notag\\
&= \frac{Z^{(q+\lambda)}(l-x;\Phi(q))}{Z^{(q+\lambda)}(l-a;\Phi(q))} \bigg( f_c'(a) + f_c(a) \, \frac{\Phi(q) \, Z^{(q+\lambda)}(l-a;\Phi(q)) + \lambda \, W^{(q+\lambda)}(l-a)}{Z^{(q+\lambda)}(l-a;\Phi(q))} \bigg) \notag\\
&= \frac{Z^{(q+\lambda)}(l-x;\Phi(q))}{Z^{(q+\lambda)}(l-a;\Phi(q))} \Big( f_c'(a) + \Big( \Phi(q) \, Z^{(q+\lambda)}(l-a;\Phi(q)) + \lambda \, W^{(q+\lambda)}(l-a) \Big) v_c(l;a,l) \Big) \,,
\end{align*}
where the last equality holds by \eqref{vCl}.

{\it Proof of part }(ii). Using \eqref{Gamma} we have
$\frac \partial {\partial l} \Gamma(x;l) 
= f_p(l) W^{(q+\lambda)}(l-x)$ and by  \eqref{Z_Phi_der}
    \begin{align*}
\frac \partial {\partial l}\dfrac{Z^{(q+\lambda)}(l-x;\Phi(q))}{Z^{(q+\lambda)}(l-a;\Phi(q))} 
&= \dfrac{\Phi(q)Z^{(q+\lambda)}(l-x;\Phi(q)) + \lambda W^{(q+\lambda)}(l-x)}{Z^{(q+\lambda)}(l-a;\Phi(q))} \\
&\quad- \dfrac{ Z^{(q+\lambda)}(l-x;\Phi(q)) (\Phi(q)Z^{(q+\lambda)}(l-a;\Phi(q)) + \lambda W^{(q+\lambda)}(l-a))}{(Z^{(q+\lambda)}(l-a;\Phi(q)))^2}  \\
&= \dfrac{\lambda }{Z^{(q+\lambda)}(l-a;\Phi(q))} \bigg( W^{(q+\lambda)}(l-x) - \dfrac{ Z^{(q+\lambda)}(l-x;\Phi(q)) W^{(q+\lambda)}(l-a)}{Z^{(q+\lambda)}(l-a;\Phi(q))} \bigg).
  \end{align*}
  Hence, differentiating \eqref{vf_1P}, we get
  \begin{align*}
\lambda^{-1}\frac \partial {\partial l} v_p(x;a,l) 
 &= \dfrac{\lambda }{Z^{(q+\lambda)}(l-a;\Phi(q))} \bigg( W^{(q+\lambda)}(l-x) - \dfrac{ Z^{(q+\lambda)}(l-x;\Phi(q)) W^{(q+\lambda)}(l-a)}{Z^{(q+\lambda)}(l-a;\Phi(q))} \bigg)  \Gamma(a;  l) \\
 &\quad+ \dfrac{Z^{(q+\lambda)}(l-x;\Phi(q))}{Z^{(q+\lambda)}(l-a;\Phi(q))} f_p(l) W^{(q+\lambda)}(l-a) - f_p(l) W^{(q+\lambda)}(l-x) \\
   &= \bigg( \dfrac{Z^{(q+\lambda)}(l-x;\Phi(q))}{Z^{(q+\lambda)}(l-a;\Phi(q))}W^{(q+\lambda)}(l-a) - W^{(q+\lambda)}(l-x) \bigg) \big( f_p(l) - v_p(l;a,l)  \big), 
\end{align*}
where the last equality holds by \eqref{vPl}.

\subsection{Proof of Proposition \ref{smooth}} \label{proof_smooth}


We prove the two sets of properties separately. 

{\it Proof of part} (I). 
The continuity is clear by \eqref{vf_1C}.  Fix $x \in (a^*, \infty)\backslash\{l^*\}$.
By \eqref{vf_1C}, \eqref{Z_Phi_der} and \eqref{vCl},
\begin{align}  
\begin{split}
v_c'(x;a^*,l^*) 
&= - f_c(a^*) \, \dfrac{\Phi(q) Z^{(q+\lambda)}(l^*-x; \Phi(q) ) + \lambda  W^{(q+\lambda)}(l^*-x)}{Z^{(q+\lambda)}(l^*-a^*;\Phi(q))} \\
 &= - \big(\Phi(q)Z^{(q+\lambda)}(l^*-x;\Phi(q))+\lambda W^{(q+\lambda)}(l^*-x) \big)v_c(l^*;a^*,l^*). 
 \end{split} \label{der_v_c_a_l}
\end{align}
Differentiating this further and using \eqref{der_v_c_a_l}, we get  
\begin{align}  \label{2der_v_c_a_l}
v_c''(x+;a^*,l^*) 
 &= \big(\Phi(q) (\Phi(q) Z^{(q+\lambda)}(l^*-x; \Phi(q) ) + \lambda  W^{(q+\lambda)}(l^*-x)) + \lambda W^{(q+\lambda)\prime}((l^*-x)-) \big)v_c(l^*;a^*,l^*) \nonumber\\
 &= - \Phi(q) v_c'(x;a^*,l^*) + \lambda W^{(q+\lambda)\prime}((l^*-x)-) v_c(l^*;a^*,l^*). 
\end{align}
By \eqref{der_v_c_a_l} and \eqref{2der_v_c_a_l} together with Remark  \ref{remark_scale_function_properties}, we prove part  (i).


To prove part (ii), we use the definition \eqref{Jal} and the fact that $I(a^*; l^*)=0$, to conclude from \eqref{der_v_c_a_l} that
\begin{align*}
v_c'(a^*+;a^*,l^*) 
&=- \big(\Phi(q)Z^{(q+\lambda)}(l^*-a^*;\Phi(q))+\lambda W^{(q+\lambda)}(l^*-a^*) \big)v_c(l^*;a^*,l^*) \notag\\ 
&=f_c'(a^*)-I(a^*;l^*)=f_c'(a^*).
\end{align*}
This  coincides with $v_c'(a^*-;a^*,l^*)=f_c'(a^*)
$ (see \eqref{vf_1C}),
hence $v_c(\cdot,a^*,l^*)$ is continuously differentiable at $a^*$.

Part (iii) holds true by combining \eqref{der_v_c_a_l} with Remark \ref{remark_scale_function_properties}.

%

Using once again \eqref{der_v_c_a_l} with the positivity of the scale function, yielding $v_c'(x;a^*,l^*) \leq 0$ for $x \in (a^*, \infty)\backslash\{l^*\}$, together with the continuity of  $v_c(\cdot;a^*,l^*)$ at $l^*$ from part (i), we conclude the monotonicity in part (iv).
Then, \eqref{2der_v_c_a_l}, the monotonicity and positivity of $v_c(x;a^*,l^*)$ imply that $v_c''(x;a^*,l^*) > 0$ for all $x$ except for $l^*$ and the discontinuity points of $W^{(q+\lambda)'}(\cdot)$.
However, in the unbounded variation case, $v_c'(\cdot;a^*,l^*)$ is continuous at $l^*$ from part (iii), 
%
%
while in the bounded variation case, we have from \eqref{der_v_c_a_l} that
\begin{align*}
v_c'(l^*-;a^*,l^*) = v_c'(l^*+;a^*,l^*) - \lambda \, W^{(q+\lambda)}(0) \, v_c(l^*;a^*,l^*) < v_c'(l^*+;a^*,l^*).
\end{align*}
This shows (in view of $v_c$ begin decreasing) the desired convexity on $(a^*, \infty)$ in part (iv).

{\it Proof of part} (II). 
The continuity is clear by \eqref{v_p_opt}. 
Fix $x \in (a^*, \infty)\backslash\{l^*\}$.
By \eqref{v_p_opt}, \eqref{Z_Phi_der} and \eqref{cond_lP}, 
\begin{align}\label{der_v_p_a_l}
v_p'(x;a^*,l^*)
&=- \big(\Phi(q)Z^{(q+\lambda)}(l^*-x;\Phi(q))+\lambda W^{(q+\lambda)}(l^*-x)\big) v_p(l^*;a^*,l^*)+\lambda f_p(l^*)W^{(q+\lambda)}(l^*-x)\notag\\
&\quad -\lambda\int_0^{l^*-x}f_p'(u+x)W^{(q+\lambda)}(u) \du \notag\\
&=- \Phi(q) Z^{(q+\lambda)}(l^*-x;\Phi(q)) v_p(l^*;a^*,l^*) 
- \lambda \int_0^{l^*-x} f_p'(u+x) W^{(q+\lambda)}(u) \du,
\end{align}
while further differentiation yields 
\begin{align}
\label{der_v_p_a_l_second}
v_p''(x;a^*,l^*)&=\Phi(q) \big(\Phi(q)Z^{(q+\lambda)}(l^*-x;\Phi(q)) +\lambda W^{(q+\lambda)}(l^*-x) \big)v_p(l^*;a^*,l^*) 
\notag\\
&\quad +\lambda f_p'(l^*)W^{(q+\lambda)}(l^*-x)
-  \lambda\int_0^{l^*-x}f_p''(u+x)W^{(q+\lambda)}(u) \diff u.
\end{align}
Thus, the $C^2$ property in part (i) follows by using also Assumption \ref{Ass}.(ii), while parts (ii)--(iii) follow from the continuity and smoothness of the scale function in Remark \ref{remark_scale_function_properties} applied to \eqref{der_v_p_a_l}--\eqref{der_v_p_a_l_second}.

\section{Proof of Lemma \ref{verification_split} (Verification for player $P$)} \label{verification_split_proof}

Throughout this proof, we define $w_p(x) := v_p(x; a^*, l^*)$, for all $x \in \R$,  
$e_\lambda$ to be an exponential random variable independent of $X$ and $T^{(0)} := 0$.
This proof extends the results obtained by \cite{DW2002} in an infinite horizon optimal stopping problem for a continuous model, to our case of a random time horizon and spectrally positive L\'evy models.



The proof for $x < a^*$ is straightforward, since 
$V_p(\tau_{a^*}^-,\sigma; x) = 0$ for all $\sigma \in \mathcal{T}_p$ and therefore $\sup_{\sigma \in \mathcal{T}_p}V_p(\tau_{a^*}^-,\sigma; x) = 0 = w_p(x)$ by condition (v). 

In the rest of the proof, we assume $x \geq a^*$ and 
fix $\varepsilon > 0$ and $m > 0$. 
In view of the smoothness of $w_p$ on $[a^*+\varepsilon, \infty)$ from Proposition \ref{smooth}.(II), it follows by It\^o's formula  for all $n \in \mathbb{N}$ and $t \geq 0$, that 
\begin{equation}\label{Ver_lemma_2}
\begin{split}
e^{-(q+\lambda)(t \wedge \tau_{a^* + \varepsilon}^- \wedge \tau^+_m)}w_p\big(X_{t \wedge \tau_{a^* + \varepsilon}^- \wedge \tau^+_m}\big)  -w_p(x)
&=   \int_{0}^{t \wedge \tau_{a^* + \varepsilon}^-  \wedge \tau^+_m} \hspace{-2mm}e^{-(q+\lambda)s}   (\mathcal{L}-(q+\lambda))w_p(X_{s})   \mathrm{d}s
+ M_{t \wedge \tau_{a^* + \varepsilon}^-  \wedge \tau^+_m},
\end{split}
\end{equation}
where $(M_t)_{t \geq 0}$ is a zero-mean local martingale with respect to the filtration $\mathbb{F}$. 

Next, we aim at deriving a probabilistic expression of $w_p$,  which will involve the function $\overline{w}_p$ defined by
\begin{align}
\overline{w}_p(x) := \max\{f_p(x),w_p(x)\} = f_p(x) 1_{\{x \leq l^*\}} + w_p(x) 1_{\{x > l^*\}}, \quad x \in [a^*, \infty), 
\label{barw_P}
\end{align}
where the latter equality holds true due to conditions (iii) and (iv).

\begin{lemma} \label{lem:w_P_form}
For $x \geq a^*$, we have
\begin{align}
w_p(x) = \E_x\bigg[\lambda   \int_{0}^{\tau^-_{a^*}} e^{-(q+\lambda)s} \, \overline{w}_p(X_{s})   \mathrm{d}s\bigg] = \E_x\left[e^{-q \, e_\lambda} \, \overline{w}_p(X_{e_\lambda}) \, 1_{\{ e_\lambda < \tau_{a^*}^-\}}\right]. \label{w_P_form}
\end{align}
\end{lemma}
\begin{proof}
For $x > a^*$ (where $\mathcal{L} w_p(x)$ is well-defined by Proposition \ref{smooth}), we have
\begin{align*} 
(\mathcal{L}-q)w_p(x) + \lambda \, \max\{f_p(x)-w_p(x),0\}
= (\mathcal{L}-(q+\lambda))w_p(x) + \lambda \overline{w}_p(x) 
= 0.
\end{align*}
The last equality holds for $x \geq l^*$ by conditions (i) and (iii), and for $a^* < x < l^*$ by conditions (ii) and (iv). 
Substituting this back in \eqref{Ver_lemma_2}, we obtain
\begin{align*}
e^{-(q+\lambda)(t \wedge \tau_{a^* + \varepsilon}^-  \wedge \tau^+_m)}w_p \big(X_{t \wedge \tau_{a^* + \varepsilon}^-  \wedge \tau^+_m} \big) 
= w_p(x) 
-\lambda   \int_{0}^{t \wedge \tau_{a^* + \varepsilon}^-  \wedge \tau^+_m}e^{-(q+\lambda)s}  \,\overline{w}_p(X_{s})   \mathrm{d}s
+ M_{t \wedge \tau_{a^* + \varepsilon}^-  \wedge \tau^+_m}.
\end{align*}
Therefore, by rearranging the terms and taking expectations, the optional sampling theorem gives
\begin{align*}
w_p(x) 
= \E_x\left[e^{-(q+\lambda)(t \wedge \tau_{a^* + \varepsilon}^-  \wedge \tau^+_m)}w_p \big(X_{t \wedge \tau_{a^* + \varepsilon}^- \wedge \tau^+_m}\big)\right]+\E_x\bigg[\lambda   \int_{0}^{t \wedge \tau_{a^* + \varepsilon}^-  \wedge \tau^+_m}e^{-(q+\lambda)s}  \,\overline{w}_p(X_{s})   \mathrm{d}s \bigg].
\end{align*}
Since $\overline{w}_p$ is bounded, the dominated convergence theorem then gives, upon taking $m \to \infty$, that 
\begin{align}\label{Ver_lemma_4}
w_p(x) 
= \E_x\left[e^{-(q+\lambda)(t \wedge \tau_{a^* + \varepsilon}^-)}w_p \big(X_{t \wedge \tau_{a^* + \varepsilon}^- } \big)\right]+\E_x \bigg[\lambda   \int_{0}^{t \wedge \tau_{a^* + \varepsilon}^-}e^{-(q+\lambda)s}  \,\overline{w}_p(X_{s})   \mathrm{d}s\bigg].
\end{align}
%
Moreover, the boundedness and non-negativity of $w_p$, implies again by the dominated convergence theorem that
\begin{align*}
\lim_{t \uparrow\infty}\E_x\left[e^{-(q+\lambda)(t \wedge \tau_{a^* + \varepsilon}^-)}w_p \big(X_{t \wedge \tau_{a^* + \varepsilon}^- } \big) \right]
&= \E_x\left[e^{-(q+\lambda)\tau_{a^* + \varepsilon}^-} \, w_p \big(X_{\tau_{a^* + \varepsilon}^- } \big) 1_{\{\tau_{a^* + \varepsilon}^- < \infty\}} \right] \\
&\leq \E_x\left[e^{-(q+\lambda)\tau_{a^* + \varepsilon}^-} \, 1_{\{\tau_{a^* + \varepsilon}^- < \infty\}} \right] \max_{a^* \leq y \leq a^* + \varepsilon} w_p(y)
\end{align*}
where the latter inequality holds because $X_{\tau_{a^* + \varepsilon}^- } \leq a^*+\varepsilon$ a.s.  and by the condition (v).
 Since $w_p$ is continuous on $[a^*,\infty)$ by Proposition \ref{smooth}.(II), and $w_p(a^*) = 0$ by condition (v), we have $\max_{a^* \leq y \leq a + \varepsilon} w_p(y) 
\rightarrow 0$ as $\varepsilon \downarrow 0$, hence \begin{align*}
\lim_{\varepsilon \downarrow 0} \lim_{t \uparrow\infty}\E_x\left[e^{-(q+\lambda)(t \wedge \tau_{a + \varepsilon}^- )} \,w_p \big(X_{t \wedge \tau_{a^* + \varepsilon}^- } \big)\right] = 0.
\end{align*}

As $X$ is a spectrally positive \lev process, $(\tau_{-b}^-)_{b \geq 0}$ is a $\p$-subordinator with potential killing (see, e.g., the proof of Lemma VII.23 of  \cite{B}), hence $\tau_{a^*}^-$ at any time $a^*$ is continuous $\p_x$-a.s.; this implies $\tau_{a^*+\varepsilon}^- 
\rightarrow \tau_{a^*}^-$ as $\varepsilon \downarrow 0$ on $\{ \tau_{a^*}^- < \infty\}$. Therefore, by the monotone convergence theorem, we obtain
\[
\E_x\bigg[\int_{0}^{t \wedge \tau_{a^* + \varepsilon}^- } e^{-(q+\lambda)s}\, \overline{w}_p(X_{s})   \mathrm{d}s\bigg] \xrightarrow{t \uparrow \infty, \varepsilon \downarrow 0}  \E_x\bigg[ \int_{0}^{\tau_{a^*}^-}e^{-(q+\lambda)s}\, \overline{w}_p(X_{s})   \mathrm{d}s\bigg]. 
\]
Finally, taking the limit as $t \uparrow\infty$ and $\varepsilon \downarrow 0$ in \eqref{Ver_lemma_4}, we complete the proof of the first equality in \eqref{w_P_form}.
The second equality in \eqref{w_P_form} then holds true due to  the definition of $e_\lambda$.
\end{proof}

%

Then, by Lemma \ref{lem:w_P_form} 
and the definition \eqref{barw_P} of $\overline{w}_p(\cdot)$, we have the inequality 
\begin{align} \label{w_o_bigger}
\overline{w}_p(x) \geq w_p(x) = \E_x\left[e^{-q \, e_\lambda} \, \overline{w}_p(X_{ e_\lambda}) \, 1_{\{ e_\lambda < \tau_{a^*}^-\}}\right].
\end{align}
To proceed further, 
{recall that player $P$'s filtration satisfies $\mathcal{G}_n \subset \tilde{\mathcal{G}}_n$ for all $n \geq 0$ (cf.\ Remark \ref{remark_filtration}).}
Now, fix $n \geq 0$ and observe that the event $\{T^{(n)} < \tau_{a^*}^- \}$ is $\tilde{\mathcal{G}}_n$-measurable, hence by the strong Markov property of $(X,N)$, we get
\begin{align*}
&\E_x\left[e^{-qT^{(n+1)}} \overline{w}_p(X_{T^{(n+1)}}) 1_{\{T^{(n+1)}<\tau_{a^*}^-\}} \,\Big|\, \tilde{\mathcal{G}}_n 
\right] \\
&= e^{-qT^{(n)}} 1_{\{ T^{(n)} < \tau_{a^*}^- \}} \E_x \left[e^{-q(T^{(n+1)}-T^{(n)})} \overline{w}_p(X_{T^{(n+1)}} ) 1_{\{T^{(n+1)} < \tau_{a^*}^- \}} \Big| \tilde{\mathcal{G}}_n  \right] \notag\\
&= e^{-qT^{(n)}} 1_{\{ T^{(n)} < \tau_{a^*}^- \}} \E_x \left[e^{-q(T^{(n+1)}-T^{(n)})} \overline{w}_p(X_{T^{(n+1)}} ) 1_{\{T^{(n+1)} < \tau_{a^*}^- \}} \Big| X_{T^{(n)}}, T^{(n)} \right] \notag\\
&= e^{-qT^{(n)}} 1_{\{ T^{(n)} < \tau_{a^*}^- \}} \E_{X_{T^{(n)}}}\left[e^{-q \, e_\lambda} \, \overline{w}_p(X_{ e_\lambda}) \, 1_{\{ e_\lambda < \tau_{a^*}^-\}}\right] 
  \leq e^{-qT^{(n)}} 1_{\{T^{(n)} < \tau_{a^*}^- \}}\overline{w}_p(X_{T^{(n)}}),
\end{align*}
where the last inequality holds by  \eqref{w_o_bigger}.
This shows that 
$$ \left\{e^{-qT^{(n)}}\overline{w}_p(X_{T^{(n)}})1_{\{ T^{(n)}<\tau_{a^*}^-\}}\right\}_{n\in\N} 
\quad \text{is a $
\tilde{\mathbb{G}}
$-supermartingale}.  
$$ 
{Using the definition \eqref{barw_P} of $\overline{w}_p(\cdot)$, 
we can conclude 
by Fatou's lemma followed by the optional sampling theorem (noting that $\sigma$ is a $\tilde{\mathbb{G}}$-stopping time as $\mathcal{G}_n \subset \tilde{\mathcal{G}}_n$) that}
\begin{multline*}
V_p(\tau_{a^*}^-,\sigma;x) 
\equiv \E_x\left[e^{-q\sigma}f_p(X_{\sigma})1_{\{ \sigma < \tau_{a^*}^-\}}\right] 
\leq \E_x\left[e^{-q\sigma} \overline{w}_p(X_{\sigma})1_{\{ \sigma < \tau_{a^*}^-\}}\right] \\
\leq \liminf_{N\uparrow \infty}\E_x\left[e^{-q(\sigma\wedge T^{(N)})} \overline{w}_p(X_{\sigma\wedge T^{(N)}})1_{\{ \sigma\wedge T^{(N)} < \tau_{a^*}^-\}}\right] 
\leq \E_x\left[e^{-q(\sigma\wedge T^{(1)})} \overline{w}_p(X_{\sigma\wedge T^{(1)}})1_{\{ \sigma\wedge T^{(1)} < \tau_{a^*}^-\}}\right] 
= w_p(x),
\end{multline*}
where the last equality holds because $\sigma \geq T^{(1)}$ a.s. together with \eqref{w_P_form}. The arbitrariness of $\sigma \in \mathcal{T}_p$ then completes the proof.

\end{appendix}

\end{document}